\documentclass[leqno,11pt]{article}
\usepackage{amssymb,amsmath,mathabx}

\usepackage{pgf,ifpdf,graphics,xcolor}

\usepackage{amsfonts,amssymb,euscript,epsfig,color}
\usepackage[all]{xy}

\newtheorem{theo}{Theorem}[section]
\newtheorem{defi}[theo]{Definition}

\newtheorem{lem}[theo]{Lemma}
\newtheorem{prop}[theo]{Proposition}
\newtheorem{rem}[theo]{Remark}
\newtheorem{coro}[theo]{Corollary}
\newtheorem{exam}[theo]{Example}

\newenvironment{proof}{{\bf Proof.}}

\newcommand{\ggot}{\ensuremath{\mathfrak{g}}}
\newcommand{\hgot}{\ensuremath{\mathfrak{h}}}
\newcommand{\kgot}{\ensuremath{\mathfrak{k}}}

\newcommand{\qgot}{\ensuremath{\mathfrak{q}}}

\newcommand{\tgot}{\ensuremath{\mathfrak{t}}}

\newcommand{\zgot}{\ensuremath{\mathfrak{z}}}

\newcommand{\Acal}{\ensuremath{\mathcal{A}}}
\newcommand{\Bcal}{\ensuremath{\mathcal{B}}}

\newcommand{\Dcal}{\ensuremath{\mathcal{D}}}
\newcommand{\Ecal}{\ensuremath{\mathcal{E}}}
\newcommand{\Fcal}{\ensuremath{\mathcal{F}}}
\newcommand{\Hcal}{\ensuremath{\mathcal{H}}}

\newcommand{\Lcal}{\ensuremath{\mathcal{L}}}

\newcommand{\Ncal}{\ensuremath{\mathcal{N}}}
\newcommand{\Ocal}{\ensuremath{\mathcal{O}}}
\newcommand{\Pcal}{\ensuremath{\mathcal{P}}}
\newcommand{\Qcal}{\ensuremath{\mathcal{Q}}}
\newcommand{\Scal}{\ensuremath{\mathcal{S}}}
\newcommand{\Ucal}{\ensuremath{\mathcal{U}}}

\newcommand{\Vcal}{\ensuremath{\mathcal{V}}}
\newcommand{\Wcal}{\ensuremath{\mathcal{W}}}
\newcommand{\Xcal}{\ensuremath{\mathcal{X}}}
\newcommand{\Ycal}{\ensuremath{\mathcal{Y}}}
\newcommand{\Zcal}{\ensuremath{\mathcal{Z}}}

\newcommand{\Pbb}{\ensuremath{\mathbb{P}}}
\newcommand{\Z}{\ensuremath{\mathbb{Z}}}
\newcommand{\C}{\ensuremath{\mathbb{C}}}

\newcommand{\R}{\ensuremath{\mathbb{R}}}
\newcommand{\N}{\ensuremath{\mathbb{N}}}

\newcommand{\Lbb}{\ensuremath{\mathbb{L}}}

\newcommand{\mm}{\ensuremath{\hbox{\rm m}}}

\newcommand{\tr}{\ensuremath{\hbox{\bf Tr}}}
\newcommand{\ntr}{\ensuremath{\hbox{\bf nTr}}}

\newcommand{\croc}{\ensuremath{\hookrightarrow}}

\newcommand{\indice}{\ensuremath{\hbox{\rm Index}}}

\newcommand{\Char}{\ensuremath{\hbox{\rm Char}}}
\newcommand{\End}{\ensuremath{\hbox{\rm End}}}

\newcommand{\Sym}{\ensuremath{{\rm Sym}}}

\newcommand{\QS}{\ensuremath{\mathrm{Q}^{\mathrm{spin}}}}

\newcommand{\spinc}{\ensuremath{{\rm Spin}^{c}}}

\newcommand{\Qrm}{\ensuremath{\mathrm{Q}}}

\def \K  {{\rm\bf K}}
\def \wK {\widehat{K}}
\def \T {{\rm T}}

\def \what {\widehat}

\def \ad {{\rm ad}}

\def \clif {{\bf c}}

\begin{document}

\title{Equivariant Dirac operators and differentiable geometric invariant theory}

\author{Paul-Emile PARADAN\footnote{Institut de Math\'ematiques et de Mod\'elisation de Montpellier, CNRS UMR 5149,
Universit\'e Montpellier 2, \texttt{paradan@math.univ-montp2.fr}}
\hspace{1mm}  and
Mich\`ele VERGNE\footnote{Institut de Math\'ematiques de Jussieu, CNRS UMR 7586,
Universit\'e Paris 7, \texttt{vergne@math.jussieu.fr}}}


\maketitle

{\small
\tableofcontents}



\section{Introduction}

When  $D$ is an elliptic operator on a manifold $M$ preserved by a    symmetry group $K$,
one can understand the aim of ``geometric invariant theory" as the realization of the space of $K$-invariant solutions of $D$ as
the space of  solutions of an elliptic operator on a ``geometric quotient" $M_0$  of $M$.

The by now classical case is concerned with an action of a compact group $K$ on a compact complex manifold $M$: we may consider the
Dolbeault operator $D$ acting on sections of a holomorphic line bundle $L$.
When $L$ is {\bf ample}, Guillemin-Sternberg \cite{Guillemin-Sternberg82}  proved that the $K$-invariant solutions
of $D$ can be realized on Mumford's GIT quotient $M_0:=\Phi_L^{-1}(0)/K$: here $\Phi_L$ is the moment
map associated to the $K$-action on the line bundle $L$. This result was extended to the other cohomology groups of an {\bf ample} line bundle
by Teleman in \cite{Teleman-Annals00} (see also \cite{Sjamaar95}).

In our article, we show that the same construction can be generalized to the differentiable case if properly reformulated.
We consider  a compact connected Lie group $K$  with Lie algebra $\kgot$ acting on a compact, oriented and
even dimensional manifold $M$. In this introduction we assume for simplicity that $M$ carries a $K$-invariant
spin structure: the corresponding Dirac operator plays the role of the Dolbeaut operator.

For {\bf any} line bundle $L$, we consider the Dirac operator $D:=D_L$  twisted by $L$. It acts on sections of
the Clifford bundle $\Scal=\Scal_{\mathrm{spin}}\otimes L$ on $M$, where $\Scal_{\mathrm{spin}}$ is
the spinor bundle of $M$.
We are concerned with the equivariant index of $D$, that we denote by $\Qcal_K(M,\Scal)$
and we also say  that $\Qcal_K(M,\Scal)$ is the space  of virtual solutions of $D$. It belongs to the
Grothendieck group of representations of $K$. More generally, we can consider any
irreducible equivariant Clifford module $\Scal$ over $M$, when $M$ admits a $\spinc$ structure.

 An important example is when $M$ is a compact complex manifold, $K$ a compact group of
 holomorphic transformations of $M$, $L$ a holomorphic $K$-equivariant line bundle on $M$,
 {\bf  not necessarily ample}, and $D$ the Dolbeault operator acting on sections on the Clifford bundle
 $\Scal$ of $L$-valued differential forms of type $(0,q)$.
Then $\Qcal_K(M,\Scal)=\sum_{q=0}^{\dim_\C M} (-1)^q H^{0,q}(M,L)$.

Our aim is to show that the virtual space of $K$-invariant solutions of the twisted Dirac operator $D$
can be identified with the space of  virtual solutions of a twisted Dirac operator on a ``geometric quotient"
$M_{0}$  of $M$, constructed with the help of a moment map. 
To formulate a clean result in the context of Dirac operators is not obvious.
Let us first state the vanishing  theorem  (surprisingly difficult to prove) which will allow us to do so.

We use Duflo's notion of admissible coadjoint orbits (see Section \ref{sec:coadjoint-orbit})
to produce   unitary irreducible representations of $K$. There is a map  $\QS_K$
associating to an admissible coadjoint orbit $\Pcal$  a virtual representation  $\QS_K(\Pcal)$ of $K$.
By this correspondence, regular  admissible  coadjoint orbits parameterize the set $\wK$ of
classes of unitary irreducible representations of $K$. The coadjoint orbit of $\rho$ is regular admissible
and parameterizes the trivial representation of $K$.
 However, if $r$ is the rank of $[\kgot,\kgot]$, there are $2^r$ admissible orbits $\Pcal$
such that
$\QS_K(\Pcal)$ is the trivial representation of $K$.
We will say that such an orbit  $\Pcal$ is an ancestor of the trivial representation.

For $\hgot$ a subalgebra of $\kgot$, we denote by $(\hgot)$ the conjugacy class of $\hgot$.
If $\xi\in \kgot^*$, we denote by $\kgot_\xi$ its infinitesimal stabilizer.
The set $\Hcal_\kgot$ of conjugacy classes of the algebras $\kgot_\xi$,  $\xi$ running in $\in\mathfrak k^*$,
is a finite set. Indeed the complexified Lie algebras of $\kgot_\xi$ varies over the Levi subalgebras of $\kgot_\C$.
For $(\hgot)\in \Hcal_\kgot$, we say that a coadjoint orbit $K\xi$ is of type $(\hgot)$  if
 $\kgot_\xi$ belongs to the conjugacy class $(\hgot)$. The semi-simple part of $\kgot_\xi$ is $[\kgot_\xi,\kgot_\xi]$.

Let $(\kgot_M)$ be  the generic infinitesimal stabilizer of the $K$-action on $M$.
We prove the following theorem.
\begin{theo}\label{theo-intro-1}
If   $([\kgot_M,\kgot_M])$ is not equal to some $([\hgot,\hgot])$, for $\hgot\in \Hcal_\kgot$,
then for any $K$-equivariant line bundle $\Lcal$, $\Qcal_K(M,\Scal)=0$.
\end{theo}

We may thus assume that there exists   $(\hgot)\in \Hcal_\kgot$, such that
$([\kgot_M,\kgot_M])=([\hgot,\hgot])$ : this class is unique and is denoted $(\hgot_M)$.
This condition on the $K$-action is always satisfied in the Hamiltonian setting \cite{L-M-T-W}, but
not always in the spin setting (see the case of spheres in Example \ref{example-sphere}).

%

Consider our line bundle $L$. The choice of an Hermitian connection $\nabla$ determines a moment map
$$
\Phi_{L} : M\to \kgot^*
$$
by the relation $\Lcal(X)-\nabla_{X_M}=i\langle \Phi_{L},X\rangle$, for  all   $X\in\kgot$.

We now describe the geometric quotient $M_0$.
Let us first state the result, when the infinitesimal stabilizer $(\kgot_M)$ is abelian.
The corresponding $(\hgot_M)$ is the conjugacy class of Cartan subalgebras, and we consider
$$
M_0=\Phi_L^{-1}(K\rho)/K
$$
where $K\rho$ is the regular admissible orbit that parameterizes the trivial representation.
In the general case, we define
$\Ocal_M=\bigcup \Pcal$ to be the union of the ancestors of the trivial representation {\bf which are of type $(\hgot_M)$}. Thus $\Ocal_M$ is a  union of a
finite number of admissible coadjoint orbits, a number that might be greater than $1$
(see Example \ref{sec:induced-example} in the last section).
We then consider
$$M_0=\Phi_L^{-1}(\Ocal_M)/K.$$

Then we  define, by a desingularization procedure,
a virtual vector space $\QS(M_0)$ which coincides when $M_0$ is smooth with
the space of virtual solutions of a twisted Dirac operator on $M_0$.
We prove the following theorem.
\begin{theo}\label{theo-intro-2}
$$[\Qcal_K(M,\Scal)]^K=\QS(M_0).$$
\end{theo}

This is an equality of dimensions. However, this equality holds also in the Grothendieck
group of irreducible representations of $G$, if $G$ is a  compact group of symmetry
commuting with the action of $K$.

Thus our space $M_0$ plays the role of the geometric quotient in this purely
differentiable setting. The space $M_0$ may vary dramatically with the choice of the connection $\nabla$ (see Example  \ref{exa:P1} in the last section),
but not its quantized space $\QS(M_0)$.

Let us recall that we did not make any assumption on the line bundle $L$.
So this equality is true for any  line bundle $L$, and any choice of $K$-invariant connection $\nabla$ on $L$.
In particular, the curvature of $\nabla$ might be always degenerate, whatever choice
of connection.
In the last section, Section \ref{sec:examples}, we raise a question on existence of ``best connections".

\medskip

Let us recall  the previous results on this subject. After their work on the K\"ahler case, \cite{Guillemin-Sternberg82}, Guillemin-Sternberg
formulated the  conjecture ``Quantization commutes with reduction" denoted by $[Q,R]=0$
when $M$ was  symplectic and $L$ a Kostant line bundle  on $M$. The curvature of $L$ is thus $i$ times the symplectic form and  the reduced space $M_0$ is again a symplectic manifold.
 This conjecture was
proved  in full generality by Meinrenken-Sjamaar \cite{Meinrenken-Sjamaar}, following partial results notably by
\cite{Guillemin95, Vergne96, Jeffrey-Kirwan1997, Meinrenken98}. Later,  other proofs by analytic or topological methods
were given by \cite{Tian-Zhang98,pep-RR}.

After the remarkable results of Meinrenken-Sjamaar \cite{Meinrenken-Sjamaar},
 it was tempting to find in what way we can extend their results for the Hamiltonian case to the general $\spinc$ situation.
 In this general context, our manifold $M$ is not necessarily complex, nor even almost-complex.
 So the only elliptic operators which make sense in this case are
 twisted Dirac operators. We restrict ourselves to line bundles, the case of vector bundles
 being obtained by pushforward of index of line bundles.

When $M$ is a $\spinc$ manifold, with an action of $S^1$,
a partial answer to
 the question of quantization commutes with reduction in the spin setting  has been obtained by
\cite{Grossberg-Karshon94,Grossberg-Karshon98,CdS-K-T}.
The case of toric manifolds and non ample line bundles have been treated in
\cite{Karshon-Tolman93}.  These interesting examples (we give an example
  due to
 Karshon-Tolman in the last section, Example \ref{sec:hirzebruch-surface})
 motivated us to search for a
general result.
However,  to formulate what should be the result in the general non abelian case  was not immediately
clear to us, although  a posteriori very natural.
We really had to use (in the case where the generic stabilizer is non abelian)
non regular admissible orbits.

Let us also say that, due to the inevitable $\rho$-shift in the spin context,
our theorem does not imply  immediately the $[Q,R]=0$ theorem of the Hamiltonian case.
 Both theorems are somewhat  magical, but each one on its own.
We will come back to the
 comparison between these two formulations  in future work devoted to the special case of almost complex manifolds.

Recently, using analytic methods adapted from those of Braverman, Ma, Tian and Zhang \cite{Tian-Zhang98,Braverman-02,Ma-Zhang14,Braverman-14},
Hochs-Mathai \cite{Hochs-Mathai2014} and Hochs-Song \cite{Hochs-Song2015}  have extended our theorem to other natural settings
where the group and/or the manifold are not compact. Note that in their works, the authors have to use our result in the compact setting to obtain these extensions.

\subsection{Description of the results}

We now give a detailed description of the theorem proved in this article.

Let $M$ be a compact connected manifold. We assume that $M$ is even dimensional and oriented.
We consider a spin$^c$ structure on $M$, and denote by $\mathcal S$   the corresponding  spinor bundle.
Let $K$ be a compact connected Lie group acting on $M$ and $\Scal$ and
we denote by $D: \Gamma(M, \Scal^+)\to \Gamma(M, \Scal^-)$ the corresponding $K$-equivariant $\spinc$ Dirac operator.

Our aim is to describe the space of $K$-invariant solutions, or more generally,
the equivariant index  of $D$, denoted  by $\Qcal_K(M,\Scal)$. It belongs to the Grothendieck group of representations of $K$:
$$
\Qcal_K(M,\Scal) = \sum_{\pi\in \widehat{K}} \mm(\pi) \ \pi.
$$

Consider the determinant line bundle $\det(\Scal)$ of the  spin$^c$ structure. This is a $K$-equivariant complex line bundle on $M$.
The choice of a $K$-invariant hermitian metric  and of a $K$-invariant hermitian connection  $\nabla$ on $\det(\Scal)$
determines a  moment map
$$
\Phi_{\Scal} : M\to \kgot^*.
$$
If $M$ is spin and $\Scal=\Scal_{\mathrm{spin}}\otimes L$, then $\det(\Scal)=L^{\otimes 2}$ and
$\Phi_{\Scal}$ is equal to the moment map $\Phi_L$ associated to a connection on $L$.

We start to explain our result on the geometric description of $\mm(\pi)$ in the torus case.
The general case reduces (in spirit) to this case, using an appropriate slice for the $K$-action on $M$.

Let $K=T$ be a torus acting effectively on $M$. In contrast to the symplectic case, the image $\Phi_\Scal(M)$ might not be convex.
Let $\Lambda\subset \tgot^*$ be the lattice of weights. If $\mu\in \Lambda$, we denote by $\C_\mu$ the corresponding
one dimensional representation of $T$. The equivariant index $\Qcal_T(M,\Scal)$ decomposes as
$\Qcal_T(M,\Scal)=\sum_{\mu\in \Lambda} \mm_\mu\, \C_\mu$.

The topological space $M_{\mu}=\Phi_\Scal^{-1}(\mu)/T$, which may not be connected,
is an orbifold provided with a $\spinc$-structure when $\mu$ in $\tgot^*$ is a regular value of $\Phi_\Scal$.
 In this case  we  define the  integer $\QS(M_{\mu})$ as the index of the corresponding $\spinc$ Dirac operator on the
 orbifold  $M_{\mu}$. We can define $\QS(M_{\mu})$ even if $\mu$ is a singular value.
Postponing this definition, our result states that
\begin{equation}\label{eq:mm-mu}
\mm_\mu=\QS(M_{\mu}),\quad \forall\mu\in\Lambda.
\end{equation}

Here is the definition of $\QS(M_{\mu})$ (see Section \ref{sec:spin-index-singular}). We approach $\mu$
by a regular value $\mu+\epsilon$, and we define $\QS(M_{\mu})$ as the index of a $\spinc$
Dirac operator on the  orbifold  $M_{\mu+\epsilon}$, and this is independent of the choice of $\epsilon$ sufficiently close.
Remark here that $\mu$ has to be an interior point of  $\Phi_\Scal(M)$ in order for  $\QS(M_{\mu})$ to be non zero,
as otherwise we can take $\mu+\epsilon$ not in the image. In a forthcoming article, we will give  a more detailed description
of the function $\mu\to \QS(M_{\mu})$ in terms of locally quasi-polynomial functions on $\tgot^*$.

The identity (\ref{eq:mm-mu}) was obtained by Karshon-Tolman \cite{Karshon-Tolman93} when $M$ is a toric manifold, by
Grossberg-Karshon \cite{Grossberg-Karshon98} when $M$ is a locally toric space, and by
Cannas da Silva-Karshon-Tolman \cite{CdS-K-T}
when $\dim T=1$. In Figure \ref{nonample-intro}, we draw the picture of the function
$\mu\mapsto \QS(M_{\mu})$ for the Hirzebruch surface, and  a non ample line bundle on it (we give the details
of this example due to Karshon-Tolman in the last section).
The image of $\Phi$ is the union of the two large triangles in red and blue. The multiplicities are $1$
on the integral points of the interior of the red triangle, and $-1$ on the integral points of the interior of the blue triangle.

\begin{figure}
\begin{center}
  \includegraphics[width=2 in]{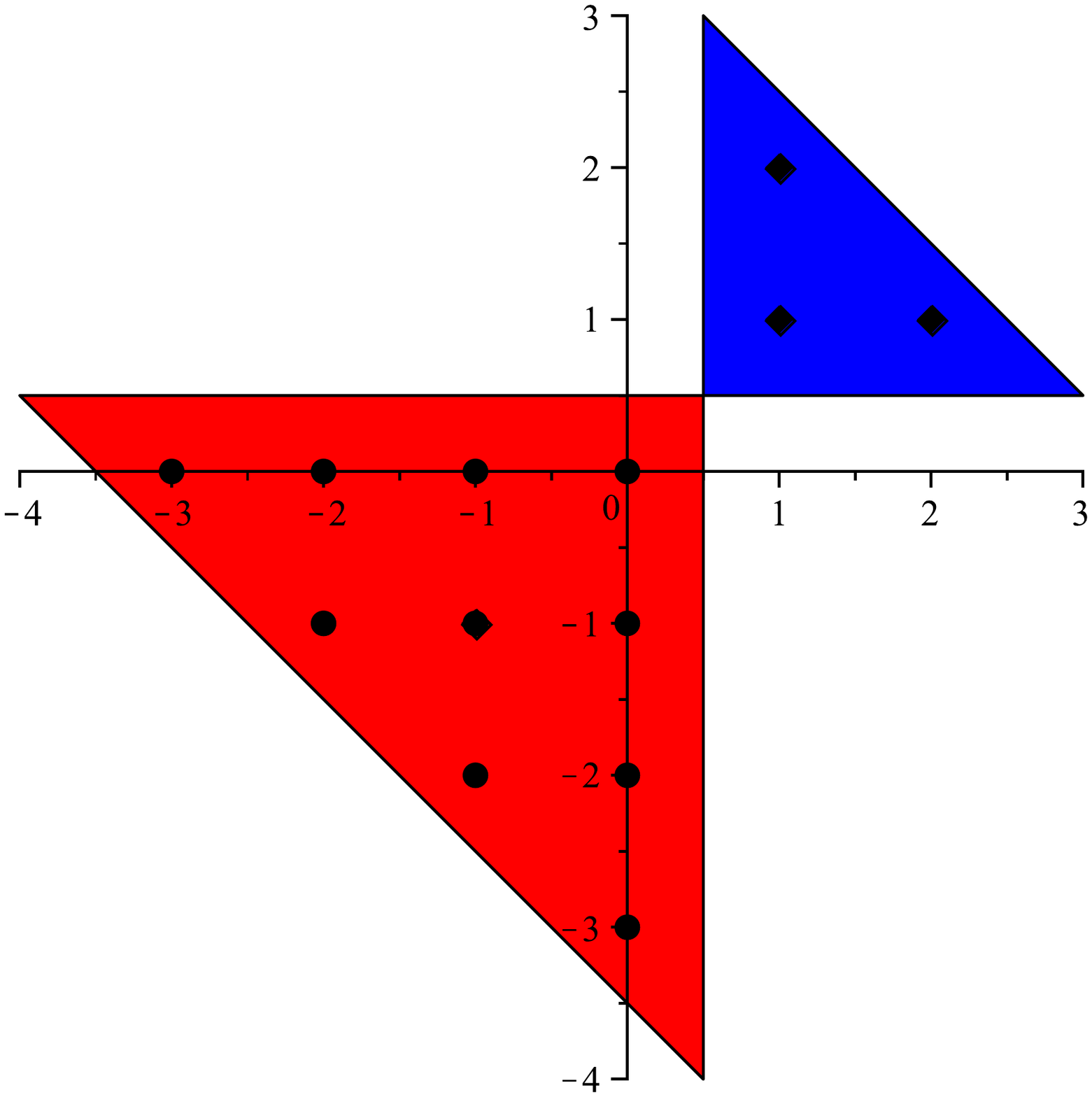}
\caption{$T$-multiplicities for  non ample bundle on Hirzebruch surface}
 \label{nonample-intro}
\end{center}
\end{figure}

Now consider the general case of a compact connected Lie group $K$ acting on
$M$ and $\Scal$. So we may assume that $([\kgot_M,\kgot_M])=([\hgot_M,\hgot_M])$ for $(\hgot_M)\in \Hcal_\kgot$, as otherwise $\Qcal_K(M,\Scal)=0$.

We say that a coadjoint orbit  $\Pcal\subset \kgot^*$ is admissible if $\Pcal$ carries a $\spinc$-bundle $\Scal_\Pcal$ such that
the corresponding moment map  is the inclusion $\Pcal\croc \kgot^*$. We denote simply by
$\QS_K(\Pcal)$ the element $\Qcal_K(\Pcal,\Scal_\Pcal)\in R(K)$. It is either $0$ or an irreducible representation
of $K$, and the map
$$
\Ocal\mapsto \pi_\Ocal:=\QS_K(\Ocal)
$$
defines a bijection between the regular admissible orbits and the dual $\widehat{K}$.
When $\Ocal$ is a regular admissible orbit, an admissible coadjoint orbit  $\Pcal$
is called an ancestor of $\Ocal$ (or an ancestor of $\pi_\Ocal$)
 if $\QS_K(\Pcal)=\pi_\Ocal$.

Denote by $\Acal((\hgot_M))$ the set of admissible orbits  of type $(\hgot_M)$.
If $\Pcal\in \Acal((\hgot_M))$, we can define the $\spinc$ index
$\QS(M_\Pcal)\in \Z$ of the reduced space $M_\Pcal=\Phi^{-1}_\Scal(\Pcal)/K$
(by a deformation procedure if $M_\Pcal$ is not smooth).

We obtain the following  theorem which is the main result of the paper.

\begin{theo}\label{theo:intro-QR}
Assume that $([{\mathfrak k}_M,{\mathfrak k}_M])=([{\mathfrak h_M},{\mathfrak h_M}])$ for $(\hgot_M)\in \mathcal H_\mathfrak k$.

$\bullet$ The multiplicity of the representation $\pi_\Ocal$ in $\Qcal_K(M,\Scal)$ is equal to
$$
\sum_{\Pcal}  \QS(M_{\Pcal})
$$
where the sum runs over the ancestors of $\Ocal$ of type $(\hgot_M)$. In other words
$$
\Qcal_K(M,\Scal)=\sum_{\Pcal\in \Acal((\hgot_M))}\QS(M_{\Pcal})\QS_K(\Pcal).
$$

\end{theo}

When we consider the orbit $K\rho$, the multiplicity of the representation $\pi_{K\rho}$ in $\Qcal_K(M,\Scal)$ is the space of $K$-invariant virtual solutions of $D$
and Theorem  \ref{theo:intro-QR} implies Theorem \ref{theo-intro-2}.

\medskip

It may be useful to rephrase this theorem by describing the parametrization of  admissible orbits by
parameters belonging to the closed Weyl chamber ${\mathfrak t}^*_{\geq 0}$. Let $\Lambda_{\geq 0}:=
\Lambda\cap {\mathfrak t}^*_{\geq 0}$ be the set of dominant weights, and let $\rho$ be the half sum of the positive roots.

The set of regular admissible orbits is indexed by the set $\Lambda_{\geq 0}+\rho$: if
$\lambda\in\Lambda_{\geq 0}+\rho$, the coadjoint orbit $K\lambda$ is regular admissible and $\pi_{K\lambda}$ is the representation with highest weight $\lambda-\rho$.

Denote by $\mathcal{F}$  the set of the relative interiors of the  faces of ${\mathfrak t}^*_{\geq 0}$.
Thus ${\mathfrak t}^*_{\geq 0}=\coprod_{\sigma\in {\mathcal F}} \sigma$. The face $\tgot^*_{>0}$ is the open face in $\Fcal$.

Let  $\sigma\in \mathcal{F}$. The stabilizer $K_\xi$ of a point $\xi\in \sigma$ depends only of $\sigma$.
We denote it by $K_\sigma$, and by $\kgot_\sigma$ its Lie algebra. We choose on ${\mathfrak k}_\sigma$
the system  of positive roots compatible with $\tgot^*_{\geq 0}$, and let $\rho^{K_\sigma}$ be the corresponding $\rho$.
When $\mu\in\sigma$, the coadjoint orbit $K\mu$
is admissible  if and only if
 $\lambda=\mu-\rho+\rho^{K_\sigma} \in \Lambda$.

The map $\Fcal\longrightarrow \mathcal{H}_\kgot$, $\sigma\mapsto (\kgot_\sigma)$, is surjective but not injective. We denote by $\Fcal(M)$ the set of faces of $\tgot^*_{\geq 0}$ such that $(\kgot_\sigma)=(\hgot_M)$.

Using the above parameters, we may rephrase Theorem \ref{theo:intro-QR} as follows.

\begin{theo}\label{th:mult}
Assume that $([{\mathfrak k}_M,{\mathfrak k}_M])=([{\mathfrak h_M},{\mathfrak h_M}])$ with $(\hgot_M)\in \mathcal H_\mathfrak k$.
Let $\lambda\in \Lambda_{\geq 0}+\rho$ and let $\mm_\lambda\in\Z$ be the multiplicity of  the  representation
$\pi_{K\lambda}$ in $\Qcal_K(M,\Scal)$.
We have
\begin{equation}\label{eq:m-lambda}
\mm_\lambda=\sum_{\stackrel{\sigma\in \Fcal(M)}{\lambda-\rho^{K_\sigma}\in \sigma}} \QS(M_{K(\lambda-\rho^{K_\sigma})}).
\end{equation}
\end{theo}

More explicitly, the sum (\ref{eq:m-lambda}) is taken over the faces $\sigma$ of the Weyl chamber such that
\begin{equation}\label{eq:condition}
([\kgot_M,\kgot_M])=([\kgot_\sigma,\kgot_\sigma]),\hspace{0.5cm}\Phi(M)\cap \sigma\neq \emptyset,\hspace{0.5cm}\lambda\in\{\sigma +\rho^{K_\sigma}\}.
\end{equation}

In Section \ref{sec:induced-example}, we give an example  of a $SU(3)$-manifold $M$ with generic stabilizer $SU(2)$, and a $\spinc$ bundle $\Scal$
where several $\sigma$
contribute to the multiplicity of a representation $\pi_{K\lambda}$ in $\Qcal_K(M,\Scal)$.
On Figure \ref{ancestors}, the picture of the decomposition of
 $\Qcal_K(M,\Scal)$
is given in terms of the representations $\QS_K(\Pcal)$ associated to the ancestors $\Pcal$ of type $(\hgot_M)=({\mathfrak s}{\mathfrak u}(2))$.
All reduced spaces  are points, but the multiplicity $\QS(M_{\Pcal})$ are equal to $-1$, following from the orientation rule. On the picture,
 the links between admissible  regular orbits $\Ocal$ and their ancestors $\Pcal$ are indicated by segments.
 We see that  the trivial representation of $K$ has two ancestors $\Pcal_1$ and $\Pcal_2$,
  of type $(\hgot_M)$ so that
the multiplicity of the trivial representation is equal to
$$ \QS(M_{\Pcal_1})+\QS(M_{\Pcal_2})=-2
$$
and comes from two different faces of the Weyl chamber.

\begin{figure}
\begin{center}
\includegraphics[width=2 in]{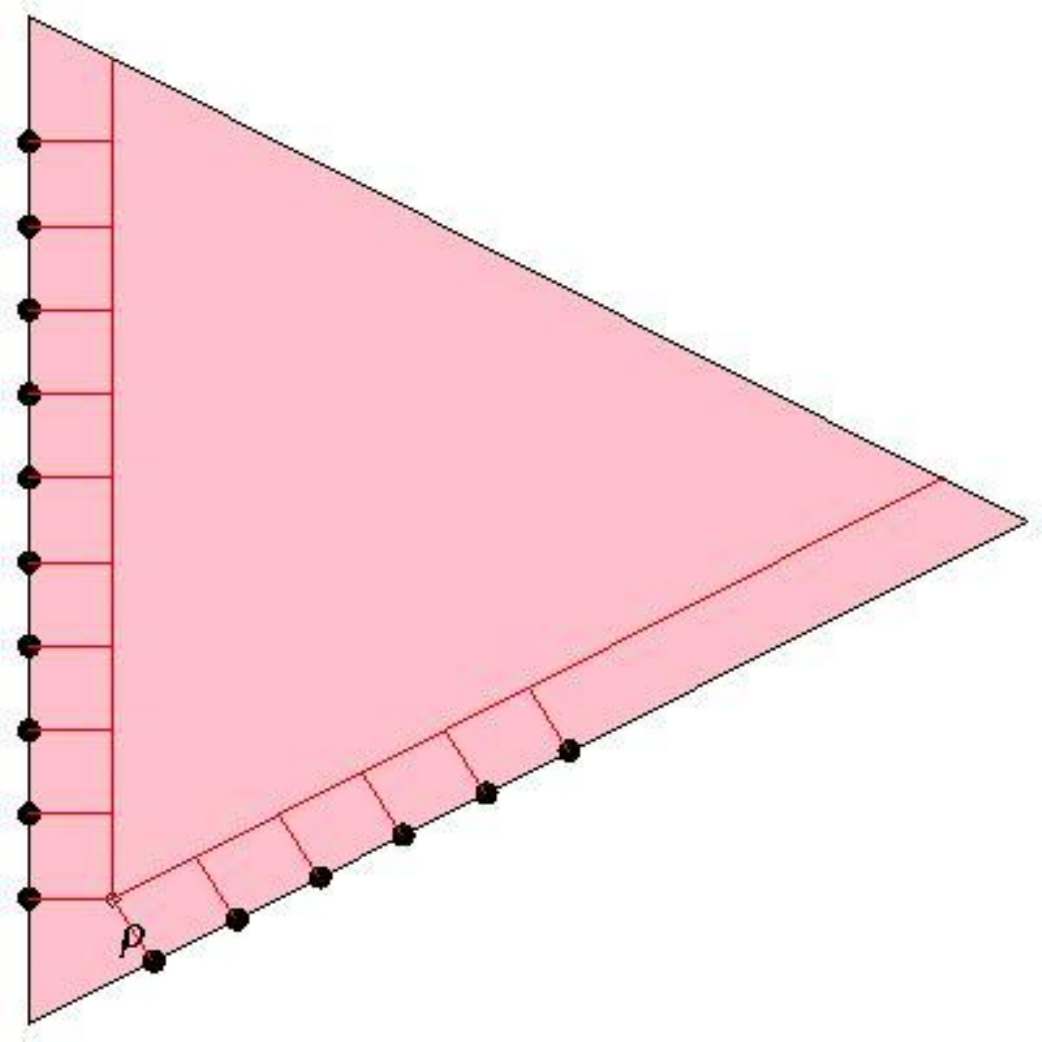}
\caption{$K$-multiplicities and ancestors}
 \label{ancestors}
\end{center}
\end{figure}

\subsection{Strategy}

The moment map $\Phi_\Scal$ allows us to define the Kirwan vector field $\kappa_{\Scal}$ on $M$:
at $m\in M$, $\kappa_\Scal$ is the tangent vector obtained by the infinitesimal action of
 $-\Phi_\Scal(m)$ at $m\in M$. Our proof is based on a localization procedure using the
 vector field $\kappa_\Scal$. Before going into the
details, let us recall the genealogy of the method.

In \cite{Atiyah-Bott-YM}, Atiyah and Bott calculate the cohomology of moduli spaces of vector bundles over Riemann
surfaces by using a stratification defined by the Yang-Mills functional. This functional turns to  be the square of a
moment map (in a infinite dimensional setting). Their approach was developed by Kirwan in \cite{Kirwan84}
to relate the cohomology of the Munford GIT quotient with the equivariant cohomology of the initial manifold.
Recall that, in the Hamiltonian setting, the Kirwan vector field is the Hamiltonian vector field of the square of the moment map.

In \cite{Witten}, Witten proposed a non abelian localization procedure on the zero set of the Kirwan vector field for
the integration of equivariant classes. This  idea had a great influence in many other contexts. For example, Tian and
Zhang \cite{Tian-Zhang98} gave an analytical proof of the $[Q,R]=0$ theorem by deforming \`a la Witten
the Dolbeault-Dirac operator with the Kirwan vector field.

In this paper, we use our $K$-theoretic analogue of the Witten non abelian localization procedure. We defined a topological deformation  of symbols
by pushing the zero section of $\T^*M$ inside $\T^*M$ using the Kirwan vector field
$\kappa_\Scal$. 
Let us briefly explain the main consequences of this powerful tool which was initiated in \cite{Vergne96,pep-RR} and developed in \cite{pep-vergne:witten}.

In Witten non abelian localization formula, computation of integrals of equivariant cohomology classes on $M$ reduces to the study  of contributions coming from a neighborhood of $Z_\Scal$, the set of zeroes of the invariant vector field $\kappa_\Scal$.
 Here, our  $K$-theoretical non abelian localization formula  allows us to compute the index $\Qcal_K(M,\Scal)$
as a sum of equivariant indices of transversally elliptic operators associated to connected components
$Z$ of $Z_\Scal$. We are able to identify them with some basic transversally elliptic symbols whose indices were
computed by Atiyah-Singer (see \cite{Atiyah74}). Although these indices are infinite dimensional representations,
they are easier to understand than the original finite
dimensional representation $\Qcal_K(M,\Scal).$
The main difficulty is in estimating which components $Z$
contributes to the $K$-invariant part. We are able to do so,
using a miraculous estimate on distance  between admissible coadjoint orbits
proved in \cite{pep-vergne:magic}. Here is when the ancestors of the trivial representation
enter the game.
 As shown by the final result, we have (in contrast to the Hamiltonian setting) to take in account several components and to identify their contributions.

\subsection{Outline of the article}
Let us explain the contents of the different sections of the article, and their main use in the final proof.

$\bullet$ In Section \ref{sec:spinc-index}, we give the definition of the index of a $\spinc$-bundle.

$\bullet$ In Section \ref{sec:coadjoint-orbit}, we describe the canonical $\spinc$-bundle on admissible
coadjoint orbits (see (\ref{eq:spinor-O})). For a $K$-admissible coadjoint orbit $\Pcal$, we determine the regular
admissible orbit $\Ocal$ such that if $\QS_K(\Pcal)$ is not zero, then $\QS_K(\Pcal)=\pi_\Ocal$
(Proposition \ref{prop:notregular}).

In Proposition \ref{prop:infernal-with-trace}, we state the ``magical inequality''  that will be used over and over again in this article.

$\bullet$ In Section \ref{sec:computing-multiplicities}, we define our $K$-theoretical
 analogue of the Witten deformation and recall some of its properties
(proved in \cite{pep-RR,pep-vergne:witten}). It allows us to reduce the computation of $\Qcal_K(M,\Scal)$ to indices $q_Z$
of simpler transversally elliptic operators defined in neighborhoods of connected components of $Z_\Scal=\{\kappa_\Scal=0\}$.

We introduce a function $d_\Scal : Z_\Scal\to \R$. If $d_\Scal$ takes strictly positive values on some component
$Z$ of $Z_\Scal$, then the $K$-invariant part of the (virtual) representation $q_Z$ is equal to $0$ (Proposition \ref{prop:annulation-d-S}). This is a very important technical proposition.

If $\Ocal$ is an admissible regular coadjoint orbit, the shifting trick leads us to study the manifold $M\times \Ocal^*$ with
$\spinc$-bundle $\Scal\otimes \Scal_{\Ocal^*}$. We want to select the component $Z$ of $Z_{\Scal\otimes \Scal_{\Ocal^*}}$
so that $[q_Z]^K$ is not zero.

Here is where we discover that, for $\Qcal_K(M,\Scal)$ to be non zero, it is necessary that the semi-simple part of the generic stabilizer
$\kgot_M$ of the action of $K$ on $M$ is equal to the semi-simple part of a Levi subalgebra $\hgot$ of $\kgot$. Let $H$ be the connected subgroup of $K$ with Lie algebra $\hgot$.
It follows that  such a   component  $Z$ is described rather simply  as an induced manifold $K\times_H(Y\times o(\hgot))$,
with $Y$ a $H/[H,H]$ manifold, and $o(\hgot)$ the $[H,H]$-orbit of the corresponding $\rho^{[H,H]}$ element.
Then we use the fact that the quantization of the orbit of $\rho$ is the trivial representation.
In fact, to determine  the contributing  components $Z$, we study  a function $d_\Ocal:Z_{\Scal\otimes \Scal_{\Ocal^*}}\to \R$ relating the representation  of $K_m$
on $\T_mM$ and the norm of $\Phi_\Scal(m)$. Here $K_m$ is the stabilizer of $m\in M$.
It relies on the ``magical inequality" (Proposition \ref{prop:infernal-with-trace}) on distance of regular
weights to faces of the Weyl chamber proved in \cite{pep-vergne:magic}.

$\bullet$ In Section \ref{sec:multiplicity}, we explain how to define indices on singular reduced spaces. The main theorem is their invariance under small deformation. We then have done all the work needed to be able to prove the main theorem.

$\bullet$ The last section is dedicated to some simple examples intended to show several features of our theory.

\section*{Acknowledgments}
We wish to thank the Research in Pairs program at Mathematisches \break
Forschungsinstitut Oberwolfach (February 2014), where this work was started.
The second author wish to thank Michel Duflo for many discussions.

\begin{center}
\bf Notations
\end{center}

Throughout the paper :
\begin{itemize}
\item $K$ denotes a compact connected Lie group with Lie algebra $\kgot$.
\item $T$ is a maximal torus in $K$ with Lie algebra $\tgot$.
\item $\Lambda\subset \tgot^*$ is the weight lattice  of $T$ : every $\mu\in \Lambda$ defines a $1$-dimensional
$T$-representation, denoted $\C_\mu$, where $t=\exp(X)$ acts by $t^\mu:= e^{i\langle\mu, X\rangle}$.

\item We fix a $K$-invariant inner product $(\cdot,\cdot)$ on $\kgot$. This allows us to identify $\kgot$ and $\kgot^*$ when needed.

We denote by $\langle \cdot,\cdot\rangle$ the natural duality between $\kgot$ and $\kgot^*$.

\item We denote by $R(K)$ the representation ring of $K$ : an element $E\in R(K)$ can be represented
as finite sum $E=\sum_{\mu\in\what{K}}\mm_\mu \pi_\mu$, with $\mm_\mu\in\Z$. The multiplicity
of the trivial representation is denoted $[E]^K$.

\item We denote by $\hat R(K)$ the space of $\Z$-valued functions on $\hat K$. An element $E\in
    \hat R(K)$ can be represented
as an infinite sum $E=\sum_{\mu\in\what{K}}\mm(\mu) V_\mu$, with $\mm(\mu)\in\Z$.

\item If $H$ is a closed subgroup of $K$, the induction map $\mathrm{Ind}_H^K: \hat{R}(H)\to \hat{R}(K)$ is the dual of the restriction morphism $R(K)\to R(H)$.

\item When $K$ acts on a set $X$, the stabilizer subgroup of $x\in X$ is denoted $K_x:=\{k\in K\ \vert\ k\cdot x=x\}$. The Lie algebra of $K_x$ is denoted $\kgot_x$.

\item An element $\xi\in \kgot^*$  is called regular if $K_\xi$ is a maximal torus of $K$.

\item When $K$ acts on a manifold $M$, we denote $X_M(m):=\frac{d}{dt}\vert_{t=0} e^{-tX}\cdot m$
the vector field generated by $-X\in \kgot$. Sometimes we will also use the notation $X_M(m)=-X\cdot m$.
The set of zeroes of the vector field $X_M$ is denoted $M^X$.

\item If $V$ is a complex (ungraded) vector space, then the exterior space $\bigwedge V=\bigwedge^+ V\oplus \bigwedge^-V$ will be $\Z/2\Z$ graded in even and odd elements.

\item
If $E_1=E_1^+\oplus E_1^-$ and
 $E_2=E_2^+\oplus E_2^-$  are two
$\Z/2\Z$ graded vector spaces (or vector bundles), the tensor product $E_1\otimes E_2$ is $\Z/2\Z$-graded with
$(E_1\otimes E_2)^+=E_1^+\otimes E_2^+\oplus E_1^-\otimes E_2^-$ and
$(E_1\otimes E_2)^-=E_1^-\otimes E_2^+\oplus E_1^+\otimes E_2^-$.
Similarly  the spaces $\End(E_i)$ are $\Z/2\Z$ graded. The action of
$\End(E_1)\otimes \End(E_2)$ on $E_1\otimes E_2$ obeys the usual sign rules: for example, if $f\in \End(E_2)^-$, $v_1\in E_1^-$ and $v_2\in E_2$, then
$f(v_1\otimes v_2)= -v_1 \otimes f v_2$.

\item If $E$ is a vector space and $M$ a manifold, we denote by
    $[E]$ the trivial vector bundle on $M$ with fiber $E$.

\end{itemize}

\section{Spin$^c$ equivariant index}\label{sec:spinc-index}

\subsection{Spin$^c$ modules}

Let $V$ be an oriented Euclidean space of even dimension $n=2\ell$. We denote by  $\mathrm{Cl}(V)$  its Clifford algebra.
If $e_1,\ldots,e_{n}$ is an oriented orthonormal frame of $V$, we define the element
$$
\epsilon:= (i)^\ell e_1\cdots e_{n}\in \mathrm{Cl}(V)
$$
that depends only of the orientation. We have $\epsilon^2=1$ and $\epsilon v=-v\epsilon$ for any $v\in V$.

If $E$ is a $\mathrm{Cl}(V)$-module, the Clifford map is denoted $\clif_E : \mathrm{Cl}(V)\to \End(E)$. We see then
that the element of order two $\epsilon_E:=\clif_E(\epsilon)$ defines a $\Z/2\Z$-graduation on $E$ by defining
$E^{\pm}:=\ker(\mathrm{Id}_E\mp\epsilon_E)$. Moreover the maps $\clif_E(v):E\to E$ for $v\in V$ interchanges
the subspace $E^\pm$. This graduation will be called the canonical graduation of the Clifford module $E$.

 We follow the conventions of \cite{B-G-V}.
Recall the following fundamental fact.
\begin{prop}
Let $V$ be an even dimensional Euclidean space.
\begin{itemize}
\item There exists a complex $\mathrm{Cl}(V)$-module $S$ such that the Clifford morphism
$\clif_S : \mathrm{Cl}(V)\to \End(S)$ induces an isomorphism of complex algebra $\mathrm{Cl}(V)\otimes\C\simeq \End(S)$.
\item The Clifford module $S$ is unique up to isomorphism. We call it  the spinor $\mathrm{Cl}(V)$-module.

\item Any complex $\mathrm{Cl}(V)$-module $E$ has the following decomposition
\begin{equation}\label{eq:iso-spinor}
E\simeq  S\otimes \hom_{\mathrm{Cl}(V)}(S,E)
\end{equation}
where $\hom_{\mathrm{Cl}(V)}(S,E)$  is the vector space spanned by the $\mathrm{Cl}(V)$-complex linear
maps from $S$ to $E$. If $V$ is oriented and the Clifford
modules $S$ and $E$ carry their canonical grading, then (\ref{eq:iso-spinor}) is an isomorphism of graded Clifford $\mathrm{CL}(V)$-modules.

\end{itemize}
\end{prop}


Let $V=V_1\oplus V_2$ be an orthogonal decomposition of even dimensional Euclidean spaces. We choose an orientation $o(V_1)$ on $V_1$.
There is a one-to-one correspondence between the graded $\mathrm{Cl}(V_2)$-modules and the graded $\mathrm{Cl}(V)$-modules defined as follows.
Let $S_1$ be the spinor module for $\mathrm{Cl}(V_1)$.
If $W$ is a $\mathrm{Cl}(V_2)$-module,  the vector space $E:=S_1\otimes W$ is a  $\mathrm{Cl}(V)$-module with the Clifford map defined by
$$
\clif_E(v_1\oplus v_2):=\clif_{S_1}(v_1)\otimes \mathrm{Id}_W + \epsilon_{S_1}\otimes \clif_{W}(v_2).
$$
Here $v_i\in V_i$ and $\epsilon_{S_1}\in \End(S_1)$ defines the canonical graduation of $S_1$.
Conversely, if $E$ is a graded $\mathrm{Cl}(V)$-module,   the vector space $W:=\hom_{\mathrm{Cl}(V_1)}(S_1,E)$
formed by the complex linear maps $f:S_1\to E$ commuting with the action of $\mathrm{Cl}(V_1)$ has
 a natural structure of $\mathrm{Cl}(V_2)$ graded module and $E\simeq S_1\otimes W$.

If we fix an orientation $o(V)$ on $V$, it fixes an orientation $o(V_2)$ on $V_2$ by the relation $o(V)=o(V_1)o(V_2)$. Then the Clifford modules
$E$ and $W$ carries their canonical $\Z/2\Z$ graduation, and  $E\simeq S_1\otimes W$  becomes an identity
of graded Clifford modules.

\medskip

\begin{exam}\label{ex:wedge-H}
Let  $H$ be an Euclidean vector space equipped with a complex structure $J\in O(H)$: we consider
the complex vector space  $\bigwedge_J H$.
Denote by $m(v)$ the exterior multiplication by $v$.
 The action $\clif$ of $H$ on $\bigwedge_J H$ given by
$\clif(v)=m(v)-m(v)^*$ satisfies $c(v)^2=-\|v\|^2 {\rm Id}$.
Thus, $\bigwedge_J H$, equipped with the action $\clif$, is a realization of the spinor module
for $H$. Note that the group $\mathrm{U}(J)$ of unitary transformations of
$H$ acts naturally on $\bigwedge_J H$. If one choose the orientation on $H$ induced by the complex structure, one sees that the
canonical grading is $(\bigwedge_J H)^\pm=\bigwedge^{\pm}_J H$.

Suppose now that we have another complex structure $J'\in O(H)$ :  the vector space $\bigwedge_{J'} H$ is
another spinor module for $H$. We denote by $\epsilon_J^{J'}$ the ratio between the orientations defined by $J$ and $J'$.
One can check that
\begin{equation}\label{eq:wedge-J-J-prime}
\bigwedge_{J'} H\simeq \epsilon_J^{J'} \,\C_\chi\otimes \bigwedge_J H,
\end{equation}
as a graded $\mathrm{Cl}(H)$-module and also as a graded $\mathrm{U}(J')\cap\mathrm{U}(J)$-module. Here $\C_\chi$ is the $1$-dimensional representation of $\mathrm{U}(J')\cap\mathrm{U}(J)$ associated to the unique character $\chi$ defined by the relation
$\chi(g)^2=\det_{J'}(g)\det_J(g)^{-1}$, $\forall g\in \mathrm{U}(J')\cap\mathrm{U}(J)$.
\end{exam}

\medskip

%
%

\begin{exam}\label{ex:QpluqQ}
When $V= Q\oplus Q$ with $Q$ an Euclidean space, we can identify $V$ with $Q_\C$ by $(x,y)\to x\oplus iy$. Thus
$S_Q:=\bigwedge Q_\C$ is a realization of the spinor module for $V$. It carries a natural action of the orthogonal group
$\mathrm{O}(Q)$ acting diagonally. If $Q$ carries a complex structure $J\in O(Q)$, we can consider the spin modules
$\bigwedge_J Q$ and $\bigwedge_{-J} Q$ for $Q$. We have then the isomorphism
$S_Q\simeq \bigwedge_J Q\otimes\bigwedge_{-J} Q$ of graded $\mathrm{Cl}(V)$-modules (it is also an isomorphism of
$\mathrm{U}(J)$-modules).
\end{exam}

\subsection{Spin$^c$ structures}

Consider now the case of an Euclidean vector bundle $\Vcal\to M$ of {\em even rank}. Let $\mathrm{Cl}(\Vcal)\to M$ be the associated Clifford algebra bundle. A complex vector bundle $\Ecal\to M$ is a $\mathrm{Cl}(\Vcal)$-module if there is a bundle algebra morphism
$\clif_\Ecal : \mathrm{Cl}(\Vcal)\longrightarrow \End(\Ecal)$.

\begin{defi}
Let $\Scal\to M$ be  a $\mathrm{Cl}(\Vcal)$-module  such that the  map  $\clif_\Scal$ induces an isomorphism
$\mathrm{Cl}(\Vcal)\otimes_\R \C\longrightarrow \End(\Scal)$. Then we say that
$\Scal$ is a $\spinc$-bundle  for  $\Vcal$.
\end{defi}

As in the linear case, an orientation on the vector bundle $\Vcal$ determines a $\Z/2\Z$ grading of the vector bundle $\Scal$
(called the canonical graduation) such that for any $v\in \Vcal_m$, the linear map\footnote{The map $\clif_{\Scal}(m,-) : \Vcal_m\to \End(\Scal_m)$
will also be denoted by $\clif_{\Scal_m}$.} $\clif_{\Scal}(m,v): \Scal_m\to \Scal_m$ is odd.

\begin{exam}
When $\Hcal\to M$ is a Hermitian vector bundle, the complex vector bundle $\bigwedge \Hcal$ is a $\spinc$ bundle for $\Hcal$.  If one choose the orientation
of the vector bundle $\Hcal$ induced by the complex structure, one sees that the
canonical grading is $(\bigwedge \Hcal)^\pm=\bigwedge^{\pm} \Hcal$.
\end{exam}

We assume that the vector bundle $\Vcal$ is oriented, and we consider two $\spinc$-bundles $\Scal,\Scal'$
for $\Vcal$, both with their canonical grading.
We have the following identity of graded $\spinc$-bundles : $\mathcal{S}'\simeq\mathcal{S}\otimes \Lbb_{\Scal,\Scal'}$
 where $\Lbb_{\Scal,\Scal'}$ is a complex line bundle on $M$ defined by the relation
\begin{equation}\label{eq:differenceLine}
 \Lbb_{\Scal,\Scal'}:=\hom_{\mathrm{Cl}(\Vcal)}(\mathcal{S},\mathcal{S}').
\end{equation}

\begin{defi}\label{def:determinantLine}
Let $\Vcal\to M$ be an Euclidean vector bundle of even rank. The determinant line bundle of a $\spinc$-bundle $\Scal$ on $\Vcal$ is the line bundle $\det(\Scal)\to M$ defined by the relation
$$
\det(\Scal):=\hom_{\mathrm{Cl}(\Vcal)}(\overline{\mathcal{S}},\mathcal{S})
$$
where $\overline{\mathcal{S}}$ is the $\mathrm{Cl}(\Vcal)$-module with opposite complex structure.
\end{defi}

\begin{exam}
When $\Hcal\to M$ is a Hermitian vector bundle, the determinant line bundle of the $\spinc$-bundle $\bigwedge \Hcal$ is
$\mathrm{det}(\Hcal):=\bigwedge^{\mathrm{max}} \Hcal$.
\end{exam}

If $\Scal$ and $\Scal'$ are two $\spinc$-bundles for  $\Vcal$, we see that
\begin{eqnarray*}
\det(\Scal')=\det(\Scal)\otimes \Lbb_{\Scal,\Scal'}^{\otimes 2}.
\end{eqnarray*}

Assume that $\Vcal=\Vcal_1\oplus \Vcal_2$ is an orthogonal sum of Euclidean vector bundles of even rank. We assume that $\Vcal_1$ is oriented, and let $\Scal_1$ be a
$\spinc$-bundle for $\Vcal_1$ that we equip with its canonical grading. If $\Ecal$ is a Clifford bundle for $\Vcal$, then we have the following isomorphism\footnote{The proof is identical to the linear case explained earlier.}
\begin{equation}\label{eq:iso-clifford-bundle}
\Ecal\simeq \Scal_1\otimes \Wcal
\end{equation}
where $\Wcal:=\hom_{\mathrm{Cl}(\Vcal_1)}(\mathcal{S}_1,\mathcal{E})$  is a Clifford bundle for $\Vcal_2$. If $\Vcal$ is oriented,
it fixes an orientation $o(\Vcal_2)$ on $\Vcal_2$ by the relation $o(\Vcal)=o(\Vcal_1)o(\Vcal_2)$. Then the Clifford modules
$\Ecal$ and $\Wcal$ carries their canonical $\Z/2\Z$ grading, and  (\ref{eq:iso-clifford-bundle})  becomes an identity of graded Clifford modules.

In the particular situation where $\Scal$ is a $\spinc$-bundle for $\Vcal$, then $\Scal\simeq \Scal_1\otimes \Scal_2$ where
$\Scal_2:=\hom_{\mathrm{Cl}(\Vcal_1)}(\mathcal{S}_1,\mathcal{S})$  is a $\spinc$-bundle for $\Vcal_2$. At the level of determinant
line bundles we obtain $\det(\Scal)=\det(\Scal_1)\otimes\det(\Scal_2)$.

\medskip

Let us end this section by recalling the notion of Spin-structure and $\spinc$-structure.
 Let $\Vcal\to M$ be an oriented Euclidean vector bundle of rank $n$, and let $\mathrm{P}_{SO}(\Vcal)$ be its orthogonal frame bundle : it is a principal $\mathrm{SO}_n$ bundle over $M$.

Let us consider the spinor group $\mathrm{Spin}_n$ which is the double cover of the group $\mathrm{SO}_n$.
The group  $\mathrm{Spin}_n$ is a subgroup of the group $\mathrm{Spin}^{c}_n$ which covers  $\mathrm{SO}_n$ with fiber $U(1)$.

A Spin structure on $\Vcal$ is  a $\mathrm{Spin}_n$-principal bundle $\mathrm{P}_{Spin}(\Vcal)$ over $M$ together with a $\mathrm{Spin}_n$- equivariant map $\mathrm{P}_{Spin}(\Vcal)\to \mathrm{P}_{SO}(\Vcal)$.

We assume now that $\Vcal$ is of even rank $n=2\ell$. Let $\mathrm{S}_{n}$ be the irreducible complex spin representation of $\mathrm{Spin}_{n}$. Recall that $\mathrm{S}_{n}=\mathrm{S}_{n}^+\oplus\mathrm{S}_{n}^-$  inherits a canonical Clifford action $c:\R^{n} \to\End(\mathrm{S}_{n})$ which is $\mathrm{Spin}_n$-equivariant, and which interchanges the graduation: $c(v): \mathrm{S}_{n}^\pm\to \mathrm{S}_{n}^\mp$. The spinor bundle attached to the Spin-structure $\mathrm{P}_{Spin}(\Vcal)$ is
$$
\Scal:=\mathrm{P}_{Spin}(\Vcal)\times_{\mathrm{Spin}_{n}}\mathrm{S}_{n}.
$$


A $\spinc$-bundle for $\Vcal$ determines
a $\spinc$ structure, that is a principal bundle  over $M$ with structure group
$\mathrm{Spin}^{c}_n$. When $\Vcal$ admits a  Spin-structure, any $\spinc$-bundle for $\Vcal$ is of the form
$\Scal_L=\Scal_{\mathrm{spin}}\otimes L$  where $\Scal_{\mathrm{spin}}$
is the spinor bundle attached to the Spin-structure and
 $L$ is a line bundle on $M$. Then the determinant line bundle for $\Scal_L$ is $L^{\otimes 2}$.

\subsection{Moment maps and Kirwan vector field}\label{sec:abstract-moment-map}

In this section, we consider the case of a Riemannian  manifold $M$ acted on by a compact Lie group $K$.
Let $\Scal\to M$ be a $\spinc$-bundle on $M$.
If the
$K$-action lifts to the $\spinc$-bundle $\Scal$ in such a way that the bundle map
$\mathrm{c}_\Scal : \mathrm{Cl}(\T M)\to \End(\Scal)$ commutes with the $K$-action, we say
that $\Scal$ defines a $K$-equivariant  $\spinc$-bundle on $M$. In this case, the $K$-action lifts also
to the determinant line bundle $\det(\Scal)$.  The choice of an invariant Hermitian
connection $\nabla$ on $\det(\Scal)$ determines an equivariant map $\Phi_\Scal: M\to \kgot^{*}$ and a 2-form
$\Omega_\Scal$ on $M$ by means of the  Kostant relations
\begin{equation}\label{eq:kostant-L}
    \Lcal(X)-\nabla_{X_M}=2i \langle\Phi_\Scal,X\rangle\quad \mathrm{and} \quad \nabla^2= -2i \Omega_\Scal
\end{equation}
for every $X\in\kgot$. Here $\Lcal(X)$ denotes the infinitesimal action on the sections of $\det(\Scal)$.
We will say that $\Phi_\Scal$ is the moment map for $\Scal$ (it depends however of the choice of a connection).

Via the equivariant
 Bianchi formula, Relations (\ref{eq:kostant-L}) induce the  relations
\begin{equation}\label{eq:hamiltonian-action}
    \iota(X_M)\Omega_\Scal= -d\langle\Phi_\Scal,X\rangle \quad \mathrm{and} \quad d\Omega_\Scal=0
\end{equation}
for every $X\in\kgot$.

In particular the function $m\to \langle\Phi_\Scal(m),X\rangle$ is locally constant on $M^X$.

\begin{rem}\label{rem:L-S-gamma}
Let $b\in\kgot$ and  $m\in M^b$, the set of zeroes of $b_M$.
 We consider the linear actions
$\Lcal(b)\vert_{\Scal_m}$ and $\Lcal(b)\vert_{\det(\Scal)_m}$ on the fibers at $m$ of the $\spinc$-bundle
$\Scal$ and the line bundle $\det(\Scal)$. Kostant relations imply $\Lcal(b)\vert_{\det(\Scal)_m}=2i\langle \Phi_\Scal(m),b\rangle.$
The  irreducibility of $\Scal$ implies that $$\Lcal(b)\vert_{\Scal_m}=i\, \langle\Phi_\Scal(m),b\rangle \, \mathrm{Id}_{\Scal_m}.
$$
Furthermore the function
$m\to \langle\Phi_{\Scal}(m),b\rangle$ is locally constant on $M^b$.
\end{rem}

Note that the closed $2$-form $\Omega_\Scal$, which is {\em half} of the curvature of $\det(\Scal)$, is not (in general) a
symplectic form. Furthermore, if we take any (real valued) invariant $1$-form $A$ on $M$,
$\nabla+ i A$ is another connection on $\det(\Scal)$. The corresponding curvature and  moment map will be modified in
$\Omega_\Scal - \frac{1}{2}dA$ and $\Phi_\Scal -\frac{1}{2}\Phi_A$ where $\Phi_A: M\to \kgot^{*}$ is
defined by the relation $\langle\Phi_A,X\rangle=-\iota(X_M)A$.

\medskip


\medskip

Let $\Phi: M\to\kgot$ be a  $K$ equivariant map. We  define the $K$-invariant vector field $\kappa_\Phi$ on $M$ by
\begin{equation}\label{eq:K-S}
\kappa_\Phi(m):=-\Phi(m)\cdot m,
\end{equation}
and  we call it the {\bf Kirwan vector field} associated to $\Phi$.
The set where $\kappa_\Phi$ vanishes is a K-invariant subset that we denote by
$
Z_\Phi\subset M$.

\medskip

We identify $\kgot^*$ with $\kgot$ by our choice of $K$-invariant scalar product and we will have a particular interest in the equivariant
map  $\Phi_\Scal: M\to\kgot^*\simeq\kgot$ associated to the $\spinc$-bundle $\Scal$.
In this case we may denote
the $K$-invariant vector field $\kappa_{\Phi_\Scal}$ simply by $\kappa_\Scal$ (even if it depends of the choice of a connection):
$$
\kappa_\Scal(m):=-\Phi_\Scal(m)\cdot m.
$$
and we denote $Z_\Phi$ by $Z_\Scal$.

As $\Phi_\Scal$ is a moment map,
we have the following basic description (see \cite{pep-RR,pep-vergne:witten}).

\begin{lem} \label{lem:crit} If the manifold $M$ is compact, the set $\Phi_\Scal(Z_\Scal)$ is a {\em finite} collection of coadjoint orbits.
For any coadjoint orbit $\Ocal=K\beta$, we have
$$
Z_\Scal\cap\Phi_\Scal^{-1}(\Ocal)= K(M^\beta\cap\Phi^{-1}_\Scal(\beta)).
$$

Here we have  identified $\beta\in \kgot^*$ with an  element  in $ \kgot$ still denoted by $\beta$.
Furthermore, any $\beta$ in the image $\Phi_\Scal(Z_\Scal)$ is such that $\|\beta\|^2$
is a critical value of the map $\|\Phi\|^2$.
 \end{lem}

\begin{rem}\label{rem: rigidity}
Although the map $\Phi_\Scal$ as well as the set $Z_\Scal$ vary when we vary the connection, we see that the image
$\Phi_\Scal(Z_\Scal)$ is contained in a finite set of coadjoint orbits that does not depend of
the connection (see \cite{pep-vergne:witten}).
\end{rem}

Figure \ref{fig:SU3-rho} describes the set $\Phi_\Scal(Z_\Scal)$ for the action of the diagonal torus of $K=\mathrm{SU}(3)$
on the orbit $K\rho$ equipped with its canonical $\spinc$-bundle.

\begin{figure}[!h]
\begin{center}
  \includegraphics[width=1 in]{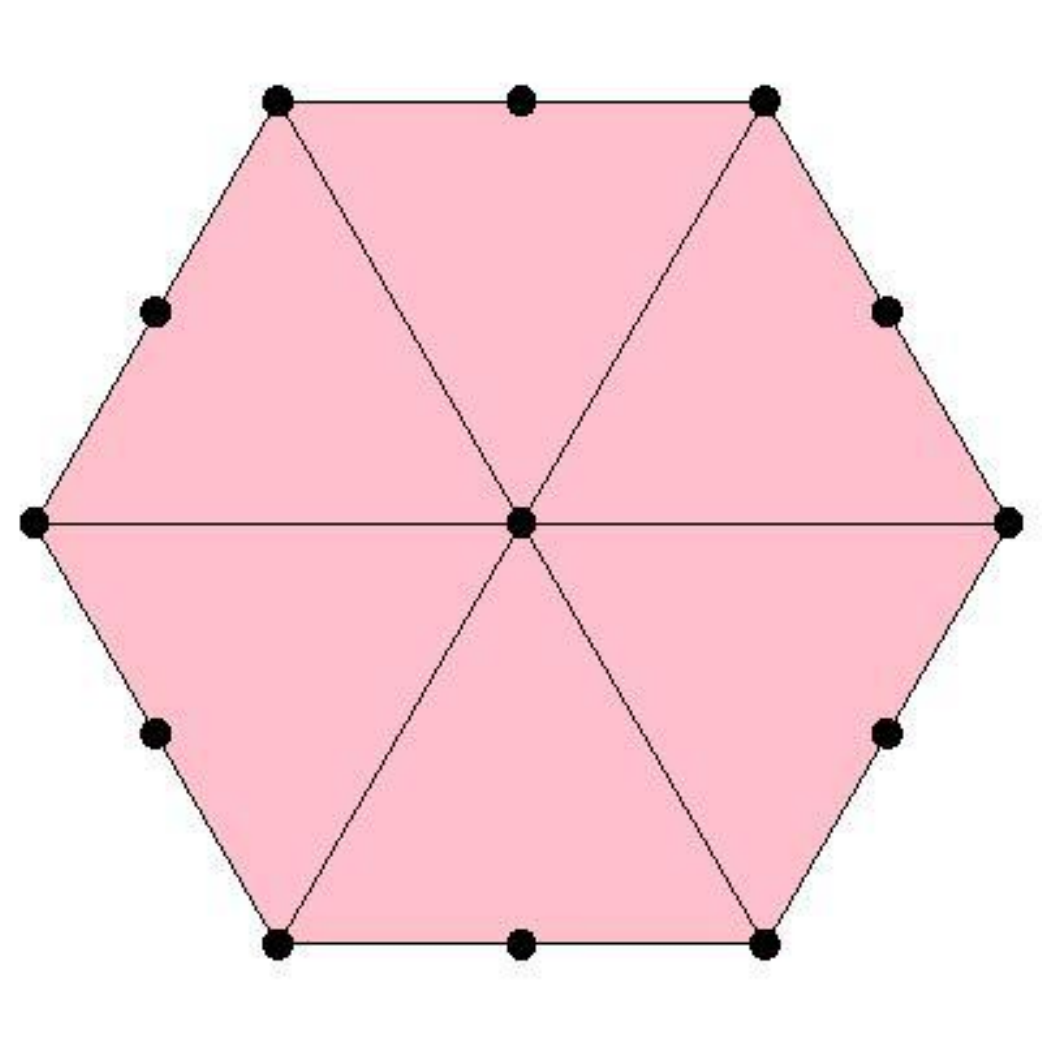}
\caption{ The set $\Phi_\Scal(Z_\Scal)$}
 \label{fig:SU3-rho}
\end{center}
\end{figure}
%

\subsection{Equivariant index}

Assume in this section that the Riemannian $K$-manifold $M$ is compact, even dimensional, oriented, and  equipped
with a $K$-equivariant $\spinc$-bundle $\Scal\to M$. The orientation
induces a decomposition $\Scal=\Scal^+\oplus \Scal^-$, and the corresponding Spin$^c$ Dirac operator
is a first order elliptic operator $\Dcal_\Scal: \Gamma(M,\Scal^+)\to \Gamma(M,\Scal^-)$ \cite{B-G-V,Duistermaat96}. Its principal symbol
is the bundle map  $\sigma(M,\Scal)\in \Gamma(\T^*M,\hom(p^*\Scal^+,p^*\Scal^-))$ defined by the relation
$$
\sigma(M,\Scal)(m,\nu)=\clif_{\Scal_m}(\tilde{\nu}): \Scal^+_m\longrightarrow\Scal_m^-.
$$
Here $\nu\in\T^* M\to \tilde{\nu}\in \T M$ is the identification defined by the Riemannian structure.

If $\Wcal\to M$ is a complex $K$-vector bundle, we can define similarly the twisted Dirac operator
$\Dcal_\Scal^\Wcal: \Gamma(M,\Scal^+\otimes \Wcal)\to \Gamma(M,\Scal^-\otimes \Wcal)$.

\begin{defi} \label{defi:symboldirac}
Let $\Scal\to M$ be an equivariant $\spinc$-bundle. We denote :

$\bullet$ $\Qcal_K(M,\Scal)\in R(K)$ the equivariant index of the operator $\Dcal_\Scal$,

$\bullet$ $\Qcal_K(M,\Scal\otimes \Wcal)\in R(K)$ the equivariant index of the operator $\Dcal_\Scal^\Wcal$.

\end{defi}

Let $\widehat{A}(M)(X)$ be the equivariant \^A-genus class of $M$: it is an equivariant analytic function from a neighborhood
of $0\in\kgot$ with value in the algebra of differential forms on $M$. Berline-Vergne equivariant index formula
\cite{B-G-V}[Theorem 8.2] asserts that
\begin{equation}\label{eq:ASS-index}
\Qcal_K(M,\Scal)(e^X)=(\frac{i}{2\pi})^{\frac{\dim M}{2}}\int_M e^{i(\Omega_\Scal+\langle\Phi_\Scal,X\rangle)}\widehat{A}(M)(X)
\end{equation}
for $X\in\kgot$ small enough.
 Here we write
 $\Qcal_K(M,\Scal)(e^X)$ for the trace of the element $e^X\in K$ in the virtual representation $\Qrm_K(M,\Scal)$ of $K$.
 It shows in particular that $\Qrm_K(M,\Scal)\in R(K)$ is a topological invariant : it only depends of the class
of the equivariant form $\Omega_\Scal+\langle\Phi_\Scal,X\rangle$, which represents {\em half} of the first equivariant Chern class of the line bundle $\det(\Scal)$.

\begin{exam}\label{P1}
We consider the simplest case of the theory. Let $M:=\Pbb^1(\C)$ be the projective space of (complex) dimension one. We write an element of $M$ as $[z_1,z_2]$ in homogeneous coordinates.
Consider the (ample) line bundle $\Lcal\to \Pbb^1$, dual of the tautological bundle.
Let $\Scal(n)$ be the $\spinc$-bundle $\bigwedge_\C \T M \otimes \Lcal^{\otimes n}$. The character $\Qcal_T(M, \Scal(n))$ is equal to
$H^{0}(\Pbb^1,\Ocal(n))-H^{1}(\Pbb^1,\Ocal(n))$.
Then for $n\geq 0$,
$$\Qcal_T(M, \Scal(n))=  \sum_{k=0}^{n}t^{k}.$$
Here $T=\{t\in \C; |t|=1\}$ acts on $[z_1,z_2]$ via $t\cdot [z_1,z_2]=[t^{-1}z_1,z_2]$.

%
%
%

\end{exam}

\section{Coadjoint orbits}\label{sec:coadjoint-orbit}

In this section, we describe $\spinc$-bundles on admissible coadjoint orbits of $K$  and the equivariant indices of the associated Dirac operators.

For any $\xi\in \kgot^*$, the stabilizer $K_\xi$ is a connected subgroup of $K$ with same rank. We denote by $\kgot_\xi$ its Lie algebra. Let $\mathcal{H}_{\mathfrak k}$ be the set of conjugacy classes of the reductive algebras ${\kgot}_\xi,\xi\in\mathfrak k^*$ it  contains the conjugacy class formed by the Cartan sub-algebras, and it contains also $\kgot$ (stabilizer of $0$).

We denote by $\mathcal S_{\mathfrak k}$ the set of conjugacy classes of the semi-simple parts $[\hgot,\hgot]$ of
the elements $(\hgot)\in \mathcal H_{\mathfrak k}$. The map $(\hgot)\to( [\hgot,\hgot])$ induces a bijection
between $\mathcal H_{\mathfrak k}$ and $\mathcal S_{\mathfrak k}$ (see \cite{pep-vergne:magic}).



We group the coadjoint orbits  according to the conjugacy class $(\hgot)\in \mathcal{H}_{\mathfrak k}$ of the stabilizer, and
we consider  the Dixmier sheet $\kgot^*_{(\hgot)}$ of orbits $K\xi$ with $\kgot_\xi$ conjugated to $\hgot$. The connected Lie subgroup  with Lie algebra $\hgot$ is denoted $H$, that is if $\hgot=\kgot_\xi$, then $H=K_\xi$.
We write $\hgot= \zgot \oplus [\hgot,\hgot]$ where $\zgot$ is the center and  $[\hgot,\hgot]$ is the semi-simple part of $\hgot$.
Thus $\hgot^*= \zgot^* \oplus [\hgot,\hgot]^*$ and
$\zgot^*$ is the set of elements in $ \hgot^*$ vanishing on the semi-simple part of $\hgot$.
We write $\kgot=\hgot\oplus [\zgot,\kgot]$, so we embed $\hgot^*$ in $\kgot^*$ as a $H$-invariant subspace,  that is
we consider an element $\xi\in \hgot^*$ also as an element of $\kgot^*$ vanishing on  $[\zgot,\kgot]$.
Let $\zgot^*_0$ be the set of $\xi\in \zgot^*$, such that $\kgot_\xi=\hgot$.
We see then that the Dixmier sheet  $\kgot^*_{(\hgot)}$ is equal to
$K\cdot \zgot^*_0$.

\medskip



%

\subsection{Admissible coadjoint orbits}

We first define the $\rho$-orbit.
Let $T$ be a Cartan subgroup of $K$.
Then $\tgot^*$ is imbedded in $\kgot^*$ as the subspace of $T$-invariant elements.
Choose a system of positive roots $\Delta^+\subset {\mathfrak t}^*$, and let
$\rho^K=\frac{1}{2}\sum_{\alpha>0}\alpha$.
The definition of $\rho^K$ requires
  the choice of a Cartan subgroup $T$ and of a  positive root system. However a different choice leads to a conjugate element. Thus we can make the following definition.

\begin{defi}\label{def:s-O}
We denote by
$o(\kgot)$ the coadjoint orbit of $\rho^K\in \kgot^*$.
We call $o(\kgot)$ the $\rho$-orbit.

\end{defi}

If $K$ is abelian, then $o(\kgot)$ is $\{0\}$.

\medskip

 The notion of admissible coadjoint orbit is defined in  \cite{du} for any real Lie group $G$.
When $K$ is a compact connected Lie group, we adopt the following equivalent definition:
  a coadjoint orbit  $\Ocal\subset \kgot^*$ is admissible  if  $\Ocal$ carries a $K$-equivariant
$\spinc$-bundle  $\Scal_\Ocal$, such that the associated moment map is the injection
$\Ocal\croc \kgot^*$.
If $K\xi$ is an admissible orbit, we also say that the element $\xi$ is admissible.
An admissible coadjoint orbit $\Ocal$ is oriented by its symplectic structure, and we denote
by $\QS_K(\Ocal):=\Qcal_K(\Ocal,\Scal_\Ocal)$ the corresponding equivariant spin$^c$ index.

We have  $\langle \xi, [\kgot_\xi,\kgot_\xi]\rangle=0$.
The quotient space $\qgot=\kgot/\kgot_\xi$ is equipped with the symplectic form $\Omega_\xi(\bar{X},\bar{Y}):=\langle \xi,[X,Y]\rangle$, and with a unique $K_\xi$-invariant complex structure $J_\xi$ such that $\Omega_\xi(-,J_\xi -)$ is a scalar product.
We denote by $\qgot^{\xi}$ the complex $K_\xi$-module $(\kgot/\kgot_\xi,J_\xi)$.

Any element $X\in \kgot_\xi$ defines a complex linear map $\mathrm{ad}(X): \qgot^\xi\to \qgot^\xi$.
\begin{defi}\label{def-rho-a}
 For any $\xi\in \kgot^*$, we denote $\rho(\xi)\in \kgot_\xi^*$ the element defined by
$$
\langle\rho(\xi),X\rangle= \frac{1}{2i}\tr_{\qgot^\xi}\mathrm{ad}(X), \quad X\in\kgot_\xi.
$$

We extend $\rho(\xi)$ to an element of $\kgot^*$, that we still  denote by $\rho(\xi)$.

\end{defi}

\bigskip

If $i\theta  : \kgot_\xi\to i\R$
is the differential of a character of $K_\xi$, we denote by $\C_\theta$ the corresponding $1$-dimensional representation
of $K_\xi$, and by $[\C_\theta]=K\times_{K_\xi}\C_\theta$ the corresponding  line bundle over the coadjoint orbit $K\xi\subset\kgot^*$.
The condition that $K\xi$ is admissible means that there exists a $\spinc$-bundle $\Scal$ on $K\xi$ such that
$\det(\Scal)=[\C_{2\xi}]$ ($2i\xi$ needs  to be the differential of a character of $K_\xi$).

\begin{lem}\label{lem:admissible}
\begin{itemize}

\item
 $\langle\rho(\xi),[\kgot_\xi,\kgot_\xi]\rangle=0$.

\item

The coadjoint orbit  $K\xi$ is admissible if and only if $i(\xi-\rho(\xi))$ is the differential of a $1$-dimensional representation of $K_\xi$.
\end{itemize}
\end{lem}

\begin{proof}
 Consider the character $k\mapsto \det_{\qgot^\xi}(k)$ of $K_\xi$. Its differential is  $2i\rho(\xi)$.
Thus $\langle\rho(\xi),[\kgot_\xi,\kgot_\xi]\rangle=0$.

We can equip $K\xi\simeq K/K_\xi$ with the $\spinc$-bundle
$\Scal_\xi:= K\times_{K_\xi}\bigwedge \qgot^\xi$
with determinant line bundle $\det(\Scal_{\xi})= [\C_{2\rho(\xi)}]$. Any other $K$-equivariant $\spinc$-bundle
on $K\xi$ is of the form $\Scal_{\xi}\otimes [\C_\theta]$
where $i\theta$
is the differential of a character of $K_\xi$.
 Then $\det(\Scal_{\xi}\otimes [\C_\theta])=[\C_{2\xi}]$ if and only if $\xi-\rho(\xi)=\theta$. The lemma then follows.
\end{proof}

\medskip

In particular the orbit $o(\kgot)$ is admissible.
Indeed if $\xi=\rho^K$, then $\xi-\rho(\xi)=0$.

\medskip

An admissible coadjoint orbit $\Ocal=K\xi$ is then equipped with the $\spinc$-bundle
\begin{equation}\label{eq:spinor-O}
\Scal_{\Ocal}^\pm:= K\times_{K_\xi}\left({\bigwedge}^\pm \qgot^\xi\otimes\C_{\xi-\rho(\xi)}\right).
\end{equation}

Its $\spinc$ equivariant index is
\begin{equation}\label{eq:induction-index-Ocal}
\QS_K(\Ocal)=\mathrm{Ind}_{K_\xi}^K\left(\bigwedge \qgot^\xi\otimes \C_{\xi-\rho(\xi)}\right).
\end{equation}
See \cite{pep-vergne:witten}.

The following proposition is well known (see \cite{pep-vergne:magic}).

\begin{prop}\label{prop:easyregular}
\begin{itemize}
\item The map $\Ocal\mapsto \pi_\Ocal:=\QS_K(\Ocal)$ defines a bijection between the set of regular admissible orbits and $\wK$.

\item  $\QS_K(o(\kgot))$ is the trivial representation of $K$.

\end{itemize}
\end{prop}

We now  describe the representation  $\QS_K(\Pcal)$ attached to any admissible orbit $\Pcal$ in terms of regular admissible orbits.

\begin{defi}\label{defi-shift}
 To any coadjoint orbit $\Pcal\subset\kgot^*$, we associate the coadjoint orbit $s(\Pcal)\subset \kgot^*$
which is defined as follows : if $\Pcal=K\mu$, take $s(\Pcal)=K\xi$ with $\xi\in \mu +o(\kgot_\mu)$.
We call $s(\Pcal)$ the shift of the orbit $\Pcal$.
\end{defi}
If $\Pcal$ is regular,  $s(\Pcal)=\Pcal$. If $\Pcal=\{0\}$, then $s(\Pcal)=o(\kgot)$.

\medskip

The following proposition summarises the results concerning the quantization of admissible orbits.

\begin{prop}[\cite{pep-vergne:magic}]
\label{prop:notregular}
 Let $\Pcal$ be an admissible orbit.
\begin{itemize}

\item  $\Pcal^*:=-\Pcal$ is also admissible and $\QS_K(\Pcal^*)=\QS_K(\Pcal)^*.$

\item  If  $s(\Pcal)$ is regular, then $s(\Pcal)$ is also admissible.

\item Conversely, if $\Ocal$ is regular and admissible, and $\Pcal$ is such that
    $s(\Pcal)=\Ocal$, then $\Pcal$ is admissible.

\item
\begin{itemize}

\item If $s(\Pcal)$  is not regular, then $\QS_K(\Pcal)=0$.

\item If  $s(\Pcal)$  is regular, then $\QS_K(\Pcal)=\QS_K(s(\Pcal))=\pi_{s(\Pcal)}$.
\end{itemize}

\end{itemize}

\end{prop}

\medskip

It is important to understand what are the admissible orbits $\Pcal$ such that $s(\Pcal)$ is equal to a fixed
regular admissible orbit $\Ocal$.

\begin{defi}\label{defi:h-ancestor}
$\bullet$ For a conjugacy class $(\hgot)\in\Hcal_\kgot$, we denote by $\Acal((\hgot))$ the set of admissible orbits belonging to the Dixmier sheet  $\kgot^*_{(\hgot)}$.

$\bullet$ If $\Pcal,\Ocal\subset \kgot^*$ are $K$-orbits, $\Pcal$ is called a $(\hgot)$-ancestor of $\Ocal$ if
$\Pcal\subset \kgot^*_{(\hgot)}$ and $s(\Pcal)=\Ocal$.
\end{defi}


We make the choice of
a connected Lie subgroup $H$ with Lie algebra $\hgot$ and write $\hgot= \zgot \oplus [\hgot,\hgot]$.
We denote by $\zgot^*_0$ the set of elements $\xi\in \zgot^*$ such that
$K_\xi=H$.
The orbit $o(\hgot)$ (the $\rho$-orbit for $H$)
is contained in $[\hgot,\hgot]^*$.
The orbit $\Pcal$ is a $(\hgot)$-ancestor to $\Ocal$, if and only if there exists $\mu\in \zgot_0^*$
such that $\Pcal=K\mu$ and  $\rho^H\in o(\hgot)$ such that
$\Ocal=K(\mu+\rho^H)$. If $\Ocal$ is admissible then $\Pcal$ is admissible (see \cite{pep-vergne:magic}).

Given a  regular admissible orbit $\Ocal$, there might be several $(\hgot)$-ancestors to $\Ocal$. There might also
be several classes of conjugacy $(\hgot)$ such that $\Ocal$ admits a $(\hgot)$-ancestor $\Pcal$.
For example, let $\Ocal=o(\kgot)$. Then, for any $\hgot\in \Hcal_\kgot$, the orbit $K(\rho^K-\rho^H)$ is a $(\hgot)$-ancestor to $\Ocal$.
Here we have chosen a Cartan subgroup $T$ contained in $H$,
 $H=K_\xi$ and a positive root system  such that $\xi$ is dominant to define $\rho^K$ and $\rho^H$.

\begin{exam}\label{exa:ExampleU3} Consider the group $K=SU(3)$ and let $(\hgot)$ be the centralizer  class of  a subregular element $f\in \kgot^*$ with centralizer $H=S(U(2)\times U(1))$.

We consider the Cartan subalgebra of diagonal matrices and choose a Weyl chamber.
Let $\omega_1,\omega_2$ be the two fundamental weights.
Let $\sigma_1,\sigma_2$ be the half lines $\R_{>0}\omega_1$, $\R_{>0}\omega_2$.
The set $\Acal((\hgot))$ is equal to the collection of orbits $K\cdot (\frac{1+2n}{2}\omega_1), n\in\Z$ (see Figure \ref{admissiblesSU2}).

\begin{figure}[!h]
\begin{center}
  \includegraphics[width=1 in]{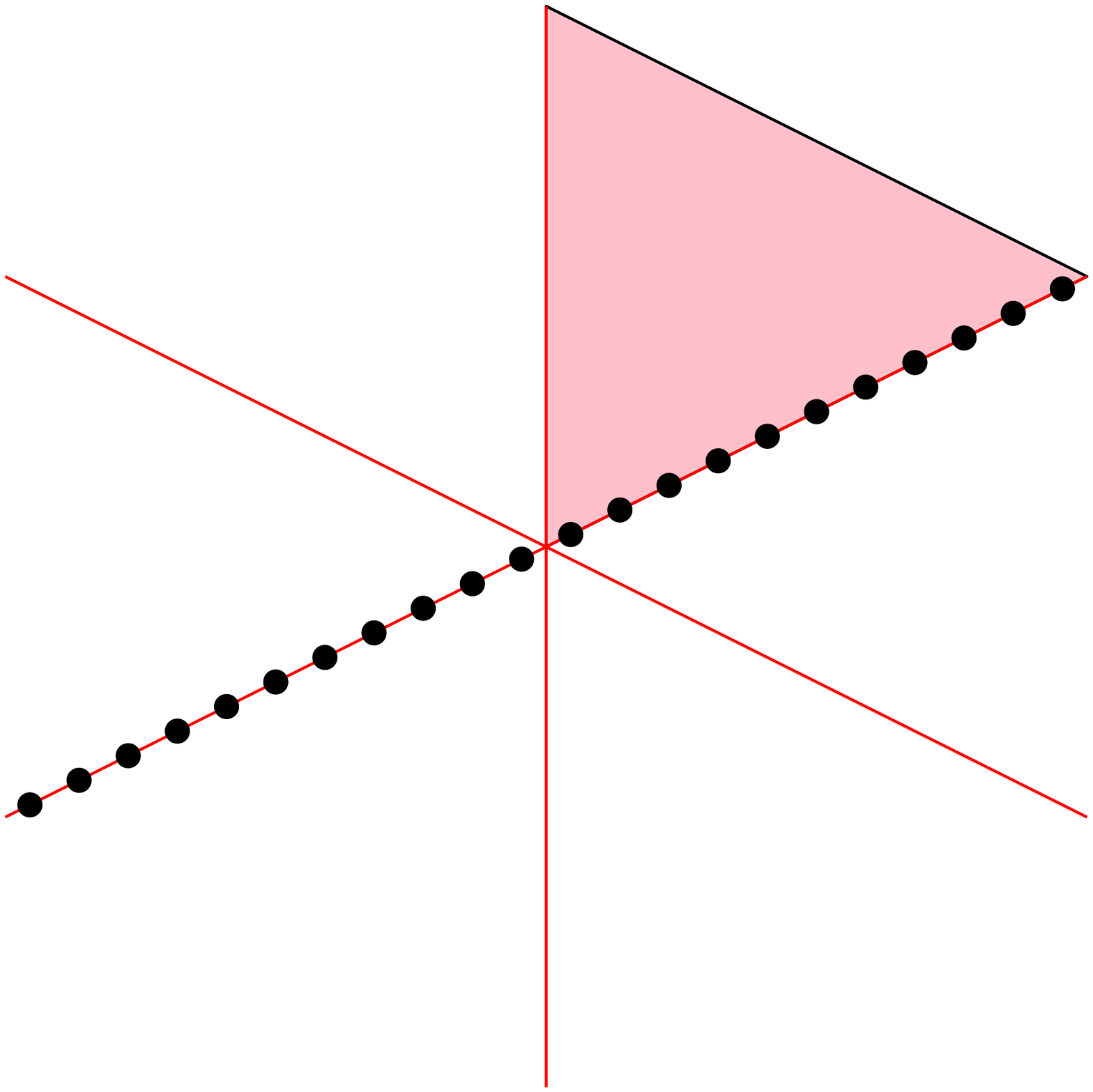}
\caption{ $H$-admissible orbits}
 \label{admissiblesSU2}
\end{center}
\end{figure}

As $-\omega_1$ is conjugated to $\omega_2$, we see
that the set $\Acal((\hgot))$ is equal to the collection of orbits $K\cdot (\frac{1+2n}{2}\omega_i), n\in\Z_{\geq 0}, i=1,2$. Here we have chosen the representatives in the chosen closed Weyl chamber.

One has
 $s(K\cdot (\frac{1+2n}{2}\omega_i))= K(\rho^K+(n-1)\omega_i)$.
Thus the shifted orbit  is a regular orbit if and only if $n>0$.
For $n=1$, both admissible orbits $K\cdot \frac{3}{2}\omega_1$ and $K\cdot (\frac{-3}{2}\omega_1)=K\cdot \frac{3}{2}\omega_2$
are $(\hgot)$-ancestors to the orbit  $K\rho^K=o(\kgot)$.

Both admissible orbits $\Pcal_1=K\cdot \frac{1}{2}\omega_1$ and $\Pcal_2=K\cdot \frac{1}{2}\omega_2$ are such that $\QS_K(\Pcal_i)=0$

In Figure \ref{ancestorall}, we draw the link between $H$-admissible orbits and their respective shifts.

\begin{figure}[!h]
\begin{center}
  \includegraphics[width=2 in]{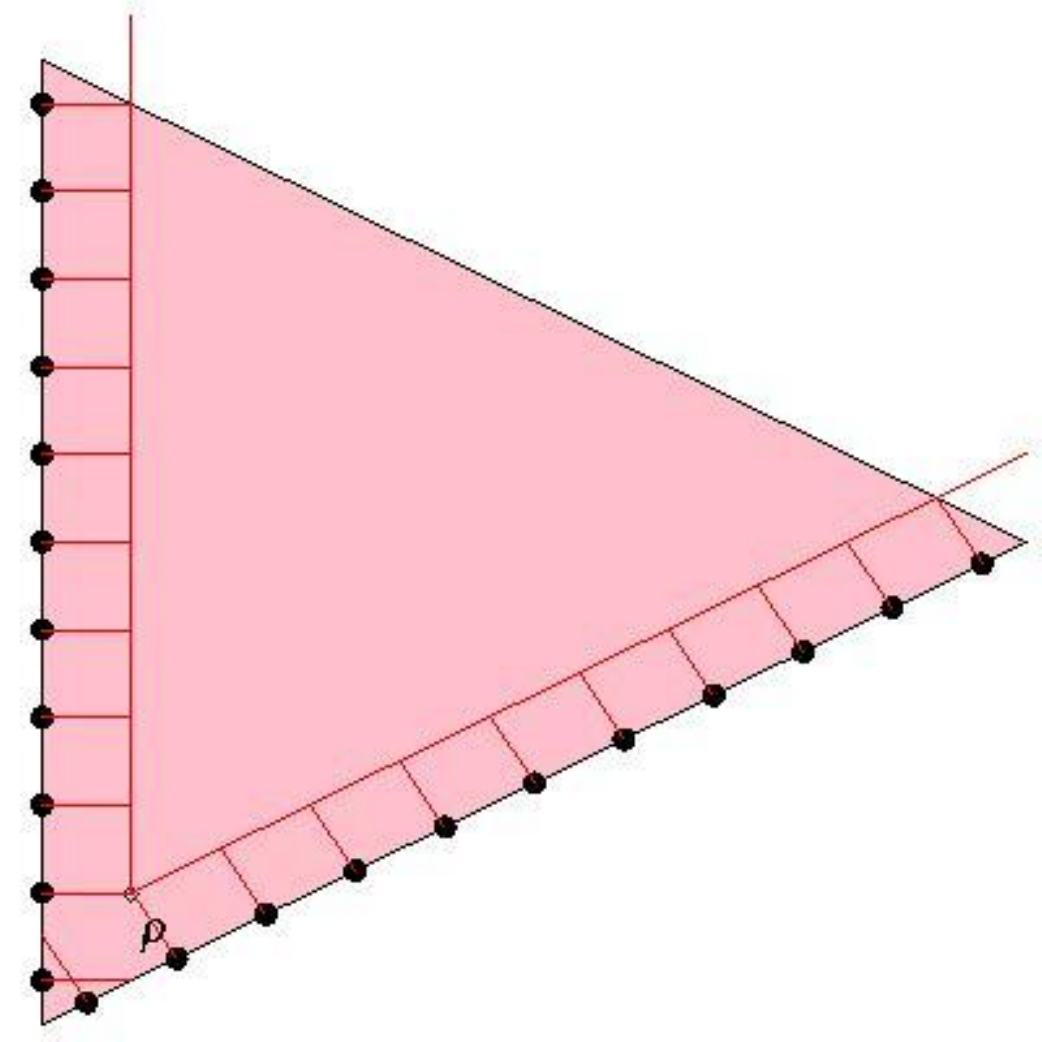}
\caption{ $H$-admissible orbits and their shifts}
 \label{ancestorall}
\end{center}
\end{figure}

\end{exam}

\subsection{Magical inequality}

We often will use complex structures and normalized traces on real vector spaces defined by the following procedure.

\begin{defi}
Let $N$ be a real vector space and $b: N\to N$ a linear transformation, such that $-b^2$ is diagonalizable
with non negative eigenvalues. Define

$\bullet$ the diagonalizable transformation $|b|$ of $N$ by $|b|=\sqrt{-b^2}$,

$\bullet$  the complex structure $J_b= b |b|^{-1}$ on $N/\ker(b)$

$\bullet$ we denote by $\ntr_N|b|=\frac{1}{2}\tr_N|b|$, that is half of the trace of the action of $|b|$ in the real vector space $N$.
We call  $\ntr_N|b|$  the normalized trace of $b$.

\end{defi}

If $N$ has a Hermitian structure invariant by $b$,   $\frac{1}{2}\tr_N|b|$  is the trace of $|b|$ considered as a Hermitian matrix. 
The interest  of our notation is that we do not need complex structures to define $\ntr_N|b|$.

If  $N$ is an Euclidean space and $b$ a skew-symmetric  transformation of $N$,
then $-b^2$ is diagonalizable with  non negative eigenvalues. By definition of $J_b$, the transformation
$b$ of $N$ determines a complex diagonalizable transformation of $N/\ker(b)$, and  the list of its complex
eigenvalues is $[i a_1,\ldots, i a_\ell]$ where the $a_k$ are strictly positive real numbers.
We have   $\ntr_N |b|=\sum_{k=1}^{\ell} a_k\geq 0$.

Recall our identification $\kgot=\kgot^*$ with the help of a scalar product.
When $\beta\in \kgot^*$, denote by $b$ the corresponding element of $\kgot$.
We have defined a complex structure $J_\beta$ on $\kgot/\kgot_\beta$. On the other hand, $b$ defines
an invertible transformation of $\kgot/\kgot_\beta$.
It can be checked  that $J_\beta=J_b$. If we choose a Cartan subalgebra containing $b$,
then
$\ntr_{\kgot}|b|=\sum_{\alpha>0}|\langle \alpha,b\rangle|$.

For further use, we include a lemma.
Let us consider $\kgot_\C$, the complexified space of $\kgot$.
Consider the complex space $\bigwedge \kgot_\C$.

\begin{lem}\label{lem:confusing}
Let $b\in \kgot$. Let $x\in \R$ be an eigenvalue for the action of $\frac{b}{i}$ in  $\bigwedge \kgot_\C$.
Then
$x\geq -\ntr_\kgot |b|$
\end{lem}

\begin{proof}
Indeed, consider a Cartan subalgebra $\tgot$ containing $b$, the system of roots $\Delta$  and  an order such that
$\langle \alpha,b\rangle\geq 0$ for all $\alpha>0$.
An eigenvalue $x$ on $\bigwedge \kgot_\C$ is thus of the form
$\sum_{\alpha\in I\subset \Delta} \langle \alpha,b\rangle$.
Thus we see that the lowest eigenvalue is
$-\sum_{\alpha>0} \langle \alpha,b\rangle=-\ntr_\kgot |b|$.
\end{proof}\bigskip

\medskip

Assume now that $\Ncal\to M$ is a real vector bundle equipped with an action of a compact Lie group
 $K$. For any $b\in\kgot$, and any $m\in M$ such that $b_M(m)=0$, we  may consider
 the linear action $\Lcal(b)\vert_{\Ncal_m}$ which is induced by $b$ on the fibers $\Ncal_m$. It is easy to check that
 $(\Lcal(b)\vert_{\Ncal_m})^2$ is diagonalizable with  eigenvalues which are negative or equal to zero.
 We denote by
$|\Lcal_m(b)|= \sqrt{- (\Lcal(b)\vert_{\Ncal_m})}^2.$

\begin{defi}
We denote by $\ntr_{\Ncal_m}|b|=\frac{1}{2} \tr |\Lcal_m(b)|$ 
that is half of the trace of the real endomorphism $|\Lcal_m(b)|$ on $\Ncal_m$. We call $\ntr_{\Ncal_m}|b|$ 
  the normalized trace of the action of $b$ on $\Ncal_m$.
\end{defi}

For any $b\in\kgot$ and $\mu\in \kgot^*$ fixed by $b$,
we may consider the action $\ad(b) : \kgot_\mu\to \kgot_\mu$ and the corresponding normalized trace
$\ntr_{\kgot_\mu}|\ad(b)|$ denoted simply by
$\ntr_{\kgot_\mu}|b|.$

\begin{defi}
A regular element $\lambda\in \kgot^*$ determines a closed positive Weyl chamber $C_\lambda\subset \kgot_\lambda^*$.
We say that $\lambda$ is very regular if $\lambda\in \rho(\lambda)+C_\lambda$.
 \end{defi}
Notice that regular admissible elements are very regular.

The following ``magical inequality'', that is proved in \cite{pep-vergne:magic}, will be a crucial tool in Section
\ref{sec:witten-deformation}.

\begin{prop}[Magical Inequality]
\label{prop:infernal-with-trace}
Let $b\in \kgot$ and
 denote by $\beta$ the corresponding element in $\kgot^*$.
Let $\lambda,\mu$ be elements of $\kgot^*$ fixed by $b$.
 Assume that $\lambda$ is very regular and that $\mu-\lambda=\beta$.
Then
 $$\|\beta\|^2 \geq \frac{1}{2} \ntr_{\kgot_\mu} |b|.$$
If the equality holds, then
 $\mu$ belongs to  the positive Weyl chamber  $C_\lambda$ and
\begin{enumerate}
\item $\lambda-\rho(\lambda)= \mu- \rho(\mu)$, hence $\lambda$ is admissible if and only if $\mu$ is admissible,
\item $s(K\mu)=K\lambda$.
\end{enumerate}
\end{prop}

\subsection{Slices and induced Spin$^c$ bundles}\label{sec:Slices}

We suppose here that $M$ is a $K$-manifold and that $\Phi:M\to \kgot^*$ is a $K$-equivariant map. If $\Ocal$ is a coadjoint orbit, a neighborhood of $\Phi^{-1}(\Ocal)$ in $M$ can be identified with an induced manifold, and the restriction of $\spinc$-bundles to a neighborhood of  $\Phi^{-1}(\Ocal)$ can be identified with  an induced bundle. To this aim,
let us recall the notion of slice \cite{L-M-T-W}.

\begin{defi}Let $M$ be a $K$-manifold and $m\in M$ with stabilizer subgroup $K_m$. A submanifold $Y\subset M$ containing $m$ is a
slice at $m$ if $Y$ is $K_m$-invariant, $KY$ is a neighborhood of $m$, and the map
$$
K\times_{K_m}Y\longrightarrow M,\ [k,y]\mapsto ky
$$
is an isomorphism  on $KY$.
\end{defi}
Consider the coadjoint action of  $K$ on $\kgot^*$. Let us fix $\mu\in\kgot^*$ and $H:=K_\mu$. Let $C$ be the connected of the open subset
$\hgot^*_0:=\{\xi\in\hgot^*\ \vert\ K_\xi\subset H\}$ containing $\mu$. The map
$K\times_{H}C\to K C$ is a diffeomorphism. Thus $C$ is a slice at $\mu$ for the coadjoint action.

The following  construction was used as a fundamental tool
in the symplectic setting \cite{Guillemin-Sternberg84}.
\begin{prop}\label{prop:slice-general}
Let $\Phi: M\to \kgot^*$ be a $K$-invariant map.
Let $\mu\in\kgot^*$, and let $C$ be the slice at $\mu$ defined previously.
\begin{itemize}
\item  $\Ycal_C=\Phi^{-1}(C)$ is a $K_\mu$-invariant submanifold of $M$ (perhaps empty).
\item  $KC$ is an open neighborhood of $\Phi^{-1}(K\mu)$ diffeomorphic to $K\times_{K_\mu} \Ycal_C$.
\end{itemize}

\end{prop}

The manifold $\Ycal_C$, when is not empty, is called the slice (of $M$) at $\mu\in \kgot^*$. Note
that $\Ycal_C$ can be disconnected.

\medskip

\begin{proof}
Let us consider the $K_\mu$-invariant decompositions
$\kgot=\kgot_\mu\oplus \qgot$,
 $\kgot^*=\kgot^*_\mu\oplus \qgot^*$:
 we denote $\xi\to [\xi]_{\qgot^*}$ the corresponding projection to $\qgot^*$.

A point $\xi$ is in $(\kgot_\mu^*)_o$ if and only if the map
$X\mapsto \xi\circ \ad(X)$ is an isomorphism  from $\qgot$ to $\qgot^*$.
Thus for any $y\in \Ycal_C$, the linear map
$\Pi_y:=[-]_{\qgot^*}\circ \T_y\Phi :\T_yM \to \qgot^*$ is onto. Indeed,
the tangent space to $Ky$ projects onto the tangent space to  $K\Phi(y)$, which contains $[\qgot, \Phi(y)]=\qgot^*$.
Thus we obtain that $\Ycal_C$ is a submanifold  with tangent space $\ker(\Pi_y)$ and
furthermore $\T_y M=\T_y \Ycal_C\oplus \qgot\cdot y $.

The rest of the assertions follow from the fact that $C$ is a slice at $\mu$ for the coadjoint action.
%
%
%
%

\end{proof}

\bigskip

Suppose now that $M$ is oriented and carries a  $K$-equivariant $\spinc$-bundle $\Scal$. Let us explain how this data induces a $\spinc$-bundle on the slice $\Ycal_C$.

Any element $\xi\in\hgot^*_0:=\{\zeta\in\hgot^*\,\vert\, K_\zeta\subset H\}$ determines a complex
structure $J_\xi$ on $\qgot:=\kgot/\hgot$ which depends only of the connected component
$C$ of $\hgot^*_0$ containing $\xi$: thus we denote by $J_C$  the corresponding complex structure on $\qgot:=\kgot/\hgot$. We denote $\qgot^C$  the complex $H$-module $(\qgot,J_C)$, and $\rho_C$ the element of
$\zgot^*$ defined by the relation
\begin{equation}\label{def:rho-C}
\langle\rho_C,X\rangle= \frac{1}{2i}\tr_{\qgot^C}\mathrm{ad}(X), \quad X\in\hgot.
\end{equation}

Consider  the $H$-manifold $\Ycal_C$ and the open subset $K\times_H \Ycal_C$ of $M$.
At the level of tangent spaces we have $\T M\vert_{\Ycal_C}=[\qgot]\oplus\T \Ycal_C$. We orient the manifold $Y$
through the relation $o(M)=o(J_C)o(\Ycal_C)$. The restriction of the $\spinc$-bundle $\Scal$ to the submanifold $\Ycal_C$ allows to define the unique $\spinc$-bundle $\Scal_{\Ycal_C}$ on $\Ycal_C$ such that
\begin{equation}\label{eq;inducedbundle}
\Scal\vert_{\Ycal_C}=\bigwedge\qgot^C\otimes \Scal_{\Ycal_C}.
\end{equation}
This gives a bijection between the
$K$-equivariant $\spinc$-bundles on $K\times_H \Ycal_C$ and the
$H$-equivariant $\spinc$-bundles on $\Ycal_C$.
If  the relation (\ref{eq;inducedbundle}) holds, we say that $\Scal_{\Ycal_C}$ is the  $\spinc$-bundle
induced by $\Scal$. Notice that at the level of determinant line bundles we have
$$
\det(\Scal)\vert_{\Ycal_C}=\det(\Scal_{\Ycal_C})\otimes \C_{2\rho_C}.
$$

\medskip

Let us consider the particular situation where the slice $\Ycal_C$ is a compact submanifold of $M$. It is the case when $M=K\times_H \Ycal_C$, and in this setting we have a simple formula that relate the $\spinc$-indices on $M$ and on the slice $\Ycal_C$:
\begin{equation}\label{eq:induction-index}
\Qcal_K(M,\Scal)=\mathrm{Ind}_H^K\left(\bigwedge \qgot^C\otimes \Qcal_H(\Ycal_C,\Scal_{\Ycal_C})\right).
\end{equation}
See \cite{pep-vergne:witten}.

\section{Computing the multiplicities}\label{sec:computing-multiplicities}

\subsection{Transversally elliptic operators}

In this subsection,  we recall the basic definitions from the theory of transversally
elliptic symbols (or operators) defined by Atiyah and Singer in \cite{Atiyah74}.
We refer to \cite{B-V.inventiones.96.2,pep-vergne:bismut} for more details.

Let $M$ be a compact $K$-manifold with cotangent bundle $\T^*M$. Let $p:\T^*
M\to M$ be the projection.
If $\Ecal$ is a vector bundle on $M$, we may denote still by $\Ecal$ the vector bundle $p^*\Ecal$ on the cotangent bundle $\T^*M$.
If $\Ecal^{+},\Ecal^{-}$ are
$K$-equivariant  complex vector bundles over $M$, a
$K$-equivariant morphism $\sigma \in \Gamma(\T^*
M,\hom(\Ecal^{+},\Ecal^{-}))$ is called a {\em symbol} on $M$.
 For $x\in  M$, and $\nu\in T_x^*M$,  thus $\sigma(x,\nu):\Ecal^{+}_{x}\to
\Ecal^{-}_{x}$
is a linear map from $\Ecal^{+}_{x}$ to
$\Ecal^{-}_{x}$.
The
subset of all $(x,\nu)\in \T^* M$ where the map $\sigma(x,\nu)$  is not invertible is called the {\em characteristic set}
of $\sigma$, and is denoted by $\Char(\sigma)$.
A symbol  is elliptic if its characteristic set is compact.
 An elliptic symbol $\sigma$  on $M$ defines an
element $[\sigma]$ in the equivariant $\K$-theory of $\T^*M$ with compact
support, which is denoted by $\K_{K}^0(\T^* M)$.
The
index of $\sigma$ is a virtual finite dimensional representation of
$K$, that we denote by $\indice_{K}^M(\sigma)\in R(K)$.

Recall the notion of {\it transversally elliptic symbol}.
Let  $\T^*_KM$ be the following $K$-invariant closed subset of $\T^*M$
$$
   \T^*_{K}M\ = \left\{(x,\nu)\in \T^* M,\ \langle \nu ,X\cdot x\rangle=0 \quad {\rm for\ all}\
   X\in\kgot \right\} .
$$
 Its fiber over a point $x\in M$ is  formed by all the cotangent vectors $\nu\in T^*_xM$  which vanish on the tangent space to the orbit of $x$  under $K$, in the point $x$.
A symbol $\sigma$ is  $K$-{\em transversally elliptic} if the
restriction of $\sigma$ to $\T^*_{K}M$ is invertible outside a
compact subset of $\T^*_{K}M$ (i.e. $\Char(\sigma)\cap
\T_{K}^*M$ is compact).

A $K$-{\em transversally elliptic} symbol $\sigma$ defines an
element of $\K_{K}^0(\T^*_{K}M)$, and the index of
$\sigma$
defines an element $\indice_K^M(\sigma)$ of $\hat R(K)$ defined in \cite{Atiyah74}.

We will use the following obvious remark.
Let
$\sigma \in \Gamma(\T^*
M,\hom(\Ecal^{+},\Ecal^{-}))$ be a transversally elliptic  symbol on $M$.
\begin{lem}\label{lem:triv}
Assume an element $b\in K$ acts trivially on $M$, and that $\Ecal^\pm$ are $K$-equivariant vector bundles on $M$ such that the subbundles $[\Ecal^{\pm}]^b$ fixed by $b$ are equal to $\{0\}$.
Then
 $[\indice_K^M(\sigma)]^K=0$
\end{lem}
\begin{proof}
The space $[\indice_K^M(\sigma)]^K$ is constructed as the (virtual) subspace of invariant  $C^{\infty}$-sections of the bundle $\Ecal^{\pm}$ which are solutions of a $K$-invariant pseudo-differential operator on $M$ with symbol $\sigma$.
But, as the action of $b$ is trivial on the basis, and
$[\Ecal^{\pm}]^b=\{0\}$, the space of
$b$-invariant  $C^{\infty}$-sections of the bundle $\Ecal^{\pm}$ is reduced to $0$.
\end{proof}

\medskip

Any elliptic symbol  is $K$-transversally
elliptic, hence we have a restriction map $\K_{K}^0(\T^*
M)\to \K_{K}^0(\T_{K}^*M)$, and a commutative
diagram
\begin{equation}\label{indice.generalise}
\xymatrix{ \K_{K}^0(\T^* M) \ar[r]\ar[d]_{\indice^M_K}
&
\K_{K}^0(\T_{K}^*M)\ar[d]^{\indice^M_K}\\
R(K)\ar[r] & \hat{R}(K)\ .
   }
\end{equation}

\medskip

Using the {\em excision property}, one can easily show that the
index map $\indice_K: \K_{K}^0(\T_{K}^*\Ucal)\to
\hat{R}(K)$ is still defined when $\Ucal$ is a
$K$-invariant relatively compact open subset of a
$K$-manifold (see \cite{pep-RR}[section 3.1]).

In the rest of this article, $M$ will be a Riemannian manifold, and we denote $\nu\in\T^* M\to \tilde{\nu}\in\T M$ the corresponding identification.

\subsection{The Witten deformation}\label{subsec:Wittendeformation}

In this section $M$ is an oriented $K$-manifold of even dimension (not necessarily compact). Let $\Phi: M\to\kgot^*$ be a
$K$-equivariant map. Let $\kappa_\Phi$ be the {\em Kirwan vector field}  associated to
$\Phi$ (see (\ref{eq:K-S})). We denote by $Z_\Phi$ the set of zeroes of $\kappa_\Phi$ : clearly $Z_\Phi$ contains the set of fixed points of the action of $K$ on $M$ as well as $\Phi^{-1}(0)$).

\begin{defi}\label{def:pushed-sigma}
Let $\sigma(M,\Scal)(m,\nu)=\clif_{\Scal_m}(\tilde{\nu}): \Scal_m^+\to \Scal_m^-$ be the symbol of the Dirac operator attached to the $\spinc$-bundle $\Scal$,
and let $\Phi:M\to\kgot^*$ be an equivariant map.  The symbol  $\sigma(M,\Scal,\Phi)$ pushed by the vector field $\kappa_\Phi$ is the
symbol defined by
$$
\sigma(M,\Scal,\Phi)(m,\nu)=\clif_{\Scal_m}(\tilde{\nu}-\kappa_\Phi(m))
: \Scal_m^+\longrightarrow \Scal_m^-
$$
for any $(m,\nu)\in\T M$.

Similarly if $\Wcal \to M$ is a $K$-equivariant vector bundle,
we define
$$
\sigma(M,\Scal\otimes \Wcal,\Phi)(m,\nu)=\sigma(M,\Scal,\Phi)(m,\nu)\otimes {\rm Id}_{\Wcal_m}.
$$

\end{defi}

Note that $\sigma(M,\Scal,\Phi)(m,\nu)$ is invertible except if
$\tilde{\nu}=\kappa_\Phi(m)$. If furthermore $(m,\nu)$ belongs to the subset $\T_K^* M$
of cotangent vectors orthogonal to the $K$-orbits, then $\nu=0$ and
$\kappa_\Phi(m)=0$.  Indeed $\kappa_\Phi(m)$ is tangent to $K\cdot m$ while
$\tilde\nu$ is orthogonal. So we note that $(m,\nu)\in \Char(\sigma(M,\Scal,\Phi_\Scal))\cap \T_K^* M$ if and only if $\nu=0$ and $\kappa_\Phi(m)=0$.

For any $K$-invariant open subset $\Ucal\subset M$ such that $\Ucal\cap Z_\Phi$ is compact in $M$, we see that the restriction
$\sigma(M,\Scal,\Phi)\vert_\Ucal$ is a transversally elliptic symbol on $\Ucal$, and so its equivariant index is a well defined element in
$\hat{R}(K)$.

Thus we can define the following localized equivariant indices.

\begin{defi}\label{def:indice-localise}
\begin{itemize}
\item A closed invariant subset $Z\subset Z_\Phi$ is called a {\em component} if it is a union of connected components of $Z_\Phi$.

\item  If $Z\subset Z_\Phi$ is a {\em compact component}, and $\Wcal$ is a $K$-equivariant
vector bundle over $M$, we denote by
$$
\Qcal_K(M,\Scal\otimes \Wcal,Z,\Phi)\ \in\ \hat{R}(K)
$$
the equivariant index of $\sigma(M,\Scal\otimes \Wcal,\Phi)\vert_\Ucal$ where $\Ucal$ is an invariant neighborhood of $Z$
so that $\Ucal\cap Z_\Phi=Z$.
\item If we make the Witten deformation with the map $\Phi=\Phi_\Scal$, the term $\Qcal_K(M,\Scal\otimes \Wcal,Z,\Phi_\Scal)$ is denoted simply by
$\Qcal_K(M,\Scal\otimes \Wcal,Z)$.
\end{itemize}
\end{defi}

By definition, $Z=\emptyset$ is a component and $\Qcal_K(M,\Scal\otimes\Wcal,\emptyset,\Phi)=0$.

When $M$ is compact it is clear that the classes of the symbols
$\sigma(M,\Scal,\Phi)$ and $\sigma(M,\Scal)$ are equal in
$\K_{K}^0(\T_{K}^*M)$, thus we get the first form of the localization theorem.

\begin{theo}\label{theo:indexsumoflocal}
Assume that $M$ is compact. If $Z_\Phi=Z_1\coprod\ldots\coprod Z_p$ is a decomposition
into disjoint (compact) components, we have the following equality in $\hat{R}(K)$ :
$$
\Qcal_K(M,\Scal)=\sum_{i=1}^p \Qcal_K(M,\Scal, Z_i,\Phi)
$$
\end{theo}

\begin{rem}\label{rem: loca}
Write
$\Phi_\Scal(Z_\Scal)=\coprod_j \Ocal_j$ as  a  disjoint union of a finite set of coadjoint orbits.
Then we obtain the decomposition
$$\Qcal_K(M,\Scal)=\sum_{j} \Qcal_{\Ocal_j}$$ with
$\Qcal_{\Ocal}=\Qcal_K(M,\Scal, \Phi_\Scal^{-1}(\Ocal)\cap Z_\Scal)$.
As in \cite{pep-RR}, this decomposition is the main tool of our study.
However, in this work, we will need to introduce a further refinement of this decomposition.
\end{rem}

\begin{exam}
We return to our basic example (Example \ref{P1}).
Let $p_+=[1,0]$ and $p_-=[0,1]$  be the fixed points of the $T$-action on $M=\Pbb^1(\C)$.
The determinant line bundle of $\Scal(n)$ is  $\Lbb_n=[\C_{-1}]\otimes \Lcal^{\otimes 2n+2}$
where $[\C_{-1}]$ is the trivial line bundle equipped with the representation $t^{-1}$ on $\C$. We choose  the moment map
$\Phi_n$ associated to a connection on the determinant bundle (see more details in Section \ref{sec:examples}): $$\Phi_n([z_1,z_2])=(n+1)\frac{|z_1|^2}{|z_1|^2+|z_2|^2}-\frac{1}{2}.$$
Then,  for $n\geq 0$, $\Zcal=\{p_+\}\cup \{p_-\}\cup \Phi_{n}^{-1}(0)$, thus
$\Phi_n(Z_\Scal)=\{-\frac{1}{2}\}\cup \{0\} \cup \{n+\frac{1}{2}\}$.
Remark that $Z_\Scal$ is smooth: it has $3$ connected components, the two fixed points, and  $\Phi_{n}^{-1}(0)$  a circle with free action of $T$.
Then we obtain the associated decomposition
$\Qcal_T(M,\Scal(n))=Q_{-\frac{1}{2}}+Q_0+Q_{\frac{1}{2}}$ with
$$
Q_{-\frac{1}{2}}=-\sum_{k=-1}^{-\infty} t^k,\quad Q_{0}=\sum_{k=-\infty}^{-\infty} t^k,\quad
Q_{\frac{1}{2}}=-\sum_{k=n+1}^{\infty} t^k.
$$
\end{exam}

\begin{exam}
 Take the product $N=\Pbb^1(\C)\times \Pbb^1(\C)$,  with spin bundle $\Scal=\Scal(0)\otimes \Scal(0)$, moment map $\Phi_0$ and we consider  the diagonal action of $T$ with moment map $\Phi(m_1,m_2)=\Phi_0(m_1)+\Phi_0(m_2)$.
As $\Qcal_T(\Pbb^1(\C),\Scal(0))$ is the trivial representation of $T$,
  $\Qcal_T(N,\Scal)$ is still the trivial representation of $T$.

We have $\Phi(Z_\Scal)=\{-1\}\cup \{0\} \cup \{1\}$. In this case $\Phi^{-1}(\pm 1)=\{(p_\pm,p_\pm)\}$, and $\Phi^{-1}(0)$  is not smooth.

Consider the associated decomposition  of $\Qcal_T(N,\Scal)=Q_{-1}+Q_0+Q_{1}$.
We have
$$
Q_{-1}=\sum_{k=-\infty}^{-2} (-k-1)t^k,\quad Q_{0}=\sum_{k=-\infty}^{-\infty} (|k|-1) t^k,\quad
Q_{1}=\sum_{k=2}^{\infty}(k-1) t^k.
$$
We see that indeed $Q_{-1}+Q_{0}+Q_{1}=t^0$. Figure  \ref{Q0} shows the corresponding multiplicity functions.

\begin{figure}[!h]
\begin{center}
 \includegraphics[width=2 in]{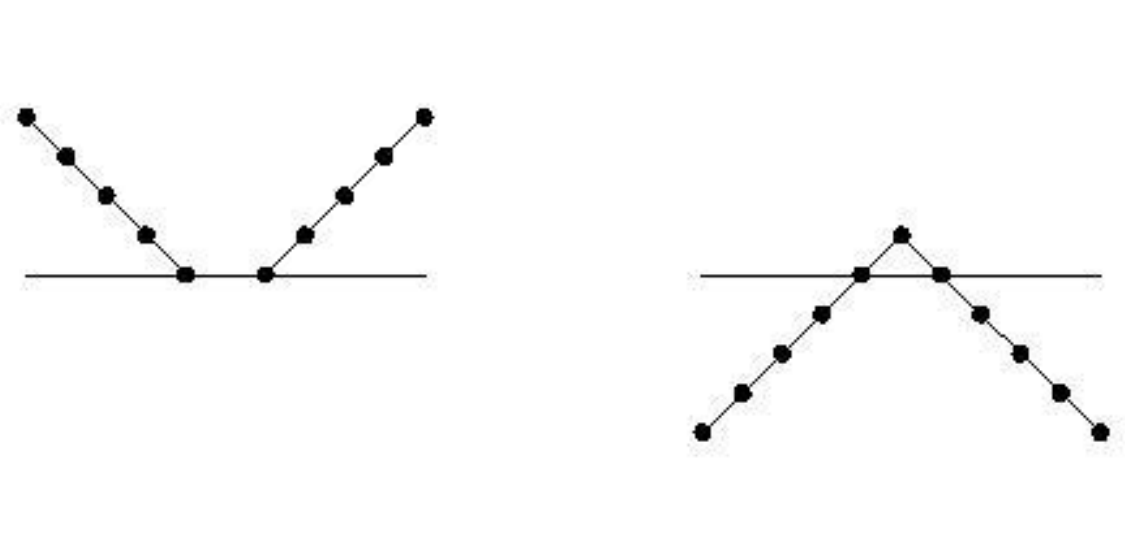}
 \end{center}
\caption{ The graph of  $Q_{-1}+Q_1$ and  the graph of $Q_0$}
 \label{Q0}
\end{figure}

\end{exam}

\subsection{Some properties of the localized index}

In this subsection, we recall the properties
of the localized index $\Qcal_K(M,\Scal, Z, \Phi)$
that we will use in this article.

\subsubsection{Fixed point submanifolds and $\spinc$-bundles}

Let $\Scal$ be a $K$-equivariant $\spinc$-bundle over the tangent bundle $\T M$ of a $K$-manifold $M$ (equipped with an invariant Riemannian metric).
The manifold $M$ is oriented and the Clifford bundle $\Scal$ is equipped with its canonical $\Z/2\Z$-grading.
Let $b\in\kgot$ be a non-zero $K$-invariant element, and consider the submanifold $M^{b}$ where
the vector field $b_M$ vanishes. We have an orthogonal decomposition
$$
\T M\vert_{M^b}=\Ncal\oplus \T M^b.
$$

The normal bundle $\Ncal$ inherits a fibrewise linear endomorphism $\Lcal(b)$ which is anti-symmetric relatively to the metric.
\begin{defi}\label{def:J-beta}
$\bullet$ We denote by $\Ncal_b$ the vector bundle $\Ncal$ over $M^b$  equipped with the complex structure
$J_b:=\Lcal(b) |\Lcal(b)|^{-1}$.

$\bullet$ We take on $\Ncal$ the orientation $o(\Ncal)$ induced by the complex structure $-J_b$.  On $M^b$  we take the orientation $o(M^b)$  defined by
$o(\Ncal)o(M^b)=o(M)$.
\end{defi}

Note that the endomorphism $\Lcal(b): \Ncal_b\to \Ncal_b$ is $\C$-linear, diagonalizable,
with eigenvalues  $i \theta^1_\Xcal,\ldots, i\theta^p_\Xcal$ that
depends of the connected component $\Xcal$ of $M^b$.
For further use, we note the following positivity result which follows directly from the definition of $J_b$.
\begin{lem}\label{lem:betapositive}
The eigenvalues  of the action of
$\frac{1}{i}\Lcal(b)$ on $\Ncal_b$ are positive.
\end{lem}

 If we consider the complex line
bundle $\det(\Ncal_b)\to M^b$, we see that $\frac{1}{i}\Lcal(b)$ acts on the fibers of
$\det(\Ncal_b)\vert_\Xcal$ by multiplication by the positive number
$$
\ntr_{\Ncal_b\vert_\Xcal}|b|= \sum_{j=1}^p\theta^j_\Xcal.
$$

\begin{prop} \label{prop:spin-induit}
Let $\det(\Scal)$ be the determinant line bundle of the spin$^c$ bundle  $\Scal$.
 There exists an equivariant $\spinc$-bundle $\mathbf{d}_b(\Scal)$ on the tangent bundle $\T M^b$  with
 determinant line bundle equal to
\begin{equation}\label{eq:L_beta}
\det(\mathbf{d}_b(\Scal)):=\det(\Scal)\vert_{M^b}\otimes \det(\Ncal_b).
\end{equation}

\end{prop}

\begin{proof} The restriction $\Scal\vert_{M^b}$ is a $\spinc$-bundle over the tangent bundle $\T M\vert_{M^b}=\Ncal\oplus \T M^b$. We denote
$\overline{\Ncal_b}$ the vector bundle $\Ncal$  with the complex structure $-J_b$. Let
$\bigwedge \overline{\Ncal_b}$ be the spinor bundle on $\Ncal$ with its canonical grading : since $o(\Ncal)=o(-J_b)$ we have
$(\bigwedge \overline{\Ncal_b})^\pm=\bigwedge^{\pm}\overline{\Ncal_b}$.

Since $\bigwedge \overline{\Ncal_b}$ is a graded $\spinc$-bundle over $\Ncal$, we know that there exists an equivariant $\spinc$
bundle $\mathbf{d}_b(\Scal)$ over the tangent bundle $\T M^b$ (with its canonical grading) such that
\begin{equation}\label{eq:Sb}
\Scal\vert_{M^b}= \bigwedge \overline{\Ncal_b}\otimes \mathbf{d}_b(\Scal).
\end{equation}
is an isomorphism of graded Clifford modules. At the level of determinant line bundle,  we get $\det(\Scal)\vert_{M^b}=\det(\overline{\Ncal_b})\otimes \det(\mathbf{d}_b(\Scal))$.
The identity (\ref{eq:L_beta}) then follows.
\end{proof}\bigskip

Consider the linear action $\Lcal(b)\vert_{\Scal_b}$ of $b$
on the fibers of the $\spinc$-bundle $\Scal_b\to M^b$.
\begin{lem}\label{lem:L-beta-S}
We have $\frac{1}{i}\Lcal(b)\vert_{\mathbf{d}_b(\Scal)} = a \, \mathrm{Id}_{\Scal_b}$ where
$$
a(m)= \langle\Phi_\Scal(m),b\rangle +\frac{1}{2} \ntr_{\T_m M}|b|
$$
is a locally constant function on $M^b$.
\end{lem}
\begin{proof}
Thanks to Remark \ref{rem:L-S-gamma}, we know that $a(m)$ is  equal to $\langle\Phi_{b}(m),b\rangle$
where $\Phi_{b}$ is a  moment map attached to the line bundle $\det(\mathbf{d}_b(\Scal))$.
Thanks to (\ref{eq:L_beta}) we see that $\langle\Phi_{b}(m),b\rangle=\langle\Phi_{\Scal}(m),b\rangle +
\frac{1}{2}\tr_{\Ncal_b}|b|$.
But $\ntr_{\T M}|b|=\tr_{\Ncal_b}|b|$ as well as and
$\langle\Phi_{\Scal}(m),b\rangle$ are locally constant on $M^b$.

\end{proof}

\medskip

The localization formula of Atiyah-Segal can be expressed in the following way (see \cite{pep-vergne:witten}):

\begin{theo}\label{theo:Atiyah-Segal}
Let $b\in\kgot$ be a non-zero $K$-invariant element and assume that $M$ is compact. For any complex $K$-vector bundle $\Wcal\to M$, we have the following equalities in $\hat{R}(K)$ :
$$
\Qcal_K(M,\Scal\otimes \Wcal)=\Qcal_{K}\left(M^b,\mathbf{d}_b(\Scal) \otimes \Wcal\vert_{M^b}
\otimes \Sym (\Ncal_b)\right).
$$
Here $\Sym (\Ncal_b)$ is the symmetric algebra of the complex vector bundle $\Ncal_b$.
\end{theo}

\subsubsection{The localization formula over a coadjoint orbit}
Let $\Phi: M\to \kgot^*$ be an equivariant map.
Let $\beta\in \kgot^*$. We also consider $\beta$ as an element of $\kgot$ that we denote by the same symbol. In this section we assume
that $Z_\beta=K(M^\beta\cap \Phi^{-1}(\beta))$ is a compact component of $Z_\Phi\subset M$. The study of
$\Qcal_K(M,\Scal\otimes \Wcal,Z_\beta,\Phi)\ \in \  \hat{R}(K)$
is thus localized in a neighborhood of $\Phi^{-1}(K\beta)$, an induced manifold. Let us recall the corresponding  induction formula.

The restriction of $\Phi$ to $M^\beta$ is a $K_\beta$-equivariant map taking value in $\kgot^*_\beta$. The subset $Z'_{\beta}=M^\beta\cap \Phi^{-1}(\beta)$ is a compact component of $Z_{\Phi\vert_{M^\beta}}=Z_\Phi\cap M^\beta$. We may then define the localized index
$$
\Qcal_{K_\beta}(M^\beta,\mathbf{d}_\beta(\Scal)\otimes \Wcal\vert_{M^\beta},Z'_\beta,\Phi\vert_{M^\beta})\ \in \  \hat{R}(K_\beta)
$$
where $\mathbf{d}_\beta(\Scal)$ is the graded $\spinc$-bundle on $M^\beta$ defined in   Proposition \ref{prop:spin-induit}.

We consider the normal bundle $\Ncal\to M^\beta$ of $M^\beta$ in $M$. Recall that $\Ncal_\beta$ denotes the vector
bundle $\Ncal$ equipped with the complex $J_\beta$.
The following formula is proved in \cite{pep-RR,pep-vergne:witten}:
\begin{eqnarray*}
\lefteqn{\Qcal_K(M,\Scal\otimes \Wcal, Z_{\beta},\Phi)}\\
&=&\mathrm{Ind}_{K_\beta}^K\left(\Qcal_{K_\beta}(M^\beta,\mathbf{d}_\beta(\Scal)\otimes\Wcal\vert_{M^\beta}\otimes\Sym(\Ncal_\beta), Z'_{\beta},\Phi\vert_{M^\beta})\otimes
\bigwedge (\kgot/\kgot_\beta)_\C\right).
\end{eqnarray*}

\begin{rem}\label{rem:Atiyahsegalabelian}
When $K$ is abelian, this gives
\begin{eqnarray*}
\lefteqn{\Qcal_K(M,\Scal\otimes \Wcal, \Phi^{-1}(\beta)\cap M^\beta,\Phi)}\\
&=&\Qcal_{K}(M^\beta,\mathbf{d}_\beta(\Scal)\otimes\Wcal\vert_{M^\beta}\otimes\Sym(\Ncal_\beta),\Phi^{-1}(\beta)\cap M^{\beta},\Phi\vert_{M^\beta})
\end{eqnarray*}
which shows that the Atiyah-Segal localization formula  (\ref{theo:Atiyah-Segal})  still holds for the Witten deformation.
\end{rem}

Thus we obtain the following proposition.

\begin{prop}
Let $\Scal$ be a $K$-equivariant $\spinc$-bundle over $M$, with its canonical grading. Let $\Phi:M\to \kgot^*$ be an equivariant map.
Let $\Wcal\to M$ be an equivariant complex vector bundle. Assume that $Z_\beta=K(M^\beta\cap \Phi^{-1}(\beta))$ is a compact component
of $Z_\Phi\subset M$. Then
\begin{eqnarray}\label{eq:loc-induction-invariant}
   \lefteqn{\left[\Qcal_K(M,\Scal\otimes \Wcal, Z_\beta,\Phi)\right]^K= }\nonumber\\
   & & \left[\Qcal_{K_\beta}(M^{\beta},\mathbf{d}_\beta(\Scal)\otimes\Wcal\vert_{M^\beta}\otimes\Sym(\Ncal_\beta), Z'_\beta,\Phi\vert_{M^\beta})\otimes
\bigwedge (\kgot/\kgot_\beta)_\C\right]^{K_\beta}.
\end{eqnarray}
\end{prop}

This proposition will be used   to obtain vanishing results, by studying  the infinitesimal action of $\beta$ on the vector bundle
$\mathbf{d}_\beta(\Scal)\otimes\Wcal\vert_{M^\beta}\otimes\Sym(\Ncal_\beta)$.

\medskip

The formula (\ref{eq:loc-induction-invariant}) can be specialized to each connected component of $M^\beta$.
For  a connected component
$\Xcal\subset M^\beta$ intersecting $\Phi^{-1}(\beta)$, we define the compact subset
$$
Z_\beta(\Xcal)=K\left(\Xcal\cap\Phi^{-1}(\beta)\right)\subset Z_\beta.
$$
First we note that $\Qcal_K(M,\Scal\otimes \Wcal, Z_\beta,\Phi)$ is equal to the sum $\sum_{\Xcal}\Qcal_K(M,\Scal\otimes \Wcal, Z_\beta(\Xcal),\Phi)$ parameterized by the connected component of $M^\beta$ intersecting $\Phi^{-1}(\beta)$ (their are finite in number).

We have a localization formula for each term $\Qcal_K(M,\Scal\otimes \Wcal, Z_\beta(\Xcal),\Phi)$ separately (see \cite{pep-RR,pep-vergne:witten}) :
\begin{eqnarray}\label{eq:loc-induction-invariant-local}
   \lefteqn{\left[\Qcal_K(M,\Scal\otimes \Wcal, Z_\beta(\Xcal),\Phi)\right]^K= }\nonumber\\
   & & \left[\Qcal_{K_\beta}(\Xcal,\mathbf{d}_\beta(\Scal)\vert_{\Xcal}\otimes\Wcal\vert_{\Xcal}\otimes\Sym(\Ncal_\beta)\vert_{\Xcal}, Z'_\beta(\Xcal),\Phi\vert_{\Xcal})\otimes
\bigwedge (\kgot/\kgot_\beta)_\C\right]^{K_\beta}
\end{eqnarray}
where $Z'_\beta(\Xcal)=\Xcal\cap\Phi^{-1}(\beta)\subset Z'_\beta$.

\subsubsection{Induction formula}
For the Witten deformation, we proved in \cite{pep-vergne:witten} the following variation on the invariance of the index under direct images.

Let $H$ be a closed subgroup of $K$, and consider
a $H$-invariant decomposition
$$\kgot=\hgot\oplus \qgot.$$

Let $B_\qgot$ be an open ball in $\qgot$, centered at $0$ and $H$-invariant.
Let $N'$ be a $H$-manifold, and
consider $N=K\times_H (B_\qgot \times N').$
Then $N'$ is a submanifold of $M$, and the normal bundle of $N'$ in $N$ is isomorphic to the trivial bundle with fiber $\qgot\oplus \qgot$.
Let $S_\qgot$ be the $\spinc$ module for $\qgot\oplus \qgot$ (we can take $\bigwedge \qgot_\C$ as realization of $S_\qgot$).
Thus if $\Ecal$ is a $K$-equivariant graded Clifford bundle on $N$, there exists a $H$-equivariant graded Clifford bundle $\Ecal'$ on $N'$ such that
$$\Ecal|_{N'}=S_\qgot \otimes \Ecal'.$$

Let $\Phi': N'\to \hgot^*$ be a $H$-equivariant map, and let $\Phi: N\to \kgot^*$ be a $K$-equivariant map. We assume that these maps are linked
by the following relations  :
\begin{equation}\label{eq-phi-prime}
\begin{cases}
\ \       \Phi\vert_{N'}=\Phi',     \\
 \ \    \Phi([1;X,n'])\in \hgot^*\Longleftrightarrow X=0,\\
 \ \    (\Phi([1;X,n']),X)\geq 0,
\end{cases}
\end{equation}
for $(X,n')\in B_\qgot \times N'$.

Under these conditions, we see that the critical sets $Z_\Phi\subset N$ and $Z_{\Phi'}\subset N'$ are related by :
$Z_\Phi=K\times_H (\{0\}\times Z_{\Phi'})$.

\begin{prop}[\cite{pep-vergne:witten}]\label{induction}
Let $Z$ be a compact component of $Z_\Phi$ and $Z'$ its intersection with $N'$. Then $Z'$ is a compact component of $Z_{\Phi'}$ and
$$\Qcal_K(N,\Ecal,Z,\Phi)={\rm Ind}_H^K \left(\Qcal_H(N',\Ecal',Z',\Phi')\right).$$
This leads to the relation $\left[\Qcal_K(N,\Ecal,Z, \Phi)\right]^K=\left[\Qcal_H(N',\Ecal', Z',\Phi')\right]^H$.
\end{prop}

\subsection{The function $d_\Scal$}\label{sec:fonction-d-S}

Let $M$ be a {\em compact} oriented even dimensional $K$-manifold, equipped with a $K$-equivariant $\spinc$
bundle $\Scal$. Let $\Phi_\Scal$ be the associated  moment map on $M$, and $\kappa_\Scal$ be the Kirwan vector field.
Let $Z_\Scal$ be the vanishing set of $\kappa_\Scal$~:
$$
Z_\Scal=\left\{m\in M\ | \ \Phi_\Scal(m)\cdot m=0\right\}=\bigcup_{\theta} M^\theta\cap \Phi^{-1}_\Scal(\theta).
$$
We now introduce a function  $d_\Scal : Z_\Scal \longrightarrow \R$ which will localize our study of
$\left[\Qcal_K(M,\Scal, Z_\Scal)\right]^K$  to
special components $Z$ of $Z_\Scal$.

\medskip

Define  $d_\Scal : Z_\Scal \longrightarrow \R$ by the following relation
\begin{equation}\label{eq:d-S}
d_\Scal(m)=\|\theta\|^2+\frac{1}{2}\ntr_{\T_m M}|\theta| -\ntr_{\kgot}|\theta|,\quad \mathrm{with}\quad  \theta=\Phi_\Scal(m).
\end{equation}

\begin{lem}
\begin{itemize}
\item The function $d_\Scal$ is a $K$-invariant locally constant function on $Z_\Scal$ that takes a finite number of values.
\item The subsets $Z_\Scal^{>0}=\{d_\Scal>0\}$, $Z_\Scal^{=0}=\{ d_\Scal=0\}$,
$Z_\Scal^{<0}=\{d_\Scal<0\}$ are components of $Z_\Scal$.
\end{itemize}
\end{lem}
\begin{proof} The $K$-invariance of $d_\Scal$ is immediate.

The image $\Phi_\Scal(Z_\Scal)$ is equal to a finite union $\bigcup_j \Ocal_j$ of coadjoint orbits. For each
coadjoint orbit $\Ocal= K\beta$, the set $Z_\Scal\cap\Phi_\Scal^{-1}(\Ocal)$ is equal to a finite disjoint union
$\bigcup_j K( \Xcal^j\cap\Phi^{-1}_\Scal(\beta))$ where
$(\Xcal^j)$ are the connected components of $M^{\beta}$ intersecting $\Phi^{-1}_\Scal(\beta)$.
Since $m\mapsto \ntr_{\T_m M}|\theta|$ is well defined and locally constant on $M^\theta$,
the map $d_\Scal$ is constant on each
component $K( \Xcal^j\cap\Phi^{-1}_\Scal(\beta))$. This proves that $d_\Scal$ is locally constant function that takes a finite number of values.

The second point is a direct consequence of the first.
\end{proof}

\medskip

We now prove the following fundamental fact.

\begin{prop}\label{prop:annulation-d-S}
Let $Z_\Scal^{>0}$ be the component of $Z_\Scal$ where $d_\Scal$ takes strictly positive values. We have
$\left[\Qcal_K(M,\Scal, Z_\Scal^{>0})\right]^K=0$.
\end{prop}

\medskip

Since $\Qcal_K(M,\Scal)=\Qcal_K(M,\Scal, Z_\Scal^{<0})+\Qcal_K(M,\Scal, Z_\Scal^{=0})+\Qcal_K(M,\Scal, Z_\Scal^{>0})$ by Theorem \ref{theo:indexsumoflocal}, note first the following immediate corollary.
\begin{coro}\label{coro:annulation-d-S}
If $d_\Scal$ takes non negative  values on $Z_\Scal$, we have
$$
[\Qcal_K(M,\Scal)]^K=[\Qcal_K(M,\Scal, Z_\Scal^{=0})]^K.
$$
\end{coro}

\medskip

We now prove Proposition \ref{prop:annulation-d-S}.

\begin{proof} Consider  a coadjoint orbit $K\beta$ contained in $\Phi_\Scal(Z_\Scal)$. Let $\Xcal$ be the connected component of $M^{\beta}$ and let
$Z'_\beta(\Xcal):=\Xcal \cap \Phi^{-1}(\beta)$. Let $Z_\beta(\Xcal)=KZ'_\beta(\Xcal)$. Let us show that $[\Qcal_K(M,\Scal, Z_\beta(\Xcal))]^K=0$ if $d_\Scal$ is strictly positive on $Z_\beta(\Xcal)$.

As $\left[\Qcal_K(M,\Scal, Z_\beta(\Xcal))\right]^K$ is equal to
\begin{equation}\label{eq:loc-Z-beta-X-S}
\left[\Qcal_{K_\beta}(\Xcal,\mathbf{d}_\beta(\Scal)\vert_{\Xcal}\otimes\Sym(\Ncal_\beta)\vert_{\Xcal}, Z'_\beta(\Xcal),\Phi_\Scal\vert_\Xcal)\otimes
\bigwedge (\kgot/\kgot_\beta)_\C\right]^{K_\beta}
\end{equation}
by the localization formula (\ref{eq:loc-induction-invariant-local}),   it is sufficient to prove that
the infinitesimal action  $\Lcal(\beta)$ on the fibers of the vector bundles
$\mathbf{d}_\beta(\Scal)\vert_{\Xcal}\otimes\mathrm{Sym}^j(\Ncal_\beta)\vert_{\Xcal}\otimes \bigwedge(\kgot/\kgot_\beta)_\C$
have only strictly positive eigenvalues. We establish this  by minorizing the possible eigenvalues : they  are sums of eigenvalues on each factor of the tensor product.

We have
$$
\frac{1}{i}\Lcal(\beta)=
\begin{cases}
\|\beta\|^2+ \frac{1}{2}\ntr_{\T M\vert_\Xcal}|\beta|\hspace{9mm} {\rm on}\ \ \mathbf{d}_\beta(\Scal)\vert_{\Xcal},\\
\geq 0\hspace{40mm} {\rm on}\ \ \mathrm{Sym}^j(\Ncal_\beta)\vert_{\Xcal},\\
\geq - \ntr_{\kgot}|\beta|\hspace{24mm} {\rm on}\ \ \bigwedge(\kgot/\kgot_\beta)_\C.
\end{cases}
$$

In the first equality, we have used
Lemma \ref{lem:L-beta-S}:
the function $m\mapsto \langle\Phi_\Scal(m),\beta\rangle$ is constant on $\Xcal$, and as $\Xcal$
contains a point projecting on $\beta$,   $\frac{1}{i}\Lcal(\beta)\vert_{\mathbf{d}_\beta(\Scal)|_\Xcal} =(\|\beta\|^2+ \frac{1}{2}\ntr_{\T M\vert_\Xcal}|\beta|)  \,
\mathrm{Id}_{\mathbf{d}_\beta(\Scal)\vert_\Xcal}.$

In the second inequality, we used Lemma \ref{lem:betapositive}, so that the action of
$\frac{1}{i}\Lcal(\beta)$ on the graded piece
$\Sym^j(\Ncal_\beta)$ is strictly positive for $j>0$ or equal to $0$ for $j=0$.

In the last inequality, we have used Lemma \ref{lem:confusing}.

If  $d_\Scal$ takes a strictly positive value on $Z_\beta(\Xcal)$, we see that $\frac{1}{i}\Lcal(\beta)>0$ on
$\mathbf{d}_\beta(\Scal)\vert_{\Xcal}\otimes \mathrm{Sym}^j(\Ncal_\beta)\vert_{\Xcal}\otimes \bigwedge(\kgot/\kgot_\beta)_\C$ : this forces (\ref{eq:loc-Z-beta-X-S}) to be equal to zero.


\end{proof}

\medskip

\subsection{The Witten deformation on the product $M\times \Ocal^*$}\label{sec:witten-deformation}

In this section, $M$ is  a {\em compact} oriented even dimensional $K$-manifold, equipped with a $K$-equivariant $\spinc$
bundle $\Scal$. Let $\Phi_\Scal$ be the associated  moment map on $M$. Our aim is to compute geometrically the multiplicities
of the equivariant index $\Qcal_K(M,\Scal)$.

\subsubsection{Vanishing theorems}\label{sec:vanishing}

Let $\mathcal{H}_{\mathfrak k}$ be the set of conjugacy classes of the reductive algebras ${\kgot}_\xi,\xi\in\mathfrak k^*$.
We denote by $\mathcal S_{\mathfrak k}$ the set of conjugacy classes of the semi-simple parts $[\hgot,\hgot]$ of
the elements $(\hgot)\in \mathcal H_{\mathfrak k}$.

Recall that an  orbit $\Pcal$ is a $(\hgot)$-ancestor of $\Ocal$ if $\Pcal$ belongs to the Dixmier sheet  $\kgot^*_{(\hgot)}$ and
$s(\Pcal)=\Ocal$. Here $s(\Pcal)$ is defined as follows~: if $\Pcal=K\mu$ with $\kgot_\mu=\hgot$, then  $s(\Pcal)=K(\mu +o(\hgot))$
(see Definition \ref{defi-shift}).

\medskip

Recall that the map $\Ocal\mapsto \pi_\Ocal:=\QS_K(\Ocal)$ is a bijection between the regular admissible orbits and $\wK$. If $\Ocal$ is a regular admissible orbit, then $\Ocal^*:=-\Ocal$ is also admissible and $\pi_{\Ocal^*}=(\pi_\Ocal)^*$. If we apply the shifting trick, we see that the multiplicity of $\pi_\Ocal$ in $\Qcal_K(M,\Scal)$
 is equal to
\begin{eqnarray}\label{eq:shifting-trick}
\mm_\Ocal&=& \left[\Qcal_K(M,\Scal)\otimes (\pi_\Ocal)^*\right]^K \nonumber\\
&=& \left[\Qcal_K(M\times \Ocal^*,\Scal\otimes\Scal_{\Ocal^*})\right]^K.
\end{eqnarray}

Let $(\kgot_M)$ be  the generic infinitesimal stabilizer of the $K$-action on $M$. In this section,
we prove the following vanishing results.

\begin{theo}\label{theo:vanishing-1}
\begin{itemize}
\item If  $([\kgot_M,\kgot_M])\neq ([\hgot,\hgot])$ for any $(\hgot)\in\mathcal{H}_{\mathfrak k}$,  then
$$
\Qcal_K(M,\Scal)=0
$$
for any $K$-equivariant $\spinc$-bundle $\Scal$ on $M$.

\item Assume that $([\kgot_M,\kgot_M])= ([\hgot,\hgot])$ for  $(\hgot)\in\mathcal{H}_{\mathfrak k}$. Then
$$
\mm_\Ocal= 0
$$
 if there is no $(\hgot)$-ancestor $\Pcal$ to $\Ocal$ contained in $\Phi_\Scal(M)$.
\end{itemize}
\end{theo}

 We consider the product $M\times \Ocal^*$ equipped with the  $\spinc$-bundle $\Scal\otimes\Scal_{\Ocal^*}$.
 The corresponding  moment map is $\Phi_{\Scal\otimes\Scal_{\Ocal^*}}(m,\xi)=\Phi_\Scal(m)+\xi$.
We use the simplified notation $\Phi_\Ocal$ for $\Phi_{\Scal\otimes\Scal_{\Ocal^*}}$, $\kappa_\Ocal$  for the
 corresponding Kirwan vector field on $M\times \Ocal^*$, $Z_\Ocal:=\{\kappa_\Ocal=0\}$, and  $d_\Ocal$ for the function $d_{\Scal\otimes\Scal_{\Ocal^*}}$ on $Z_\Ocal$.
Theorem \ref{theo:vanishing-1}  will result from a careful analysis of the function
$d_\Ocal : Z_\Ocal\to \R$ that was introduced in Section \ref{sec:fonction-d-S}.
Thanks to Proposition \ref{prop:annulation-d-S} and Corollary \ref{coro:annulation-d-S}, Theorem \ref{theo:vanishing-1} is a direct consequence of the following theorem.

\medskip

\begin{theo} \label{theo:vanishing-2}
Let $\Ocal$ be a regular admissible orbit.
\begin{itemize}
\item The function $d_\Ocal$ is non negative on $Z_\Ocal$.

\item If the  function $d_\Ocal$ is not strictly positive, then there exists a unique $(\hgot)\in\mathcal{H}_{\mathfrak k}$ such that the following conditions are satisfied:
\begin{enumerate}
\item $([\kgot_M,\kgot_M])= ([\hgot,\hgot])$.
\item the orbit $\Ocal$ has an $(\hgot)$-ancestor $\Pcal$ contained in $\Phi_\Scal(M)$.
\end{enumerate}
\end{itemize}
\end{theo}

\medskip

\begin{proof}
Let $P=M\times \Ocal^*$ and
let us compute the function $d_\Ocal$ on $Z_\Ocal$. Let $m\in M$ and $\lambda\in\Ocal$. The point $p=(m,-\lambda)\in Z_\Ocal\subset P$ if and only $\Phi_\Ocal(p)\cdot p=0$. Let $\beta=\Phi_\Ocal(p)$. This means that $\beta$ stabilizes  $m$ and $\lambda$, and if $\mu=\Phi_\Scal(m)\in \kgot^*$, then $\beta=\mu-\lambda$.

We write $\T_{(m,-\lambda)} P=\T_m M\oplus \T_{-\lambda}\Ocal^*$ and, since $\Ocal^*$ is a regular orbit, we have $\ntr_{\T_{-\lambda}\Ocal^*}|\beta|=\ntr_{\kgot}|\beta|$.

We consider a $K_m$-invariant decomposition $\T_m M= \kgot\cdot m\oplus E_m$ where
$\kgot\cdot m\simeq \kgot/\kgot_m$, we obtain $\ntr_{\T_m M}|\beta|=\ntr_{E_m}|\beta|+\ntr_{\kgot}|\beta|-\ntr_{\kgot_m}|\beta|$.
Thus,
\begin{eqnarray}\label{eq:minoration-d}
d_{\Ocal}(p)&=&\|\beta\|^2 +\frac{1}{2}\ntr_{\T_{(m,-\lambda)}P}|\beta| -\ntr_{\kgot}|\beta|  \nonumber \\
&=& \|\beta\|^2 +\frac{1}{2}\ntr_{\T_m M}|\beta| -\frac{1}{2}\ntr_{\kgot}|\beta| \nonumber\\
&=& \|\beta\|^2 +\frac{1}{2}\ntr_{E_m}|\beta|-\frac{1}{2}\ntr_{\kgot_m}|\beta| \nonumber\\
&\geq & \|\beta\|^2 +\frac{1}{2}\ntr_{E_m}|\beta|- \frac{1}{2}\ntr_{\kgot_\mu}|\beta|.
\end{eqnarray}

In the last inequality, we used  $\kgot_m\subset \kgot_{\mu}$ as $\mu=\Phi_\Scal(m)$. By Proposition
\ref{prop:infernal-with-trace},
$\|\beta\|^2 - \frac{1}{2}\ntr_{\kgot_\mu}|\beta| \geq 0$ when $\beta=\mu-\lambda$, as $\lambda$ is very regular (being regular and admissible), and $\beta\in \kgot_\mu\cap \kgot_\lambda$.
Then the first point follows.

Assume now that there exists a point $p=(m,-\lambda)\in Z_\Ocal$ such that $d_{\Ocal}(p)=0$. It implies then that $\|\beta\|^2 =\frac{1}{2}\ntr_{\kgot_{\mu}}|\beta|$
and $\ntr_{E_m}|\beta|=0$. The first equality implies, thanks to Proposition \ref{prop:infernal-with-trace}, that
 $K\mu$ is an admissible orbit such that $s(K\mu)=\Ocal$.
Let us denote $H=K_\mu$ : the relation $s(K\mu)=\Ocal$ implies that  $-\beta\in o(\hgot)\subset [\hgot,\hgot]^*$.
We write $-\beta=\rho^H$.
 Now we have to explain why the condition  $\ntr_{E_{m}}|\rho^H|=0$ implies  $([\kgot_M,\kgot_M])=([\hgot,\hgot])$. Since $\Phi_\Scal(m)=\mu$, we have
\begin{equation}\label{eq:sigma-stabiliser-1}
(\kgot_M)\subset (\kgot_{m})\subset (\hgot).
\end{equation}
Consider $Y=\Phi_\Scal^{-1}(U_\mu)$ the
 $H$-invariant slice constructed in Proposition   \ref{prop:slice-general}. The product $KY$ is an invariant neighborhood of $m$ isomorphic to $K\times_{H} Y$. The subspace
$E_{m}$ can be taken as the subspace $\T_{m}Y\subset \T_m M$.
 Now the condition $\ntr_{E_{m}}|\rho^H|=0$ implies that $\rho^H$ acts trivially on the connected component $Y_m$ of $Y$ containing $m$.
Elements $X\in [\hgot,\hgot]$ such that $X_{Y_m}=0$ form an ideal in  $[\hgot,\hgot]$.
 Since the ideal generated by $\rho^H$ in $[\hgot,\hgot]$ is equal to $[\hgot,\hgot]$, we have proved that $[\hgot,\hgot]$ acts trivially on $Y_m$.
Since $KY_m$ is an open subset of $M$, we get
\begin{equation}\label{eq:sigma-stabiliser-2}
([\hgot,\hgot])\subset (\kgot_M).
\end{equation}
With (\ref{eq:sigma-stabiliser-1}) and (\ref{eq:sigma-stabiliser-2}) we get $([\kgot_M,\kgot_M])=([\hgot,\hgot])$.
Finally we have proven that if $d_{\Ocal}$ vanishes at some point $p$, then $([\kgot_M,\kgot_M])= ([\hgot,\hgot])$
for some $(\hgot)\in\mathcal{H}_{\mathfrak k}$, and
there exists an admissible orbit $K\mu \subset \kgot^*_{(\hgot)}\cap \Phi_\Scal(M)$ such that $s(K\mu)=\Ocal$.
\end{proof}

\medskip

\subsubsection{Geometric properties}\label{sec:geometric}

We summarize here some of the  geometric properties enjoyed by ($M$, $\Phi=\Phi_\Scal$),
when  $\Qcal_K(M,\Scal)$ is not zero.

Let $(\hgot)\in \Hcal_\kgot$.
We choose a representative $\hgot$. Let $H$ the corresponding group and $N_K(H)$ the normalizer of $H$ in $K$.
Consider the decomposition $\hgot=[\hgot,\hgot]\oplus \zgot$ where $\zgot$ is the center
of $\hgot$. Thus $\zgot^*\subset \hgot^*$.
Consider the open set
$$
\hgot^*_0=\{\xi\in \hgot^*\ \vert\
\kgot_\xi\subset\hgot\}
$$
of $\hgot^*$. Let $\zgot^*_0=\hgot^*_0  \cap \zgot^*$ be the corresponding open subset of $\zgot^*_0$.

We first note   the following  basic proposition.
\begin{prop}\label{prop:basic}
Let $M$ be a $K$-manifold such that $([\kgot_M,\kgot_M])=([\hgot,\hgot])$
and let $\Phi: M\to \kgot^*$  be an equivariant map.
Then

\begin{itemize}

\item $\Phi(M)\subset K\zgot^*$.

\item Assume $\Ycal:=\Phi^{-1}(\hgot^*_0)$ non empty, then
\begin{itemize}
\item[a)] $\Ycal$ is a submanifold of $M$ invariant by the action of $N_K(H)$, and $[H,H]$ acts trivially on $\Ycal$.

\item[b)] The image $\Phi(\Ycal)$ is contained in $\zgot^*_0$.

\item[c)] The open subset $K\Ycal$ is diffeomorphic to $K\times_{N_K(H)}\Ycal$.

\end{itemize}
\end{itemize}

\end{prop}
\begin{proof}
Let us prove the first item.
Using our $K$-invariant inner
product, we consider $\Phi$ as a map $\Phi: M\to \kgot$.
The condition on the infinitesimal stabilizer $(\kgot_M)$ gives that $M=K M^{[H,H]}$. If $m\in M^{[H,H]}$, the term
$\Phi(m)$ belongs to the Lie algebra $\ggot$ of the centralizer subgroup $G:=Z_K([H,H])$. But one can easily prove that
$\zgot$ is a Cartan subalgebra of $\ggot$: hence $\Phi(m)$ is conjugated to an element of $\zgot$.
This proves the first item.

If $\Ycal$ is non empty, the proof that it is a submanifold follows the same line than the proof of Proposition \ref{prop:slice-general}.
The set $K\Ycal$ is a non empty open set  in $M$~: so on $\Ycal$ we have $(\kgot_M)=(\kgot_y)$
on a dense open subset $\Ycal'$. The condition
$([\kgot_M,\kgot_M])=([\hgot,\hgot])$
implies that $\dim [\hgot,\hgot]=
\dim[\kgot_y,\kgot_y]$ on $\Ycal'$, but since $\kgot_y\subset \kgot_{\Phi(y)}\subset \hgot$,
we  conclude that $[\hgot,\hgot]=
[\kgot_y,\kgot_y]\subset \kgot_y$ on $\Ycal'$ : in other words $[H,H]$  acts trivially on $\Ycal$,
and $[\hgot, \hgot]=[\kgot_y,\kgot_y]$ for any $y\in \Ycal$.
Furthermore, if $\xi=\Phi(y)$, then $[\hgot,\hgot]$ acts trivially on $\xi$. So $\xi$ is in the center of $\hgot$.

Let us  prove that $\pi: K\times_{N_K(H)} \Ycal\to K\Ycal$ is one to one. If $y_1=k y_2$, we have $\xi_1=k\xi_2$ with $\xi_i=\Phi(y_i)$.
As $\Phi(Y)\subset \zgot^*_0$, the stabilizers of $\xi_1,\xi_2$ are both equal to $H$. It follows that $k$ belongs to the normalizer of $H$.

\end{proof}
%

\medskip

The following theorem results directly from Theorem \ref{theo:vanishing-2} and Lemma \ref{prop:basic}.
Indeed, in the case where $\Qcal_K(M,\Scal)\neq \{0\}$, then
$([\kgot_M,\kgot_M])=([\hgot,\hgot])$ for some $(\hgot)\in \Hcal_\kgot$. Furthermore,
there exists at least a regular
admissible orbit $\Ocal$ such that
$m_\Ocal$ is non zero, and consequently
there exists orbit $\Pcal\subset \kgot^*_{(\hgot)}\cap\Phi_\Scal(M)$.

\begin{theo}\label{theo-H-slice}
Let $M$ be a $K$-manifold and let $\Scal$ be an equivariant $\spinc$-bundle on $M$ with moment map $\Phi_\Scal$.
Assume $\Qcal_K(M,\Scal)\neq \{0\}$.
Then
\begin{itemize}
\item[(1)]
There exists $(\hgot) \in \Hcal_\kgot$ such that
 $([\kgot_M,\kgot_M])=([\hgot,\hgot])$.

\item[(2)]
If $\zgot$ is the center of $\hgot$, then $\Phi_\Scal(M)\subset K \zgot^*$ and the open set  $\Phi_\Scal^{-1}(K\zgot_0^*)$ is non empty.

\item[(3)]
The group $[H,H]$ acts trivially on the submanifold
$\Ycal=\Phi_\Scal^{-1}(\zgot_0^*)$.

\end{itemize}
\end{theo}

This condition $(1)$ on the $K$-action is always satisfied in the Hamiltonian setting \cite{L-M-T-W}, but
not always in the spin setting as can be seen in the following example.

\begin{exam}\label{example-sphere}
For $n\geq 3$, the sphere $S^n$ admits a unique spin structure,
equivariant under the action of the group $K:=\mathrm{Spin}(n+1)$. The generic stabilizer for the $K$-action is isomorphic to the group
$K_M:=\mathrm{Spin}(n)$ and we see that $(\kgot_M)=([\kgot_M,\kgot_M])$ is not equal to  $([\hgot,\hgot])$, for any $(\hgot)\in \Hcal_\kgot$.
\end{exam}

\medskip

In the next sections, we will restrict the submanifold $\Ycal$ to a connected component $C$ of $\hgot^*_0$. We work with the $H$-invariant submanifold $\Ycal_C:=\Phi^{-1}_\Scal(C)\subset \Ycal$~: here the open subset $K\Ycal_C$ is diffeomorphic to
$K\times_H \Ycal_C$.

We follow the notations of Section \ref{sec:Slices}. We denote $\qgot^C$ the vector space $\kgot/\hgot$
equipped with the complex structure $J_C$. There exists a unique $H$-equivariant $\spinc$-bundle $\Scal_{\Ycal_C}$
on $\Ycal_C$ such that
\begin{equation}\label{eq:S-Y-C}
\Scal\vert_{\Ycal_C}\simeq \bigwedge  \qgot^C\otimes\Scal_{\Ycal_C}.
\end{equation}
At the level of determinant line bundles we have $\det(\Scal_{\Ycal_C})=\det(\Scal)\vert_{\Ycal_C}\otimes \C_{-2\rho_C}$,  and
the corresponding moment map satisfy the relation $\Phi_{\Ycal_C}=\Phi_\Scal\vert_{\Ycal_C}-\rho_C$.

We know already that the subgroup $[H,H]$ acts trivially on the submanifold $\Ycal_C$ (see Theorem \ref{theo-H-slice}). It acts also trivially on the bundle $\Scal_{\Ycal_C}$ since the moment map $\Phi_{\Ycal_C}$ takes value in $\zgot^*$ (see Remark \ref{rem:L-S-gamma}).

\subsubsection{Localization on $Z^{=0}_\Ocal$}\label{sec:localization-Z-0-Ocal}

Let $\Ocal$ be a regular admissible orbit. By Theorem \ref{theo:vanishing-2} and  Corollary \ref{coro:annulation-d-S},
we know that our object of study
$$
m_\Ocal= \left[\Qcal_K(M\times\Ocal^*,\Scal\otimes\Scal_{\Ocal^*})\right]^K
$$
is equal to $\left[\Qcal_K(M\times \Ocal^*,\Scal\otimes\Scal_{\Ocal^*},Z^{=0}_\Ocal,\Phi_\Ocal)\right]^K$.

Let us give a description of the subset  $Z^{=0}_\Ocal$ of $Z_\Ocal\subset M\times \Ocal^*$ where $d_\Ocal$ vanishes.
We denote by $q: M\times \Ocal^*\to \kgot^*\oplus \kgot^*$
the map given by $q(m,\xi)=(\Phi_\Scal(m),-\xi)$. If $\mu$ belongs to a coadjoint orbit $\Pcal$,
and  $\xi\in \mu+o(\kgot_\mu)$, then $\Pcal$ is an ancestor to the orbit $\Ocal$ of $\xi$.

\begin{defi}
Let $\Pcal$ be a coadjoint orbit.
\begin{itemize}
\item
 Define the following subset of $\kgot^*\oplus \kgot^*$:
$$R(\Pcal)=\{(\mu, \xi); \mu\in \Pcal;  \xi\in \mu+o(\kgot_\mu)\}.$$

\item Define $Z^\Pcal_\Ocal=q^{-1}(R(\Pcal))\subset M\times \Ocal^*$.
\end{itemize}
\end{defi}

\medskip

\begin{prop}\label{prop:Zop}
Assume $M$ is a $K$-manifold with $([\kgot_M,\kgot_M])=([\hgot,\hgot])$.
Let $\Scal$ be a $K$-equivariant $\spinc$-bundle over $M$ with moment map $\Phi_\Scal$.
Let $\Ocal$ be a regular admissible coadjoint.
 Then $$Z_\Ocal^{=0}=\bigsqcup_{\Pcal} Z^\Pcal_\Ocal$$
where the disjoint union is over the set of $(\hgot)$-ancestors to $\Ocal$. Furthermore, for $\Pcal$ a
$(\hgot)$-ancestor to $\Ocal$, the set $Z_\Ocal^{\Pcal}$ is equal to
$(\Phi_\Scal^{-1}(\Pcal)\times \Ocal^*)\cap Z_\Ocal^{=0}$.

\end{prop}

\begin{proof}
In the  proof of Proposition \ref{theo:vanishing-2}, we have seen that, if $d_\Ocal(m,-\lambda)=0$,
then the element $\mu=\Phi_\Scal(m)$ is such that $(\kgot_\mu)=(\hgot)$ and
$\lambda=\beta+\mu$ with $\beta\in o(\kgot_\mu)$. So $K\mu$ is a $(\hgot)$ ancestor of $\Ocal$
and $q(m,-\lambda)\in \bigsqcup_{\Pcal} Z^\Pcal_\Ocal$.
This proves the first assertion.

Conversely take now $(m,-\xi)\in Z^\Pcal_\Ocal$ , define $\mu=\Phi_\Scal(m)$.
So $K\mu$ is a $(\hgot)$ ancestor of $\Ocal$ and  $\xi=\mu+\beta$ with $\beta\in o(\kgot_\mu)$.
By $K$-invariance, we may assume $\mu\in \zgot^*_0$,
so $m\in \Ycal$. We have $T_m M=\kgot/\kgot_m\oplus T_m \Ycal$.
So
$$
d_\Ocal(m,-\xi)= \|\beta\|^2 -\frac{1}{2}\ntr_{\kgot_m}|\beta| +\frac{1}{2}\ntr_{\T_m \Ycal}|\beta|.
$$

As $\beta \in o(\hgot)\subset [\hgot,\hgot]$ acts trivially on $\Ycal$ by Lemma \ref{prop:basic}, we have
$d_\Ocal (m,-\xi)= \|\rho^H\|^2 -\frac{1}{2}\ntr_{\kgot_m}|\rho^H|$.
But since $[\hgot,\hgot]\subset \kgot_m\subset \hgot$, and then $\frac{1}{2}\ntr_{\kgot_m}|\rho^{H}|=
\frac{1}{2}\ntr_{\hgot}|\rho^H|=\|\rho^H\|^2$ : finally $d_\Ocal(m,-\xi)=0$.
\end{proof}

\bigskip

At this stage we have proved that
\begin{equation}\label{eq:sum-mop}
 \mm_\Ocal=\sum_{\Pcal}\mm_\Ocal^{\Pcal}
\end{equation}
where the sum runs over the $(\hgot)$-ancestor of $\Ocal$ and
$$
\mm^\Pcal_\Ocal=\left[\Qcal_K(M\times \Ocal^*,\Scal\otimes\Scal_{\Ocal^*},Z^\Pcal_\Ocal,\Phi_\Ocal)\right]^K.
$$
In the next section we will go into the computation of the terms $\mm^\Pcal_\Ocal$.
We finish this section with the following important fact.

\begin{prop}\label{prop:independant}
Each individual term
$\mm_\Ocal^{\Pcal}$
is independent of the choice of the moment map $\Phi_\Scal$.
\end{prop}
\begin{proof}
Let   $\Phi^t_\Scal, t\in [0,1]$ be a family of moment maps for $\Scal$. This gives a family
$\Phi^t_\Ocal(m,\xi):= \Phi^t_\Scal(m) +\xi$ for $\Scal\otimes\Scal_{\Ocal^*}$.

Let $\kappa_\Ocal^t$ be the Kirwan vector field associated to $\Phi_\Ocal^t$, and
let $Z_\Ocal(t):= \{\kappa_\Ocal^t=0\}$. We denote simply by $\sigma^t$ the symbol
$\sigma(M\times \Ocal^*,\Scal\otimes\Scal_{\Ocal^*}, \Phi^t_\Ocal)$. For any $t\in [0,1]$, we consider the quantity
$Q_\Ocal^{\Pcal}(t)\in \hat{R}(K)$ which is the equivariant index of $\sigma^t\vert_{U_t}$, where $U_t$ is a (small) neighborhood of
$$
Z_\Ocal^{\Pcal}(t)\subset Z_\Ocal(t)
$$
such that $U_t\cap Z_\Ocal(t)= Z_\Ocal^{\Pcal}(t)$.

Let us prove that the multiplicity $\mm_\Ocal^{\Pcal}(t)=[Q_\Ocal^{\Pcal}(t)]^K$ is independent
of $t$. It is sufficient to prove that $t\to [Q_\Ocal^{\Pcal}(t)]^K$ is locally constant : let us show
that it is constant in a neighborhood of $0$.
We follow the same line  of proof that the proof of the independence of the connection of the  local piece
$\Qcal_K(M, \Scal,\Phi_\Scal^{-1}(\Ocal)\cap Z_\Scal)$
of  $\Qcal_K(M, \Scal)$
 in \cite{pep-vergne:witten}.

Let $U_0$ be a neighborhood of $Z_\Ocal^{\Pcal}(0)$ such that
\begin{equation}\label{eq:U-0}
\overline{U_0} \cap
Z_\Ocal(0)= Z_\Ocal^{\Pcal}(0).
\end{equation}
The vector field $\kappa_\Ocal^0$ does not vanish on $\partial U_0$ : there
exist $\epsilon>0$ so that $\kappa_\Ocal^t$ does not vanish on $\partial U_0$ for $t\in [0,\epsilon]$. The family
$\sigma^t\vert_{U_0}$, $t\in [0,\epsilon]$ is then an homotopy of transversally elliptic symbols : hence they have the same equivariant index.
\begin{lem} For small $t$ we have
$$
U_0\cap Z^{=0}_\Ocal(t)= Z_\Ocal^{\Pcal}(t).
$$
\end{lem}
Indeed, by Proposition \ref{prop:Zop},
 $Z^{=0}_\Ocal(t)$ projects by the first projection $\Phi^t_\Scal: M\times \Ocal^*\to M\to \kgot^*$ to a finite union of coadjoint orbits (the $(\hgot)$-ancestors to $\Ocal$) and
$Z_\Ocal(0)$ projects on $\Pcal$.
So, for $t$ small,
 $U_0\cap Z^{=0}_\Ocal(t)$ is the subset  $Z_\Ocal^{\Pcal}(t)$ of
 $Z^{=0}_\Ocal(t)$ projecting on $\Pcal$.

So, for small $t$, we have the decomposition $U_0\cap Z_\Ocal(t)= Z_\Ocal^{\Pcal}(t)\cup Z_t$, where $Z_t$ is a component contained in
$Z^{>0}_\Ocal(t)$. Finally, for small $t$, we have
\begin{eqnarray*}
Q_\Ocal^{\Pcal}(0)&=&\indice_{K}(\sigma^0\vert_{U_0})\\
&=&\indice_{K}(\sigma^t\vert_{U_0})\\
&=& Q_\Ocal^{\Pcal}(t)+\Qcal_K(M\times\Ocal^*,\Scal\otimes\Scal_{\Ocal^*}, Z_t,\Phi^t_\Ocal).
\end{eqnarray*}
Since $[\Qcal_K(M\times\Ocal^*,\Scal\otimes\Scal_{\Ocal^*}, Z_t,\Phi^t_\Ocal)]^K=0$ by Proposition
\ref{prop:annulation-d-S} the proof of  Proposition
\ref{prop:independant} is completed.
\end{proof}

\subsubsection{Computation of $\mm_\Ocal^{\Pcal}$}

In this section we compute
$$
\mm_\Ocal^{\Pcal}:=\left[\Qcal_K(M\times \Ocal^*,\Scal\otimes\Scal_{\Ocal^*},Z^\Pcal_\Ocal,\Phi_\Ocal)\right]^K.
$$
Let $C$ be a connected component of $\hgot^*_0$ that intersects the orbit $\Pcal$. With the help of
Proposition \ref{induction}, we will reduce the computation of
$\mm_\Ocal^{\Pcal}$
to a similar computation
where the group $K$ acting on $M$ is replaced with the torus $A_H=H/[H,H]$ acting on the slice $\Ycal_C$.

\medskip

Let $\mu\in \Pcal\cap C$ : $\kgot_\mu=\hgot$ and $\mu-\rho(\mu)$ defines a character of $H$.
Then $Z^\Pcal_\Ocal$ is equal to $K(\Phi_\Scal^{-1}(\mu)\times (-\mu+o(\hgot)^*)$. Here $o(\hgot)$ is the
$\rho$ orbit for $H$, so $o(\hgot)=o(\hgot)^*$ and $\QS_H(o(\hgot)^*)$ is the trivial representation of $H$.

Let $\Ycal_C=\Phi_\Scal^{-1}(C)$ be the slice relative to the connected component $C$ (see Section \ref{sec:geometric}).
Thus $K\Ycal_C$ is an open neighborhood of  $\Phi_\Scal^{-1}(\Pcal)$ in $M$ diffeomorphic with $K\times_H \Ycal_C$.
We see that
$$
Z_\Ocal^{\Pcal}\subset (K\times_H \Ycal_C)\times \Ocal^*\simeq K\times_H \left( \Ycal_C\times \Ocal^*\right).
$$

We consider the $H$-manifold $N':= \Ycal_C\times o(\hgot)^*$  and the $K$-manifold
$$
N=K\times_H ( B_\qgot\times N')=K\times_H (B_\qgot\times \Ycal_C\times  o(\hgot)^*),
$$
 where $B_\qgot$ is a small open ball in $\qgot$, centered at $0$ and $H$-invariant.

When $B_\qgot$ is small enough, the map $(X,\xi)\mapsto \exp(X)(-\mu+\xi)$, from $B_\qgot\times o(\hgot)^*$
 into $\Ocal^*$,  defines a diffeomorphism into a $H$-invariant neighborhood of the $H$-orbit $-\mu+o(\hgot)^*$
 in $\Ocal^*$.
Hence a $K$-invariant neighborhood of $Z_\Ocal^\Pcal$ in
$M\times \Ocal^*$ is diffeomorphic to $N$.
Under this isomorphism, the equivariant  map $\Phi_\Ocal=\Phi_\Scal+i_{\Ocal^*}$ defines a map $\Phi$ on $N$.
For $k\in K, X\in B_\qgot,y\in \Ycal_C,\xi\in o(\hgot)^*$, we have
$$
\Phi([k;X,y,\xi]):=k\left( \Phi_\Scal(y)+\exp(X)(-\mu+\xi)\right).
$$
It restricts to $N'$ as the $H$-equivariant map
$\Phi'(y,\xi)=\Phi_\Scal(y)-\mu+\xi$ with value in $\hgot^*$.
Furthermore, if $B_\qgot$ is small enough,  $\Phi([1;X,y,\xi])$ belongs to $\hgot^*$  if and only $X=0$.
As $X\in \qgot$, we see also that $(\Phi([1;X,y,\xi]),X)=
(\Phi_\Scal(y),X)+(\exp(X)(-\mu+\xi),X)=(\Phi_\Scal(y)-\mu+\xi,X)=0$
 for all
$(X,y,\xi)\in B_\qgot\times \Ycal_C\times o(\hgot)^*$. Conditions (\ref{eq-phi-prime}) are satisfied. Proposition
\ref{induction} tells us that
$$
\mm_\Ocal^{\Pcal}:=\left[\Qcal_H(N',\Scal',Z',\Phi')\right]^H
$$
where $Z':=\Phi_\Scal^{-1}(\mu)\times o(\hgot)^*$.

Now we have to explain the nature of the spinor bundle $\Scal'$ over $N'=\Ycal_C\times o(\hgot)^*$.
Let $\Scal_{o(\hgot)^*}$ be the canonical $\spinc$-bundle of the orbit $o(\hgot)^*$. Let $\Scal_{\Ycal_C}$ be
the $\spinc$-bundle on  $\Ycal_C$ defined by (\ref{eq:S-Y-C}).

\begin{prop}
We have $\Scal'=\Scal_{\Ycal_C}^\Pcal\boxtimes \Scal_{o(\hgot)^*}$ where $\Scal_{\Ycal_C}^\Pcal=
\Scal_{\Ycal_C}\otimes \C_{-\mu+\rho(\mu)}$ is a $\spinc$-bundle on $\Ycal_C$. The determinant line bundle
of $\Scal_{\Ycal_C}^{\Pcal}$ is equal to $\det(\Scal)\vert_{\Ycal_C}\otimes \C_{-2\mu}$, and the corresponding
moment map is $\Phi_{\Ycal_C}^\Pcal:=\Phi_\Scal|_{\Ycal_C}-\mu$.

The subgroup $[H,H]$ acts trivially on $(\Ycal_C,\Scal_{\Ycal_C}^{\Pcal})$.

\end{prop}

\begin{proof}
Let $\lambda$ be an element of the $H$-orbit $\Ocal_\Pcal:=\mu+o(\hgot)$.
The spinor bundle $\Scal_{\Ocal^*}$ on $\Ocal^*=(K\lambda)^*$ induces a
spinor bundle $\Scal_1$ over $\Ocal_\Pcal^*$ through the relation
$\Scal_{\Ocal^*}\vert_{\Ocal_\Pcal^*}\simeq \bigwedge\overline{\qgot^C}\otimes \Scal_{1}$.

We can check that $\Scal_{1}$ is the $H$-spinor bundle on $\Ocal_\Pcal^*=(H\lambda)^*\simeq o(\hgot)^*$ equal to
\begin{eqnarray*}
H\times_{H_\lambda}\Big(\bigwedge_{-J_\lambda}\hgot/\hgot_\lambda\otimes \C_{-\lambda+\rho(\lambda)}\Big)
&\simeq& \Big(H\times_{H_\lambda}\bigwedge_{-J_\lambda}\hgot/\hgot_\lambda\Big)\otimes \C_{-\lambda+\rho(\lambda)}\\
&\simeq& \Scal_{o(\hgot)^*}\otimes \C_{-\mu+\rho(\mu)}
\end{eqnarray*}
since $\lambda-\rho(\lambda)=\mu-\rho(\mu)\in\zgot^*$.

As the spinor bundle $\Scal_\qgot$ is equal to the product $\bigwedge\overline{\qgot^C}
\otimes \bigwedge\qgot^C$ (see Example \ref{ex:QpluqQ}), we know then that
$\Scal'\simeq \Scal_{\Ycal_C}\boxtimes \Scal_{1}\simeq
\Scal_{\Ycal_C}\boxtimes \Scal_{o(\hgot)^*}\otimes \C_{-\mu+\rho(\mu)}$.

The relation $\det(\Scal_{\Ycal_C}^{\Pcal})=\det(\Scal)\vert_{\Ycal_C}\otimes \C_{-2\mu}$ comes from the fact that
$\det(\Scal_{\Ycal_C})= \det(\Scal)\vert_{\Ycal_C}\otimes \C_{-2\rho(\mu)}$ since $\rho_C=\rho(\mu)$.
\end{proof}

\medskip

We consider now the $H$-manifold $\Ycal_C$ equipped with the $\spinc$-bundle $\Scal_{\Ycal_C}^{\Pcal}$. Let
\begin{equation}\label{eq:Q-sigma-lambda}
\Qcal_{H}(\Ycal_C, \Scal_{\Ycal_C}^{\Pcal},\{0\})\in \hat{R}(H)
\end{equation}
be the equivariant index localized on the compact component $\{\Phi_{\Ycal_C}^\Pcal=0\}=\{\Phi_\Scal =\mu\}\subset \Ycal_C$.
Let $A_H$ be the torus $H/[H,H]$. Since $[H,H]$ acts trivially on $(\Ycal_C, \Scal_{\Ycal_C}^{\Pcal})$ we may also define the localized index $\Qcal_{A_H}(\Ycal_C, \Scal_{\Ycal_C}^{\Pcal},\{0\})\in \hat{R}(A_H)$.

We can now prove the main result of this section.
\begin{theo}\label{th:m-lambda-sigma-induction}
The multiplicity $\mm_\Ocal^{\Pcal}$ is equal to
$$
\Big[\Qcal_{H}(\Ycal_C, \Scal_{\Ycal_C}^{\Pcal},\{0\})\Big]^{H}=
\Big[\Qcal_{A_H}(\Ycal_C, \Scal_{\Ycal_C}^{\Pcal},\{0\})\Big]^{A_H}.
$$
\end{theo}

\begin{proof} Let $Z':=\Phi_\Scal^{-1}(\mu)\times o(\hgot)^*$.
The character $\Qcal_H(N',\Scal', Z',\Phi')\in\hat{R}(H)$ is equal to the equivariant index of
$\sigma(N',\Scal',\Phi')\vert_\Ucal$ where $\Ucal\subset N'$ is an invariant open subset such that
$\Ucal\cap Z_{\Phi'}=Z'$. For $(y,\xi)\in N'=\Ycal_C\times o(\hgot)^*$ and $(v,\eta)\in \T_{(y,\xi)}N'$, the endomorphism
$\sigma(N',\Scal_{N'},\Phi')\vert_{(y,\xi)}(v,\eta)$ is equal to
$$
\clif_1(v+(\Phi_\Scal(y)-\mu+\xi)\cdot y)\otimes {\rm Id}_{\Scal_{o(\hgot)^*}\vert_\xi} + \epsilon_1\otimes  \clif_2(\eta+(\Phi_\Scal(y)-\mu+\xi)\cdot\xi).
$$
Here $\clif_1$ acts on $\Scal_{\Ycal_C}^{\Pcal}\vert_y$, $\clif_2$ acts on $\Scal_{o(\hgot)^*}\vert_\xi$ and $\epsilon_1$
is the canonical grading operator on $\Scal_{\Ycal_C}^{\Pcal}\vert_y$.

Since $o(\hgot)^*$ is compact, we can replace the term $\clif_2(\eta+(\Phi_\Scal(y)-\mu+\xi)\cdot\xi)$
simply by  $\clif_2(\eta)$. Since $[H,H]$ acts trivially on $\Ycal_C$, and $\xi\in [\hgot,\hgot]$, the vector
field $y\mapsto (\Phi_\Scal(y)-\mu+\xi)\cdot y$ is equal to $y\mapsto (\Phi_\Scal(y)-\mu)\cdot y$. Thus our symbol
is homotopic to the  symbol
$$
\clif_1(v+(\Phi_\Scal(y)-\mu)\cdot y)\otimes
 {\rm Id}_{\Scal_{o(\hgot)}\vert_\xi} + \epsilon_1\otimes  \clif_2(\eta).
$$

This last expression is the  product symbol of the $H$-transversally elliptic symbol $\clif_1(v+(\Phi_\Scal(y)-\mu)\cdot y)$
on $\Ycal_C$ and of the elliptic symbol $\clif_2(\eta)$ on $o(\hgot)^*$. The equivariant indices multiply under the product (as one is elliptic) (\cite{Atiyah74},\cite{pep-vergne:bismut}).

Now the $H$-equivariant index of $c_2(\eta)$ acting on $\Scal_{o(\hgot)^*}$ is the trivial representation of $H$.
Thus we obtain our theorem. We have also to remark that the identity
$\left[\Qcal_{H}(\Ycal_C, \Scal_{\Ycal_C}^{\Pcal},\{0\})\right]^{H}=
\left[\Qcal_{A_H}(\Ycal_C, \Scal_{\Ycal_C}^{\Pcal},\{0\})\right]^{A_H}$
follows from the fact that $[H,H]$ acts trivially on $(\Ycal_C, \Scal_{\Ycal_C}^{\Pcal})$.
\end{proof}

\section{Multiplicities and reduced spaces }\label{sec:multiplicity}

In this section, we interpret the multiplicity as an equivariant index on a reduced space.

Let $\Ocal\subset\kgot^*$ be a regular admissible orbit, and $(\hgot)\in \Hcal_\kgot$ so that
$([\hgot,\hgot])=([\kgot_M,\kgot_M])$. In the previous section, we have proved that the multiplicity
of $\pi_\Ocal$ in $\Qcal_K(M,\Scal)$ is equal to
$$
\mm_\Ocal= \sum_{\Pcal} {\mm_\Ocal^{\Pcal}}
$$
where the sum runs over the $K$-orbits $\Pcal$ which are  $(\hgot)$-ancestors of $\Ocal$.
Furthermore, we have proved that $\mm_\Ocal^{\Pcal}=\Big[\Qcal_{A_H}(\Ycal_C, \Scal_{\Ycal_C}^{\Pcal},\{0\})\Big]^{A_H}$.

The  aim  of this section is to prove the following theorem.

\begin{theo}\label{theo:QR}
The multiplicity  $\mm_\Ocal^{\Pcal}$ is equal to the spin$^c$ index of the (possibly singular) reduced space
$M_{\Pcal}:=\Phi^{-1}_\Scal(\Pcal)/K$.
\end{theo}

However, our first task is to give a meaning to
a  $\QS(M_\Pcal)\in \Z$ even if $M_{\Pcal}$ is singular.

\subsection{Spin$^c$ index on singular reduced spaces}\label{sec:spin-index-singular}

We consider a connected oriented manifold $N$, equipped with a $\spinc$-bundle $\Scal$. We assume that a torus $G$
acts on the data $(N,\Scal)$. An invariant connection on the determinant line bundle $\det(\Scal)$ defines
a moment map $\Phi : N\to \ggot^*$. We do not assume that $N$ is compact, but we assume that the
map $\Phi$ is proper\footnote{We will use sometimes a slightly different hypothesis : $\Phi$ is proper as a map
from $N$ to an open subset of $\ggot^*$.}.
Thus, for any $\xi\in \ggot^*$, the reduced space $N_\xi:=\Phi^{-1}(\xi)/G$ is compact.

The purpose of this subsection is to explain how we can define the spin$^c$-index, $\QS(N_\mu)\in \Z$, for any $\mu$ in the weight lattice $\Lambda$ of the torus $G$.

Let $\ggot_N$ be the generic infinitesimal stabilizer of the $G$-action on $N$ : the image of $N$ under the map $\Phi$ leaves in an affine space $I(N)$ parallel to $\ggot_N^\perp$.  If $\xi\in I(N)$ is a regular value of $\Phi: N\to I(N)$, the reduced space $N_\xi$ is a compact orbifold (as proved in \cite{pep-vergne:witten}). We can define $\spinc$-bundles on orbifolds, as well as $\spinc$-indices.

We start with the following basic fact.
\begin{lem}
For any regular value $\xi\in I(N)$ of $\Phi: N\to I(N)$, the orbifold $N_\xi$ is oriented and equipped with a family of $\spinc$-bundles $\Scal_{\xi}^{\mu}$  parameterized by $\mu\in \Lambda\cap I(N)$.
\end{lem}

\begin{proof}
Let $G_N$ be the subtorus with Lie algebra $\ggot_N$. Let $G'=G/G_N$. The dual of the Lie algebra $\ggot'$ of $G'$ admits a canonical identification with $\ggot_N^\perp$.

We assume that $\xi$ is a regular value of $\Phi:N\to I(N)$ : the fiber $Z=\Phi^{-1}(\xi)$ is a submanifold
equipped with a locally free action of $G'$. Let $N_{\xi} :=Z/G'$ be the corresponding ``reduced'' space, and
let $\pi: Z\to N_{\xi}$ be the projection map.
We can define the tangent (orbi)-bundle $TN_\xi$ to $N_\xi$.

On $Z$, we obtain an exact sequence
$0\longrightarrow \T Z\longrightarrow \T N\vert_Z \stackrel{d\Phi_\epsilon}{\longrightarrow} Z\times (\ggot')^*\to 0$, and
an orthogonal decomposition $\T Z= \T_{G'} Z \oplus \ggot'_Z$ where $\ggot'_Z$ is the trivial bundle corresponding to the subspace of $\T Z$ formed by the vector fields generated by the infinitesimal action of $\ggot'$. So $\T N\vert_Z$ admits the decomposition
$\T N\vert_Z \simeq \T_{G'} Z\oplus \ggot'_Z \oplus  [(\ggot')^*]$. We rewrite this as
\begin{equation}\label{eq:tangent-Z-1}
\T N\vert_Z \simeq \T_{G'} Z\oplus [\ggot'_\C]
\end{equation}
with the convention $\ggot'_Z=Z\times(\ggot'\otimes i\R) $ and $Z\times (\ggot')^* = Z\times (\ggot'\otimes \R)$.
Note that the bundle $\T_{G'} Z$ is naturally identified with $\pi^*(\T N_{\xi})$.

If we take on $\ggot'_\C$ the orientation $o(i)$ given by the complex structure, there exists a unique orientation $o(N_{\xi})$ on $N_{\xi}$ such that $o(N)=o(N_{\xi}) o(i)$.

\begin{defi} Let $\widetilde{\Scal}_{\xi}$ be the $\spinc$  bundle on the vector bundle $\T_{G'} Z\to Z$ such that
$$
\Scal\vert_Z \simeq\widetilde{\Scal}_{\xi}\otimes [\bigwedge\ggot'_\C].
$$
\end{defi}

The Kostant relation shows that for any $X\in \ggot_N$, the element $e^X$ acts on the fibers of
$\widetilde{\Scal}_{\xi}$ as a multiplication by $e^{i\langle \nu, X\rangle}$ where $\nu$ is any element of
$I(N)$. Hence, for any $\mu\in\Lambda\cap I(N)$, the action of $G_N$ on the tensor $\widetilde{\Scal}_{\xi}\otimes [\C_{-\mu}]$ is trivial. We can then define a $\spinc$-bundle $\Scal_\xi^\mu$ on  $\T N_{\xi}$ by the relation
$$
\widetilde{\Scal}_{\xi}\otimes [\C_{-\mu}]=\pi^*\left(\Scal_\xi^\mu\right).
$$
\end{proof}

\medskip

The proof of the following theorem is given in the next subsection.

\begin{theo}\label{theo:Q-N-mu}
For any $\mu\in I(N)\cap \Lambda$, consider the compact oriented orbifold $N_{\mu+\epsilon}$
associated to a generic\footnote{So that $\mu+\epsilon$ is a regular value of $\Phi: N\to I(N)$.}
element $\epsilon\in\ggot_N^\perp$. Then the index
$$
\Qcal(N_{\mu+\epsilon},\Scal_{\mu+\epsilon}^{\mu})
$$
is independent of the choice of a generic and small enough $\epsilon$.
\end{theo}

Thanks to the  previous Theorem, one defines the spin$^c$ index of singular reduced spaces as follows.

\begin{defi}\label{def:Q-M-mu}
If $\mu\in \Lambda$, the number $\QS(N_{\mu})$ is defined by the following dichotomy
$$
\QS(N_{\mu})=
\begin{cases}
   0\qquad\hspace{22mm} {\rm if}\ \mu\notin \  I(N),\\
   \Qcal(N_{\mu+\epsilon},\Scal_{\mu+\epsilon}^{\mu})
   \qquad {\rm if}\ \mu\in \  I(N)\ {\rm  and}\ \epsilon\in\ggot_N^\perp\
    \mathrm{is\ generic}\\
    \mathrm{and\ small\ enough}.
\end{cases}
$$
\end{defi}

The invariant $\QS(N_{\mu})\in \Z$ vanishes if
$\mu$ does not belongs to the relative interior of $\Phi(N)$ in the affine space $I(N)$. It is due to the fact that we can then approach $\mu$ by elements $\mu+\epsilon$ that are not in the image $\Phi(N)$.

\medskip

Let us consider the particular case where $\mu \in I(N)\cap \Lambda$ is a regular value of $\Phi :N\to I(N)$ such that
the reduced space $N_\mu$ is reduced to a point. Let $m_o\in \Phi^{-1}(\mu)$, and let $\Gamma\subset G'$ be the
stabilizer subgroup of $m_o$ ($\Gamma$ is finite). In this case (\ref{eq:tangent-Z-1}) becomes
$\T_{m_o}N\simeq \ggot'_\C$, and $o(N_\mu)$ is the quotient between the orientation of $N$ and those of
$\ggot'_\C$. At the level of graded $\spinc$-bundles we have
$$
\Scal_{m_o}\simeq o(N_\mu) \bigwedge\ggot'_\C\otimes \det(\Scal)\vert_{m_o}^{1/2}
$$
where $\det(\Scal)\vert_{m_o}^{1/2}$ is a one dimensional representation of $\Gamma$ such that \break
$(\det(\Scal)\vert_{m_o}^{1/2})^{\otimes 2}=\det(\Scal)\vert_{m_o}$. In this case Definition \ref{def:Q-M-mu} becomes
\begin{equation}\label{eq:Q-point}
\QS(N_{\mu})=o(N_\mu) \, \mathrm{dim}\left[\det(\Scal)\vert_{m_o}^{1/2}\otimes\C_{-\mu}\right]^\Gamma\ \in\ \{-1,0,1\}.
\end{equation}

\subsection{Proof of Theorem \ref{theo:Q-N-mu}}

In this subsection we work with a fixed $\mu\in I(N)\cap \Lambda$. For any $\epsilon\in \ggot(N)^\perp$, we consider
the moment map $\Phi_\epsilon=\Phi-\mu-\epsilon$.

We start with the fundamental Lemma
\begin{lem}\label{lem:epsilon}
The map $\epsilon\mapsto[\Qcal_G(N,\Scal,\Phi_\epsilon^{-1}(0),\Phi_\epsilon)\otimes\C_{-\mu}]^G$ is constant in a neighborhood of $0$.
\end{lem}

\begin{proof}
Changing $\Scal$ to $\Scal\otimes [\C_{-\mu}]$, we might as well take $\mu=0$.

Let $r>0$ be smallest non-zero critical value of $\|\Phi\|^2$, and let $\Ucal:=\Phi^{-1}(\{\xi\ |\ \|\xi\|< r/2\})$. Using Lemma \ref{lem:crit}, we have
$\Ucal\cap \{\kappa_0=0\}=\Phi^{-1}(0)$.

We describe now $\{\kappa_\epsilon=0\}\cap \Ucal$ using
a parametrization  similar to those introduced in \cite{pep-1999}[Section 6].

Let $\ggot_i, i\in I$ be the finite collection of infinitesimal stabilizers for the $G$-action on the compact set $\overline{\Ucal}$.
Let $\Dcal$ be the subset of the collection of subspaces $\ggot_i^{\perp}$ of $\ggot^{*}$ such that
 $\Phi^{-1}(0)\cap N^{\ggot_i}\neq \emptyset.$

Note that $\Dcal$ is reduced to $I(N)$ if $0$ is regular value of $\Phi : N\to I(N)$.
If $\Delta=\ggot_i^{\perp}$ belongs to $\Dcal$, and $\epsilon\in I(N)$, write the orthogonal decomposition $\epsilon=\epsilon_\Delta+\beta_\Delta$ with $\epsilon_\Delta\in \Delta$, and $\beta_\Delta\in \ggot_i$.
Let
$$\Bcal_\epsilon=\{\beta_\Delta=\epsilon-\epsilon_\Delta, \Delta\in \Dcal\}$$
the set of $\beta$ so obtained.

\begin{figure}[!h]
\begin{center}
  \includegraphics[width=2 in]{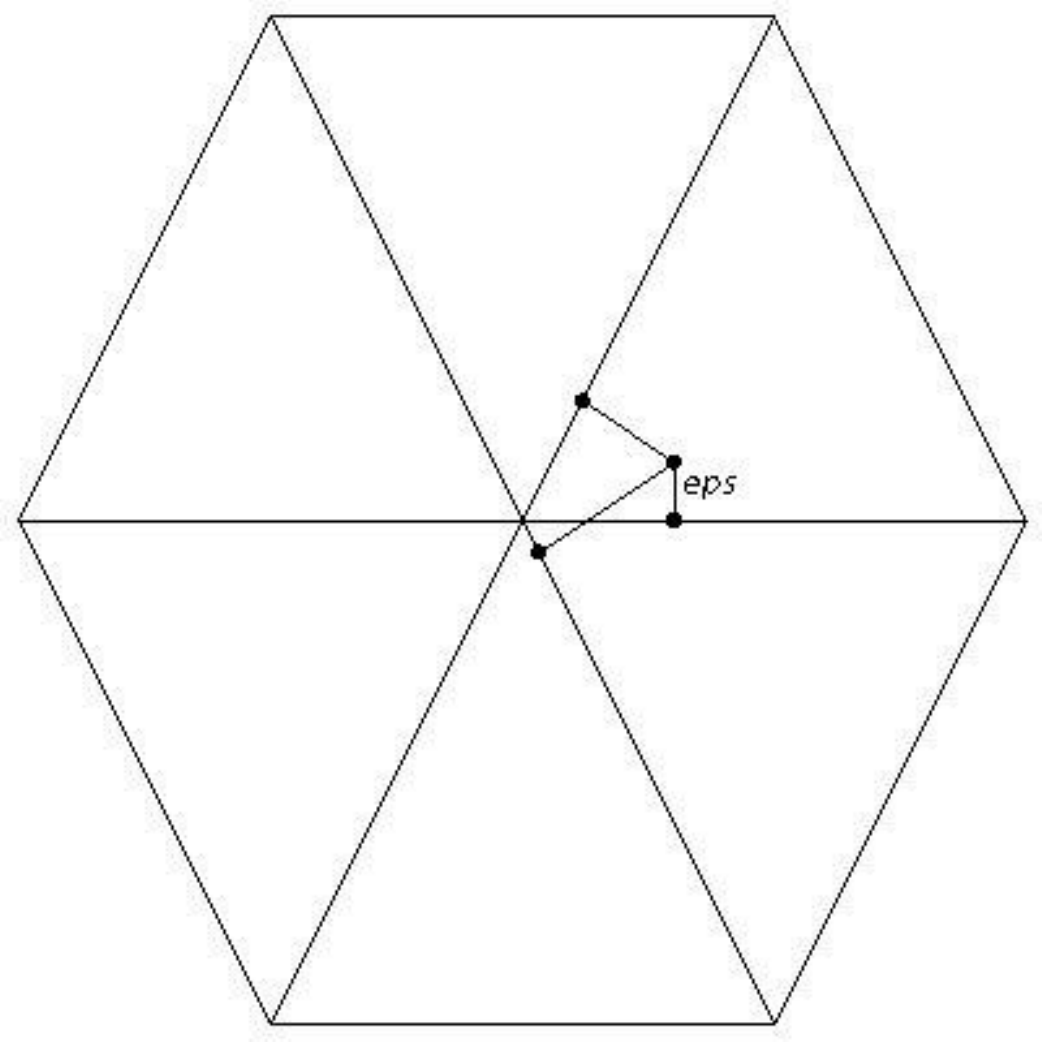}
\caption{ The point $\epsilon$ and its projections $\epsilon_\Delta$}
 \label{fig:betas}
\end{center}
\end{figure}

We denote by $Z_\epsilon$ the zero set of the vector field $\kappa_\epsilon$ associated to $\Phi_\epsilon$.
Thus, if $\epsilon$ is sufficiently small ($\|\epsilon\|<r/2$),
\begin{equation}\label{eq.decomposition.delta}
Z_\epsilon\cap\Ucal
=\bigcup_{\beta\in\Dcal_\epsilon}
N^{\beta}\cap\Phi^{-1}_\epsilon(\beta).
\end{equation}

With (\ref{eq.decomposition.delta}) in hands, we see easily that $t\in [0,1]\mapsto \sigma(N,\Scal,\Phi_{t\epsilon})\vert_{\Ucal}$ is an homotopy of transversally elliptic symbols on $\Ucal$. Hence they have the same index
\begin{eqnarray*}
\Qcal_G(\Ucal,\Scal,\Phi^{-1}(0),\Phi)
&=& \Qcal_G(\Ucal,\Scal, Z_\epsilon\cap \Ucal,\Phi_\epsilon)\\
&=& \sum_{\beta\in\Bcal_\epsilon} \Qcal_G(N,\Scal,\Phi_\epsilon^{-1}(\beta)\cap N^{\beta },\Phi_\epsilon).
\end{eqnarray*}
The lemma will be proved if we check that
$[\Qcal_G(N,\Scal,\Phi_\epsilon^{-1}(\beta)\cap N^{\beta },\Phi_\epsilon)]^G=0$ for any non-zero $\beta\in\Bcal_\epsilon$.

If $\beta_\Delta\in \Bcal_\epsilon$ and $n\in
\Phi_\epsilon^{-1}(\beta_\Delta)\cap N^{\beta_\Delta }$, $\Phi(n)=\beta_\Delta+\epsilon=\epsilon_\Delta$. So
$\langle\Phi(n),\beta_\Delta\rangle=\langle\epsilon_\Delta,\beta_\Delta\rangle=0$. So the infinitesimal action,
$\Lcal(\beta)$, on the fiber of the vector bundle $\Scal_n$ is equal to $0$.

The Atiyah-Segal localization formula for the Witten deformation (Remark \ref{rem:Atiyahsegalabelian}) gives
\begin{eqnarray*}
\Qcal_G(N,\Scal,\Phi_\epsilon^{-1}(\beta)\cap N^{\beta },\Phi_\epsilon)
&=& \Qcal_G(N^\beta,\mathbf{d}_\beta(\Scal)\otimes \mathrm{Sym}(\Vcal_\beta),\Phi_\epsilon^{-1}(\beta),\Phi_\epsilon)\\
&=&\sum_{\Xcal\subset N^\beta}
\Qcal_G(\Xcal,\mathbf{d}_\beta(\Scal)\vert_\Xcal\otimes \mathrm{Sym}(\Vcal_\beta)\vert_\Xcal,\Phi_\epsilon^{-1}(\beta),\Phi_\epsilon)
\end{eqnarray*}
where $\Vcal_\beta\to N^\beta$ is the normal bundle of $N^\beta$ in $N$ and  the sum runs over the connected components $\Xcal$ of $N^\beta$ that intersects $\Phi^{-1}_\epsilon(\beta)$.

Let us look to the infinitesimal action of $\beta$, denoted $\Lcal(\beta)$, on the fibers of the vector bundle $\mathbf{d}_\beta(\Scal)\vert_{\Xcal}\otimes\mathrm{Sym}(\Ncal_\beta)\vert_{\Xcal}$.

This action can be checked at a point  $n\in\Phi^{-1}_\epsilon(\beta)\cap N^{\beta}$.
As the action of $\beta$ on the fiber of the vector bundle $\Scal_n$ is equal to $0$, we obtain

$$
\frac{1}{i}\Lcal(\beta)=
\begin{cases}
\frac{1}{2}\tr_{\T N\vert_\Xcal}(|\beta|)\hspace{25mm} {\rm on}\ \ \mathbf{d}_\beta(\Scal)\vert_{\Xcal},\\
\geq 0\hspace{40mm} {\rm on}\ \ \mathrm{Sym}(\Ncal_\beta)\vert_{\Xcal}
\end{cases}
$$

So we have checked that $\frac{1}{i}\Lcal(\beta)\geq \frac{1}{2}\tr_{\T N\vert_\Xcal}(|\beta|)$ on
$\mathbf{d}_\beta(\Scal)\vert_{\Xcal}\otimes\mathrm{Sym}(\Ncal_\beta)\vert_{\Xcal}$.

Now we remark that $\beta$ does not acts trivially on $N$, since $\beta$ belongs to the direction
of the  subspace $I(N)=\ggot_N^{\perp}$: this forces $\frac{1}{2}\tr_{\T N\vert_\Xcal}(|\beta|)$ to be strictly positive.
Finally we see that $\frac{1}{i}\Lcal(\beta)>0$ on $\mathbf{d}_\beta(\Scal)\vert_{\Xcal}\otimes
\mathrm{Sym} (\Ncal_\beta)\vert_{\Xcal}$, and then
$$\left[
\Qcal_G(\Xcal,\mathbf{d}_\beta(\Scal)\vert_\Xcal\otimes \mathrm{Sym}(\Vcal_\beta)\vert_\Xcal,\Phi_\epsilon^{-1}(\beta),\Phi_\epsilon)\right]^G=0.
$$
if $\beta\neq 0$. The Lemma \ref{lem:epsilon} is proved.
\end{proof}

\medskip

The proof of Theorem \ref{theo:Q-N-mu} will be completed with the following

\begin{lem}\label{lem:epsilon-index}
If $\mu+\epsilon$ is a regular value of $\Phi :N\to I(N)$, the invariant
$[\Qcal_G(N,\Scal,\Phi_\epsilon^{-1}(0),\Phi_\epsilon)\otimes\C_{-\mu}]^G$ is equal to the index
$\Qcal(N_{\mu+\epsilon},\Scal_{\mu+\epsilon}^{\mu})$.
\end{lem}

We assume that $\mu+\epsilon$ is a regular value of $\Phi :N\to I(N)$ : the fiber $Z=\Phi^{-1}(\mu+\epsilon)$ is a submanifold equipped with a locally free action of $G'=G/G_N$. Let $N_{\mu+\epsilon} :=Z/G'$ be the corresponding ``reduced'' space, and let $\pi: Z\to N_{\mu+\epsilon}$ be the projection map.  We have the decomposition
\begin{equation}\label{eq:tangent-Z-preuve}
\T N\vert_Z \simeq \pi^*(\T N_{\mu+\epsilon})\oplus [\ggot'_\C].
\end{equation}
For any $\nu\in \Lambda\cap I(N)$, $\Scal_{\mu+\epsilon}^\nu$ is a the spinor bundle on $N_{\mu+\epsilon}$ defined by the relation
$$
\Scal\vert_Z\otimes \C_{-\nu} \simeq\pi^*\left(\Scal_{\mu+\epsilon}^\nu\right)\otimes [\bigwedge \ggot'_\C].
$$

The following result is proved in  \cite{pep-vergne:witten}.

\begin{prop}We have the following equality in $\hat{R}(G)$
$$\Qcal_G(N,\Scal,\Phi_\epsilon^{-1}(0),\Phi_\epsilon)=\sum_{\nu\in\Lambda\cap I(N)}
\Qcal(N_{\mu+\epsilon}, \Scal_{\mu+\epsilon}^{\nu}) \ \C_\nu.
$$
In particular $[\Qcal_G(N,\Scal,\Phi_\epsilon^{-1}(0),\Phi_\epsilon)\otimes\C_{-\mu}]^G$ is equal to
$\Qcal(N_{\mu+\epsilon},\Scal_{\mu+\epsilon}^{\mu})$.
\end{prop}

\subsection{$[Q,R]=0$}

We come back to the setting of a compact $K$-manifold $M$, oriented and of even dimension, that is equipped with a $K$-spinor bundle $\Scal$. Let $\det(\Scal)$ its determinant bundle, and let $\Phi_\Scal$ be the moment map that is attached to an invariant connection on $\det(\Scal)$. We assume that there exists $(\hgot)\in \Hcal_\kgot$ such that
$([\kgot_M,\kgot_M])=([\hgot,\hgot])$. Let $\zgot$ be the center of $\hgot$.

We consider an admissible element $\mu\in\zgot^*$ such that $K_\mu=H$ : the coadjoint orbit $\Pcal:= K\mu$ is admissible and contained in the Dixmier sheet $\kgot^*_{(\hgot)}$. Let
$$
M_\Pcal:= \Phi_\Scal^{-1}(\Pcal)/K
$$
In order to define $\QS(M_\Pcal)\in\Z$ we proceed as follows.

\medskip

We follow here the notations of Section \ref{sec:Slices}. Let $C$ be the connected component of
$\hgot^*_0:=\{\xi\in\hgot^*\ \vert K_\xi\subset H\}$ containing $\mu$. The slice $\Ycal_C=\Phi^{-1}_\Scal(C)$ is
a $H$-submanifold of $M$ equipped with a $H$-$\spinc$ bundle $\Scal_{\Ycal_C}$: the associated
moment map is $\Phi_{\Ycal_C}:=\Phi_\Scal\vert_{\Ycal_C}-\rho_C$ where $\rho_C$ is defined by
(\ref{def:rho-C}).

The element $\tilde{\mu}:=\mu-\rho(\mu)=\mu -\rho_C$ belongs to the weight lattice $\Lambda$ of the torus $A_H:=H/[H,H]$, and the reduced space $M_{K\mu}$ is equal
to
$$
(\Ycal_C)_{\tilde{\mu}}:=\{\Phi_{\Ycal_C}=\tilde{\mu}\}/A_H.
$$

By definition, we take $\QS(M_{K\mu}):=\QS((\Ycal_C)_{\tilde{\mu}})$ where the last term is computed as explained
in the previous section. More precisely, let us decompose $\Ycal_C$ into its connected components $\Ycal_1,\ldots, \Ycal_r$.
For each $j$, let $\zgot_j\subset\zgot$ be the generic infinitesimal stabilizer relative to the $A_H$-action on $\Ycal_j$.
Then we take
$$
\QS(M_\Pcal)=\QS(M_{K\mu}):=\sum_{j} \QS\left( (\Ycal_j)_{\tilde{\mu}+\epsilon_j}\right)
$$
where $\epsilon_j\in\zgot_j^\perp$ are generic and small enough.

\medskip

With this definition of quantization of reduced spaces
$\QS(M_\Pcal)$, we obtain the main theorem of this article, inspired by the $[Q,R]=0$ theorem of Meinrenken-Sjamaar.

Let $M$ be a $K$-manifold and $\Scal$ be a $K$-equivariant $\spinc$-bundle over $M$.
Let $(\hgot)\in \Hcal_\kgot$ such that $([\kgot_M,\kgot_M])=([\hgot,\hgot])$, and consider the set
$\Acal((\hgot))$ of admissible orbits contained in the Dixmier sheet $\kgot^*_{(\hgot)}$.

\medskip

\begin{theo}\label{theo:final}
\begin{equation}\label{eq:QR=0}
\Qcal_K(M,\Scal)=\sum_{\Pcal
\in \Acal((\hgot))} \QS(M_\Pcal) \QS_K(\Pcal).
\end{equation}
\end{theo}

\medskip

We end this section by giving yet another criterium for the vanishing of $\Qcal_K(M,\Scal)$.

Consider the map $\Phi_\Scal: M\to \kgot^*$.
At each point $m\in M$, the differential $d_m\Phi_\Scal$ gives a map $T_mM\to \kgot^*$.
Let $\kgot_m^{\perp}\subset \kgot^*$.
From the  Kostant relations, we see that $d_m\Phi_\Scal$ take
value in $\kgot_m^{\perp}.$

\begin{prop}
If $\Qcal_K(M,\Scal)\neq 0$, then there exists $m\in M\setminus M^K$ such that
$\mathrm{Image}(d_m\Phi_\Scal)=\kgot_m^{\perp}.$
\end{prop}
\begin{proof}
If we consider the decomposition of the slice $\Ycal_C=\bigcup \Ycal_j$ in connected components,
for $\Qcal_K(M,\Scal)\neq 0$, then for some $j$,  $\Phi(Y_j)$ has non empty interior in $\zgot_j^{\perp}$. Here $\zgot_j$ is the infinitesimal stabilizer of the action of $H/[H,H]$ on $\Ycal_j$.
Thus $\zgot_j$ is equal to $\kgot_m\subset \hgot$ for generic $m\in \Ycal_j$.
So there exists a point $m\in \Ycal_j$ such that
the differential of $\Phi_\Scal|_{\Ycal}$ is surjective on
$\zgot_j^{\perp}\subset \hgot^*$.
Now if we consider $K\Ycal_j\subset M$, then
$\mathrm{Image}(d_m\Phi_\Scal)=\hgot^*\oplus \zgot_j^{\perp}$. This is exactly $\kgot_m^{\perp}$.
\end{proof}

\medskip

When the action of $K$ is abelian, we can always reduce ourselves to an effective action with $\kgot_M=\{0\}$.
Then the support of  decomposition of $\Qcal_K(M,\Scal)$
is contained in the interior of $\Phi_\Scal(M)\cap \Lambda$.
If this set has no interior point, then $\Qcal_K(M,\Scal)=0$.
This small remark implies the well-known Atiyah-Hirzebruch vanishing theorem in the spin case \cite{Atiyah-Hirzebruch70}, as well as the variant of Hattori \cite{Hattori78}.

We also note another  corollary.

\begin{coro}
If the two form $\Omega_\Scal$ is exact, and the $K$-action on $M$ is non-trivial then $\Qcal_K(M,\Scal)=0$.
\end{coro}
It is due to the fact that if $\Omega_\Scal=d\alpha$, by modifying the connection on $\det(\Scal)$ by $\alpha$, our moment map  is constant.
So if the action is non trivial, $\Qcal_K(M,\Scal)=0$.

\section{Examples: multiplicities and reduced spaces}\label{sec:examples}

In this last section, we give some simple examples in order to illustrate various features of our result relating
multiplicities and reduced spaces.

An open question remains even for a toric manifold $M$ equipped with a non ample
line bundle $L$.
The determinant line bundle of the spinor bundle $\Scal:=\bigwedge_{\C}\T M\otimes L$ is
equal to $\det(\Scal):=\det_\C(\T M)\otimes L^{\otimes 2}$. A connection $\nabla$ on $\det(\Scal)$ determines
a moment map $\Phi_\nabla:M\to\tgot^*$ and a curvature $\Omega_{\nabla}=\frac{1}{2i}\nabla^2$.
The push-forward of  the density $(\Omega_{\nabla})^{\dim M/2}$ by $\Phi_\nabla$  does not depends on the choice of the connection: it is a
signed measure, denoted $DH(M,\Scal)$,  and we still call it  the Duistermaat-Heckmann measure. The support of
$DH(M,\Scal)$ is a union of convex polytopes
contained in $\Phi_{\nabla}(M)$. Can we find $\nabla$
such that the image $\Phi_{\nabla}(M)$ is exactly the support of the Duistermaat-Heckman measure~?

\subsection{The reduced space might not be connected}\label{exa:P1}
We consider the simplest case of the theory. Let $\Pbb^1:=\Pbb^1(\C)$ be the projective space of (complex) dimension one.
Consider the (ample) line bundle $\Lcal\to \Pbb^1$, dual of the tautological bundle. It is obtained as quotient of the trivial line bundle
$\C^2\setminus\{(0,0)\} \times \C$ on $\C^2\setminus\{(0,0)\}$ by the action $u\cdot(z_1,z_2,z)= (u z_1, u z_2,uz)$ of $\C^*$.
We consider the action of $T=S^1$ on $\Lcal\to\Pbb^1$ defined by $t\cdot [z_1,z_2, z]= [t^{-1} z_1,  z_2, z]$.

Let $\Scal(n)$ be the $\spinc$-bundle $\bigwedge_\C \T\Pbb^1 \otimes \Lcal^{\otimes n}$. The character $\QS_T(M, \Scal(n))$ is equal to
$H^{0}(\Pbb^1,\Ocal(n))-H^{1}(\Pbb^1,\Ocal(n))$ where $\Ocal(n)$ is the sheaf of holomorphic sections of $\Lcal^{\otimes n}$. Note that the holomorphic line bundle $\Lcal^{\otimes n}$ is not ample if $n\leq 0$. We have
\begin{itemize}
\item $\QS_T(M, \Scal(n))= - \sum_{k=n+1}^{-1}t^{k}$ when $n\leq  -2$,
\item $\QS_T(M, \Scal(-1))= 0$,
\item $\QS_T(M, \Scal(n))=  \sum_{k=0}^{n}t^{k}$ when $n\geq 0$.
\end{itemize}

The determinant line bundle of $\Scal(n)$ is  $\Lbb_n=[\C_{-1}]\otimes \Lcal^{\otimes 2n+2}$
where $[\C_{-1}]$ is the trivial line bundle equipped with the representation $t^{-1}$ on $\C$.

Remark that $\Pbb^1$ is homogeneous under $U(2)$, so there exists a unique $U(2)$-invariant connection on $\Lbb_n$.
The corresponding moment map $\Phi_{\Scal(n)}$ is such that
\begin{equation}\label{eq:PhiP1}
\Phi_{\Scal(n)}([z_1,z_2])=(n+1)\frac{|z_1|^2}{|z_1|^2+|z_2|^2}-\frac{1}{2}.
\end{equation}

The image $I_n=\Phi_{\Scal(n)}(M)$ is
\begin{itemize}
\item the interval $[-\frac{1}{2},n+\frac{1}{2}]$ when $n\geq 0$,
\item  a point $\{-\frac{1}{2}\}$ when $n=-1$,
\item the interval $[n+\frac{1}{2}, -\frac{1}{2}]$ when $n\leq -2$,
\end{itemize}

It is in agreement with our theorem. Indeed all characters occurring in $\QS_T(M,\Scal(n))$ are the integral points in the relative interior of $I_n$, and all reduced spaces are points.

If we consider simply the action of $T$ on $\Pbb^1$, the choice of connection may vary.
In fact, given any smooth function $f$ on $\R$, we can define a connection $\nabla^f$ on $\Lbb_n$ such  that the corresponding moment map is
$$
\Phi_{\Scal(n)}^f([z_1,z_2])=\Phi_{\Scal(n)}([z_1,z_2])+f\left(\frac{|z_1|^2}{|z_1|^2+|z_2|^2}\right)
\frac{|z_1|^2|z_2|^2}{(|z_1|^2+|z_2|^2)^2}.
$$
Let $\Omega_{\Scal(n)}^f$ be half the curvature of $(\Lbb_n, \nabla^f)$, then the Duistermaat-Heckman
measure $(\Phi_{\Scal(n)}^f)_*\Omega_{\Scal(n)}^f$ is independent of the choice of the connection $\nabla^f$ and is equal to
the characteristic function of $I_n$.

Take for example $n=4$ and the function $f(x)=15$. the corresponding moment map is
$$
\Phi([z_1,z_2])=5\frac{|z_1|^2}{|z_1|^2+|z_2|^2}+15
\frac{|z_1|^2|z_2|^2}{(|z_1|^2+|z_2|^2)^2}-\frac{1}{2}.
$$
Figure \ref{graphPhi} is the graph of $\Phi$ in terms of $x=\frac{|z_1|^2}{|z_1|^2+|z_2|^2}$ varying between $0$ and $1$.
We see that the image of $\Phi$ is  the interval $[-\frac{13}{6}, \frac{9}{2}, ]$, but the image of the signed measure
is still  $[-\frac{1}{2}, \frac{9}{2}]$: so for this choice of connection the image of $\Phi$ is larger
than the support of the Duistermaat-Heckman measure.

Above the integral points in $[-\frac{13}{6}, -\frac{1}{2}]$, the reduced
space is not connected, it consists of two points giving opposite contributions to the index. So our theorem holds.

%

\begin{figure}[!h]
\begin{center}
  \includegraphics[width=2 in]{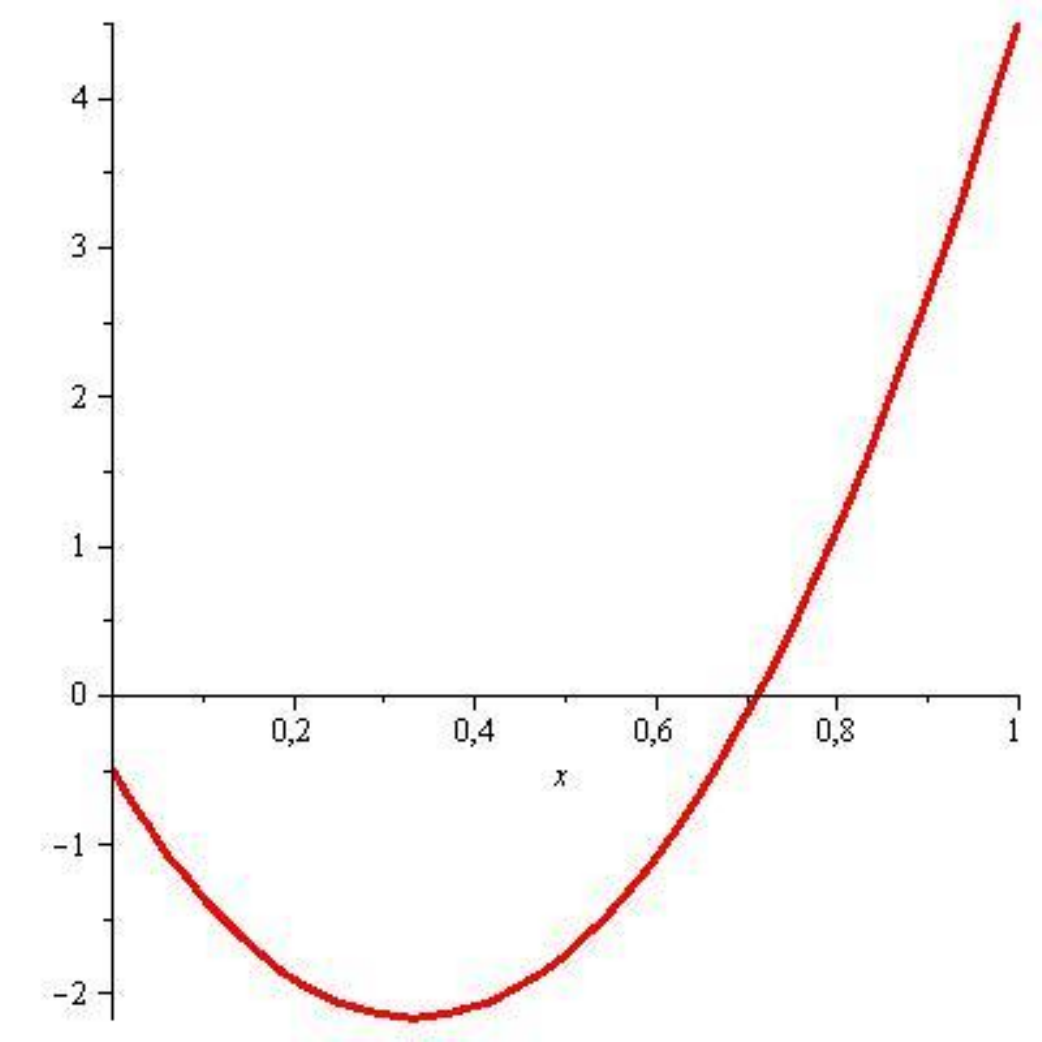}
\caption{ The graph of $\Phi$}
 \label{graphPhi}
\end{center}
\end{figure}

\subsection{The image of the moment map might be non convex}\label{sec:hirzebruch-surface}

We consider $M$ to be the Hirzebruch surface.
Represent $M$ as the quotient of
$\Ucal=\C^2-\{(0,0)\}\times \C^2-\{(0,0)\}$ by the free action of
$\C^*\times \C^*$ acting by
$$(u,v)\cdot(z_1,z_2,z_3,z_4)= (u z_1, u z_2,u v z_3,vz_4)$$
and we denote by $[z_1,z_2,z_3,z_4]\in M$ the equivalence class of
$(z_1,z_2,z_3,z_4)$.
The map  $\pi:[z_1,z_2,z_3,z_4]\to [z_1,z_2]$ is a fibration of $M$ on $P_1(\C)$ with fiber $P_1(\C)$.

Consider the  line bundle $L(n_1,n_2)$ obtained as quotient of the trivial line bundle $\Ucal\times \C$ on
$\Ucal$ by the action $$(u,v)\cdot(z_1,z_2,z_3,z_4,z)= (u z_1, u z_2,uv z_3,vz_4, u^{n_1} v^{n_2}z)$$
for $(u,v)\in \C^*\times \C^*$.
The line bundle  $L(n_1,n_2)$ is ample if and only if $n_1>n_2> 0$.

We have a canonical action of the group $K:=U(2)$ on $M$ : $g\cdot [Z_1,Z_2]=[gZ_1,Z_2]$ for $Z_1,Z_2\in \C^2-\{(0,0)\}$ and the line bundle $L(n_1,n_2)$ with action
 $g\cdot [Z_1,Z_2,z]=[gZ_1,Z_2,z]$ is $K$-equivariant.

We are interested in the (virtual) $K$-module
$$H^{0}(M,\Ocal(n_1,n_2))-H^{1}(M,\Ocal(n_1,n_2))+H^{2}(M,\Ocal(n_1,n_2))$$ where $\Ocal(n_1,n_2)$ be the sheaf of holomorphic sections of $L(n_1,n_2)$.

In this case, it is in fact possible to compute directly individual cohomology groups  $H^{i}(M,\Ocal(n_1,n_2))$.
However, we will describe here  only results on the alternate sum and relate them to the moment map.

Let $T=U(1)\times U(1)$ be the maximal torus of $K$. The set $Y:=\{[z_1,z_2,z_3,z_4]\in M\,\vert\, z_1=0\}$ is a
$T$-invariant complex submanifold of $M$ (with trivial action of $(t_1,1)$). The map
$$Y\to \Pbb^1(\C),\quad [0,z_2,z_3,z_4]\mapsto [(z_2)^{-1}z_3,z_4]$$ is a $T$-equivariant isomorphism   and
the map $(g,y)\in K\times Y\mapsto g\cdot y\in M$ factorizes through an isomorphism
$K\times_{T} Y\simeq M$.
Thus $M$ is an induced manifold.

For any $(a,b)\in \Z^2$, we denote $\C_{a,b}$ the $1$-dimensional representation of $T$ associated to the character
$(t_1,t_2)\mapsto t_1^a t_2^b$. We denote by $e_1^*,e_2^*$ the canonical bases of $\tgot^*\simeq\R^2$.
The Weyl chamber is $\tgot^*_{\geq 0}=\{x e_1^*+y e_2^*, x\geq y\}$.
The elements $e_1^*, e_2^*$ are conjugated by the Weyl group.

The line bundle $L(n_1,n_2)$, when restricted to $Y\simeq \Pbb^1(\C)$, is isomorphic to $\Lcal^{\otimes n_2}\otimes [\C_{0,-n_1}]$.

We consider $L_\kappa=L(3,2)$ the line bundle obtained from the reduction of the trivial line bundle
$\bigwedge^4\C^4$  with natural action of $\C^*\times \C^*$. We denote $\Scal_M:=\bigwedge_\C\T M$
(resp. $\Scal_Y:=\bigwedge_\C \T Y$) the $\spinc$-bundle associated to the complex structure on $M$ (resp. $Y$).

We denote by $\varphi :Y\to [0,1]$ the map defined by $\varphi(y)=\frac{|a_1|^2}{|a_1|^2+|a_2|^2}$ if $y\simeq [a_1,a_2]$.

\begin{prop}\label{prop:connection}
$\bullet$ Let $\Scal(n_1,n_2)$ be the spin bundle $\Scal_M \otimes L(n_1,n_2)$ on $M$. Its determinant line bundle is
$$
\Lbb_{n_1,n_2}=[\C_{\mathrm{det}}]\otimes L_\kappa \otimes L(2n_1,2n_2)
$$
where $[\C_{\mathrm{det}}]\to M$ is the trivial $U(2)$-equivariant line bundle associated to the character $\mathrm{det} : U(2)\to \C^*$.

$\bullet$ There exists a connection on $\Lbb_{n_1,n_2}$ such that the corresponding moment map
$\Phi_{n_1,n_2}: K\times_T Y\to \kgot^*$ is defined by
$$
\Phi_{n_1,n_2}([k,y])=\Big(-(n_1+\frac{3}{2})+ (n_2+1)\varphi(y)\Big) k\cdot e_2^*  +\frac{1}{2}(e_1^*+e_2^*).
$$

\end{prop}

\begin{proof}

For the second point, we construct a $U(2)$-invariant connection on
$\Lbb_{n_1,n_2}$ by choosing the $T$-invariant connection on
$(\Lbb_{n_1,n_2})|_Y$ having moment map
$ \left(-(n_1+\frac{3}{2})+ (n_2+1)\varphi(y)\right)  e_2^*  +\frac{1}{2}(e_1^*+e_2^*)$ under the $T$-action (see Equation (\ref{eq:PhiP1})).
\end{proof}

\bigskip

From Proposition \ref{prop:connection}, it is not difficult to describe the ``Kirwan set"
$\Delta(n_1,n_2)=\mathrm{Image}(\Phi_{n_1,n_2})\cap\tgot^*_{\geq 0}$ for all cases of $n_1,n_2$.
It depends of the signs of
$n_1+\frac{3}{2},n_2+1,n_1-n_2+\frac{1}{2}$, that is, as we are working with integers,  the signs of $n_1+1$, $n_2+1$ and $n_1-n_2$.
We concentrate in the case where $n_1+1\geq 0, n_2+1\geq 0$ (other cases are similarly treated).
Then, we have two cases:

$\bullet$ If $n_1\geq n_2$, then the Kirwan set $\Delta(n_1,n_2)$  is the interval $$[(n_1-n_2)+\frac{1}{2}, n_1+\frac{3}{2}](-e_2^*)+\frac{1}{2} (e_1^*+e_2^*).$$

$\bullet$ If $n_2> n_1$, then the Kirwan set $\Delta(n_1,n_2)$  is the union of the intervals
  $$[0,n_2-n_1-\frac{1}{2}]e_1^* +\frac{1}{2} (e_1^*+e_2^*)$$ and
  $$[0,n_1+\frac{3}{2}](-e_2^*)+\frac{1}{2}(e_1^*+e_2^*).$$

If $n_1\geq n_2\geq 0$  the curvature of the corresponding connection on  $\Lbb_{n_1,n_2}=L(2 n_1+3,2n_2+2)$ (which is an ample line bundle) is non degenerate, thus the image   is a convex subset of $\tgot^*_{\geq 0}$ (in agreement with Kirwan convexity theorem)
while for $n_2>n_1$ the image set is not convex.

\medskip

The character $\Qcal_K(n_1,n_2):=\Qcal_K(M, \Scal(n_1,n_2))$ is equal to the (virtual) $K$-module
$H^{0}(M,\Ocal(n_1,n_2))-H^{1}(M,\Ocal(n_1,n_2))+H^{2}(M,\Ocal(n_1,n_2))$ where $\Ocal(n_1,n_2)$ be the sheaf of holomorphic sections of $L(n_1,n_2)$.

Let $\Lambda_{\geq 0}=\{(\lambda_1,\lambda_2); \lambda_1\geq \lambda_2\} $ be the  set of dominant weights for $U(2)$.
We index the representations of $U(2)$ by $\rho+\Lambda_{\geq 0}$. Here $\rho=(\frac{1}{2},\frac{-1}{2})$ and $\lambda_1, \lambda_2$ are integers.
We then have
$$\pi_{(\frac{1}{2},-k-\frac{1}{2})}=S^k$$ the space of complex polynomials on $\C^2$ homogeneous of degree $k$.

If $n_2\geq 0$, we know that $\Qcal_{T}(Y, \Scal_Y \otimes \Lcal^{\otimes n_2})=\sum_{k=0}^{n_2} t^k_{2}$.
From the induction formula (\ref{eq:induction-index})
 (or direct computation via Cech cohomology !!)  we obtain

$\bullet$ If $n_1\geq n_2$,
then
$$\Qcal_K(n_1,n_2)=\sum_{k=n_1-n_2}^{n_1} \pi_{(\frac{1}{2},-k-\frac{1}{2})}.$$

$\bullet$ If $n_2>n_1$, then  $$\Qcal_K(n_1,n_2)=\sum^{n_1}_{k=0} \pi_{(\frac{1}{2},-k-\frac{1}{2})}- \sum^{n_2-n_1-2}_{k=0} \pi_{(k+\frac{3}{2},\frac{1}{2})}.$$

\medskip
Let us checked how our theorem works in these cases.
First, we notice that we are in a multiplicity free case : all the non-empty
reduced spaces are points.

$\bullet$ Consider the case where $n_1\geq n_2$.
We see that the parameter $(\frac{1}{2},-k-\frac{1}{2})$
belongs to the relative interior of  the interval $\Delta(n_1,n_2)$.
In particular for $b=(0,0)$, the unique point in
 the relative interior of  the interval $\Delta(0,0)$ is $\rho$. This is in agreement to the fact that the representation $\Qcal_K(0,0)$ is the trivial representation of $K$.

$\bullet$ Consider the case where $n_2> n_1$.
We see that the parameter $(\frac{1}{2},-k-\frac{1}{2})$ belongs to the relative interior of $[-n_1-\frac{3}{2},0]e_2^*+\frac{1}{2}(e_1^*+e_2^*)$ if and only if $k\leq n_1$. Similarly,
the parameter $(k+\frac{3}{2},\frac{1}{2})$ belongs to the relative interior of
$[0,n_2-n_1-\frac{1}{2}]e_1^*+\frac{1}{2}(e_1^*+e_2^*)$ if and only if $k\leq n_2-n_1-2$.

In Figures \ref{image-1}, \ref{image-2}, \ref{image-3},
we draw the   Kirwan subsets of $\tgot^*_{\geq 0}$ corresponding to
the values $(n_1,n_2)=(8,5)$ or $(3,6)$.
The circle points on the  red line represents the admissible points occurring with multiplicity $1$  in
$\Qcal_K(n_1,n_2)$.
The diamond points on the blue line represents the admissible points occurring  with multiplicity $-1$ in  $\Qcal_K(n_1,n_2)$.

\begin{figure}[!h]
\begin{center}
  \includegraphics[width=2 in]{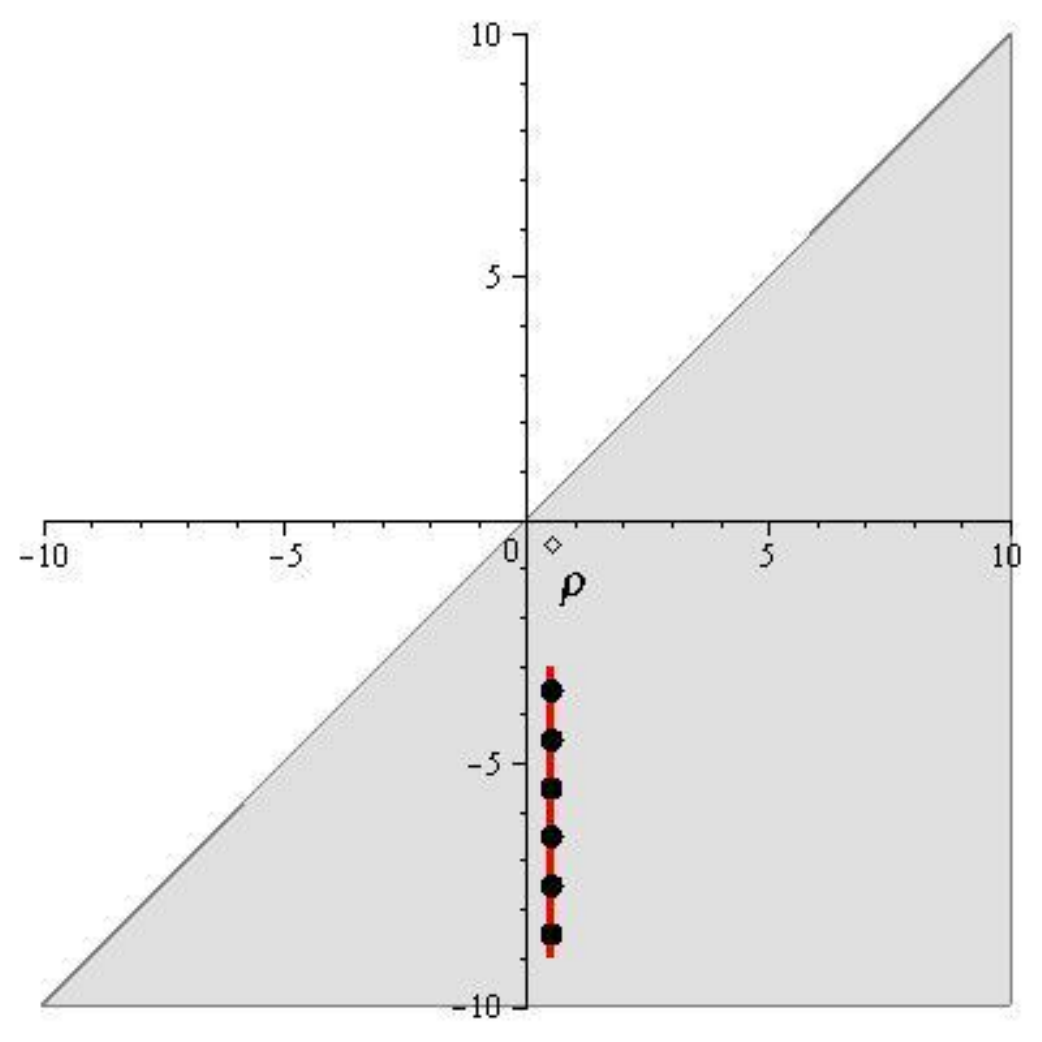}
\caption{$K$-Multiplicities for  $\Qcal_K(8,5)$}
\label{image-1}
 \end{center}
\end{figure}

\begin{figure}[!h]
\begin{center}
  \includegraphics[width=2 in]{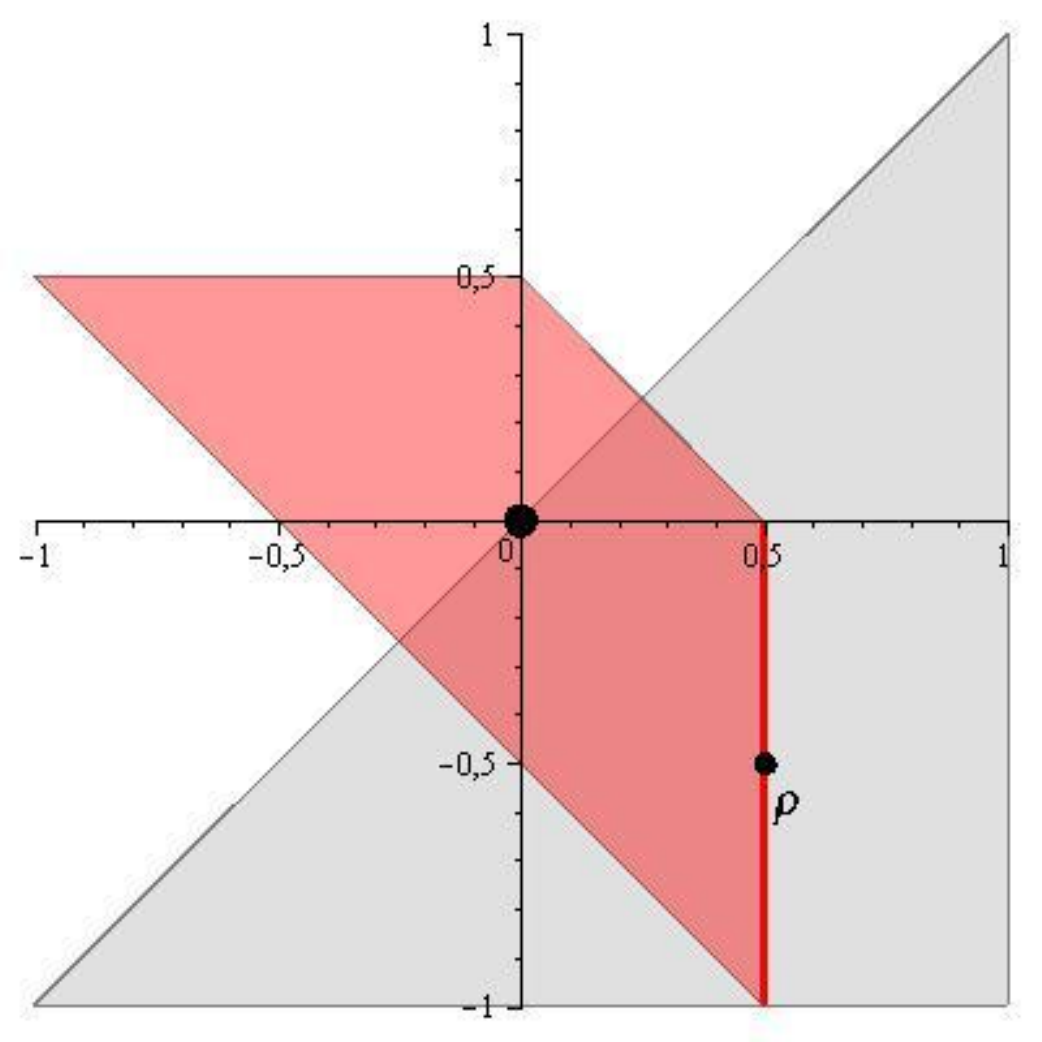}
\end{center}
\caption{$T$-multiplicities for  $\Qcal_T(0,0)$}
\label{image-2}
\end{figure}

\begin{figure}[!h]
\begin{center}
  \includegraphics[width=2 in]{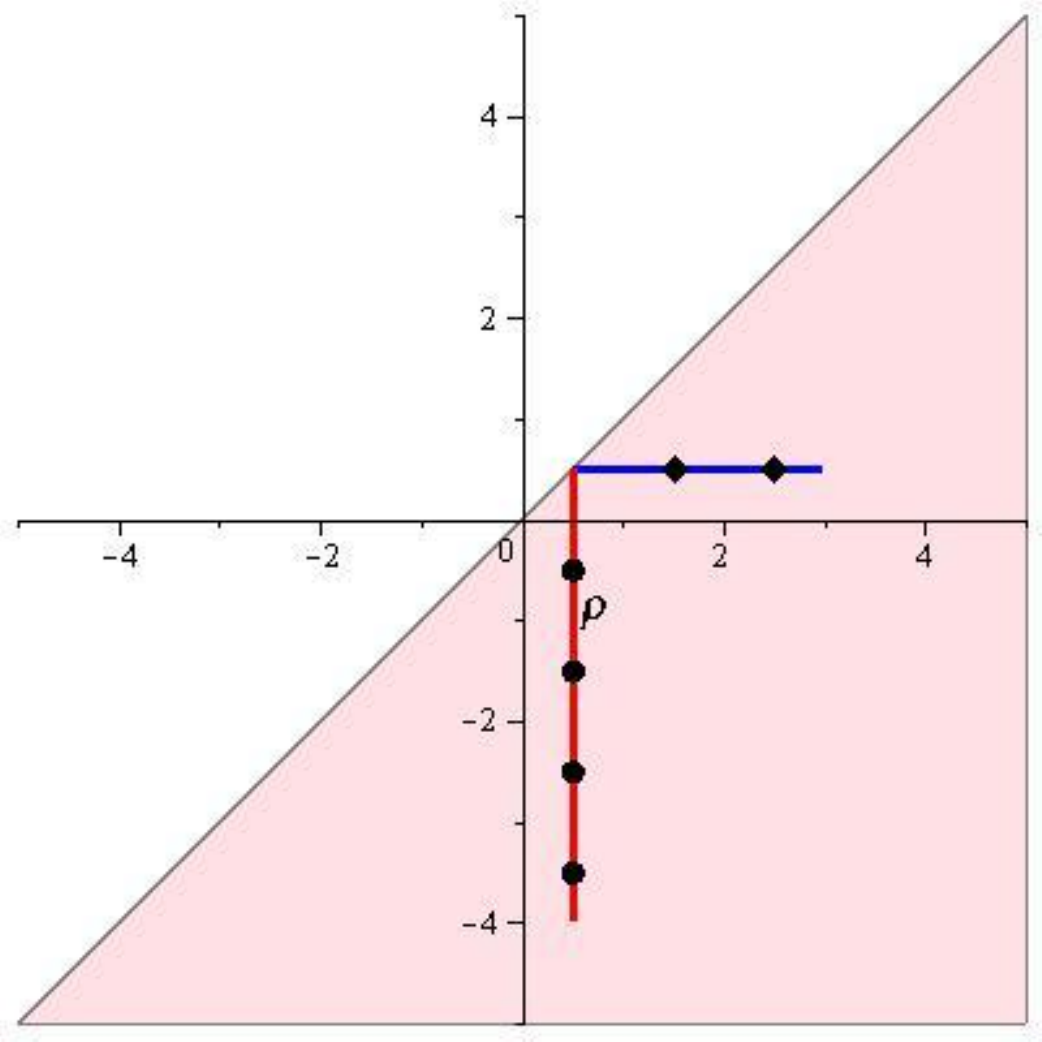}
\caption{$K$-Multiplicities for  $\Qcal_K(3,6)$}
\label{image-3}
\end{center}
\end{figure}

\medskip

Consider  now $M$ as a $T$-manifold. Let $\Phi_{n_1,n_2}^T:M\to \tgot^*$ be the moment map relative to the action of $T$ which is the composite of $\Phi_{n_1,n_2}:M\to\kgot^*$ with the projection $\kgot^*\to\tgot^*$. Thus, the image of $\Phi_{n_1,n_2}^T$ is the convex hull of $\Delta(n_1,n_2)$ 
and its symmetric image with respect to the diagonal.

Consider first the case where $n_1=n_2=0$.
Thus our determinant bundle $\Lbb_{0,0}=[\C_{\mathrm{det}}]\otimes L(3,2)$ is ample.  The image of the moment map
$\Phi_{0,0}^T: M\to \tgot^*$ is equal to the convex polytope
$\Delta$ with vertices $(0,\frac{1}{2}),(\frac{1}{2},0),(\frac{1}{2},-1),(-1,\frac{1}{2})$,
the images of the $4$ fixed points $[1,0,1,0]$, $[1,0,0,1]$,$[0,1,1,0]$,$[0,1,0,1]$.
 The only integral point in the interior of the polytope is $(0,0)$ and the reduced space
$(\Phi_{0,0}^T)^{-1}((0,0))/T$ is a point. The representation
$\QS_T(M, \Scal(0,0))$
is indeed the trivial representation of $T$.

We now concentrate on the case $(n_1,n_2)=(3,6)$. The line bundle $\Lbb:=[\C_{\mathrm{det}}]\otimes \Lbb_{3,6}$ is not an ample bundle, so that its curvature $\Omega_\Lbb$ is degenerate, and the Liouville form $\beta_{\Lbb}=\Omega_\Lbb\wedge \Omega_\Lbb $ is a signed measure on $M$.  Let us draw the Duistermaat measure $(\Phi_{\Lbb})_*\beta_{\Lbb}$, a signed measure on $\tgot^*$. In red the measure is with value $1$, in blue the measure is with value $-1$.

We also verify that our theorem is true.
Indeed the representation
$\Qcal_T(M,\Scal(3,6))=\Qcal_K(M,\Scal(3,6))\vert_T$ is
$$
1+t_1^{-1}+t_2^{-1}+t_1^{-2}+t_1^{-1}t_2^{-1}+t_2^{-2}+t_1^{-3}+t_1^{-2}t_2^{-1}+t_1^{-1}t_1^{-2}+t_2^{-3}
-t_1t_2- t_1t_2^2-t_1^2t_2.
$$
The $\lambda\in \Z^2$ such that $t^\lambda$ occurs in
$\Qcal_T(M,\Scal(3,6))$ are the integral points in the interior of the image of $\Phi_\Lbb(M)$ :
they have multiplicity $\pm1$, and the reduced space are points.

In this case, we verify thus that the image of the moment map is exactly the support of the Duitermaat-Heckman measure, however, we do not know if (even in the toric case, and non ample bundle) we can always find a connection with this property.

\begin{figure}
\begin{center}
  \includegraphics[width=2 in]{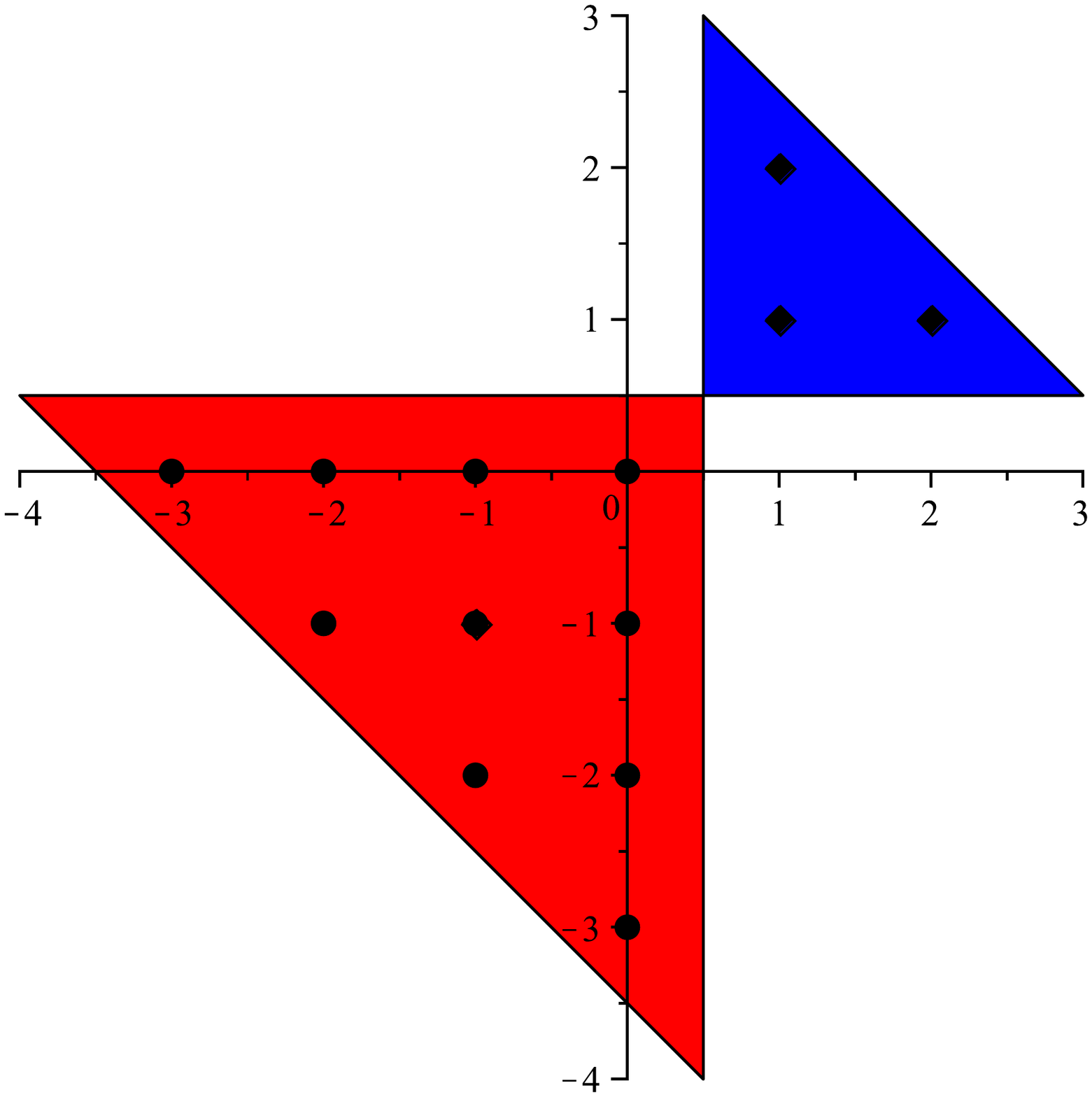}
\caption{$T$-multiplicities for  non ample bundle on Hirzebruch surface}
 \label{nonample}
\end{center}
\end{figure}

\subsection{The multiplicity of the trivial representation comes from two reduced spaces}\label{sec:induced-example}

Let $\C^4$ with its canonical basis $\{e_1,\ldots,e_4\}$. Let $K\simeq SU(3)$ be the subgroup of $SU(4)$ that fixes $e_4$.

Let $T=S(U(1)\times U(1)\times U(1))$ be the maximal torus of $K$ with Lie algebra $\tgot=\{(x_1,x_2,x_3), \sum_i x_i=0\}$, and
Weyl chamber $\tgot^*_{\geq 0}:=\{ \xi_1\geq \xi_2\geq \xi_3, \sum_i \xi_i=0\}$. We choose the fundamental roots
$\omega_1,\omega_2$ so that $K_{\omega_1}= S(U(2)\times U(1))$ and $K_{\omega_2}= S(U(1)\times U(2))$.
Recall that $\omega_1,\omega_2$ generates the weight lattice $\Lambda\subset\tgot^*$ so that
$\Lambda_{\geq 0}= \N \omega_1 +\N \omega_2$. Note also that $\rho=\omega_1+\omega_2$. For any $\lambda \in \Lambda_{\geq 0}+\rho$, we denote $\pi_{\lambda}$ the irreducible representation of $K$ with highest weight $\lambda-\rho$.

Let  $X=\{0\subset L_1\subset L_2\subset \C^4,\ \dim L_i= i\}$ be the homogeneous partial flag manifold under the action of $SU(4)$.
We have two lines bundles over $X$: $\Lcal_1(x)=L_1$ and $\Lcal_2(x)=L_2/L_1$ for $x=(L_1,L_2)$.

Our object of study is the complex submanifold
$$
M=\{(L_1,L_2)\in X\ \vert\  \C e_4\subset L_2\}.
$$
The group $K$ acts on $M$, and the generic stabilizer of the action is $[K_{\omega_1},K_{\omega_1}]\simeq SU(2)$.
We consider the family of lines bundles
$$
\Lcal(a,b)=\Lcal_1^{\otimes a}\vert_M\otimes \Lcal_2^{\otimes - b}\vert_M,\quad (a,b)\in\N^2.
$$
Let $\Scal_M:=\bigwedge_\C \T M$ be the $\spinc$-bundle associated to the complex structure on $M$.
We compute the characters
$$
\Qcal_K(a,b):= \Qcal_{K}( M,\Scal_M \otimes \Lcal(a,b))\in R(K).
$$
Again
$$\Qcal_K(a,b)=\sum_{i=0}^{\dim M} (-1)^i H^i(M, \Ocal(\Lcal(a,b))).$$

We notice that $K_{\omega_1}$ corresponds to the subgroup of $K$ that fixes the line $\C e_3$. The set
$Y:=\{(L_1,L_2)\in X\,\vert\, L_2=\C e_3\oplus \C e_4\}$ is a $K_{\omega_1}$-invariant complex submanifold of $M$
such that the map $(k,y)\in K\times Y\mapsto ky\in M$ factorizes through an isomorphism
$K\times_{K_{\omega_1}} Y\simeq M$. Notice that $[K_{\omega_1},K_{\omega_1}]$ acts trivially on $Y$.

If we take $a\geq 4$ and $b\geq 1$ we get  that
\begin{equation}\label{eq:Q-a-b}
\Qcal_K(a,b)= -\sum_{k=0}^{b-1} \pi_{k\omega_1+\rho} - \sum_{j=0}^{a-4} \pi_{j\omega_2+\rho}.
\end{equation}
In particular the multiplicity of $\pi_\rho$ (the trivial representation) in $\Qcal_K(a,b)$ is equal to $-2$.

\medskip

We now  verify the formula (\ref{eq:QR=0}) in our case. The $\spinc$-bundle $\Scal_M$ is equal to
$\Scal_{K\omega_1}\otimes K\times_{K_{\omega_1}}\Scal_Y$. The corresponding determinant line bundle $\det(\Scal_M)$ satisfies
\begin{eqnarray*}
\det(\Scal_M)&=& K\times_{K_{\omega_1}}\C_{3\omega_1}\otimes K\times_{K_{\omega_1}}\det(\Scal_Y)\\
&=& K\times_{K_{\omega_1}}\C_{2\omega_1}\otimes \Lcal_1^{\otimes -2}.
\end{eqnarray*}
Hence for the $\spinc$-bundle $\Scal_M\otimes \Lcal(a,b)$ we have
\begin{eqnarray*}
\det(\Scal_M\otimes \Lcal(a,b))&=& \det(\Scal_M)\otimes \Lcal(a,b)^{\otimes 2}\\
&=& K\times_{K_{\omega_1}}\C_{(2b+2)\omega_1}\otimes \Lcal_1^{\otimes 2(a+b-1)}.
\end{eqnarray*}
The line bundle $\det(\Scal_M\otimes \Lcal(a,b))$ is equipped with a natural holomorphic and hermitian connection
$\nabla$. To compute the corresponding moment map $\Phi_{a,b}:M\to \kgot^*$, we notice that
 $\Lcal_1=K\times_{K_{\omega_1}}\Lcal^{-1}$ where $\Lcal\to \Pbb^1$ is the prequantum line bundle over $\Pbb_1$
 (equipped with the Fubini-Study symplectic form). If we denote $\varphi :Y\simeq\Pbb^1\to [0,1]$ the function defined by $
 \varphi([z_1,z_2])=\frac{|z_1|^2}{|z_1|^2+|z_2|^2}$, we see that
 $$
 \Phi_{a,b}([k,y])=k\left[\left((b+1)-(a+b-1)\varphi(y)\right) \omega_1\right].
 $$
for $[k,y]\in M$. In this case,  the Kirwan set $\Phi_{a,b}(M)\cap \tgot^*_{\geq 0}$ is the non convex set
$[0,b+1]\omega_1\cup [0,a-2]\omega_2$.

\medskip

We know (see Exemple \ref{exa:ExampleU3}) that the set $\Acal((\kgot_{\omega_1}))$ is equal to the collection of orbits
$K(\frac{1+2n}{2}\omega_i), n\in\N, i=1,2$, and we have $\Qcal_K(K(\frac{1}{2}\omega_i))=0$ and
$\Qcal_K(K(\frac{3+2k}{2}\omega_i))= \pi_{k\omega_i+\rho}$ when $k\geq 0$.

If we apply (\ref{eq:QR=0}), we see that $\pi_{k\omega_1+\rho}$ occurs in $\Qcal_K(a,b)$ only if $\frac{3+2k}{2}<b+1$ :
so $k\in\{0,\ldots, b-1\}$. Similarly $\pi_{j\omega_2+\rho}$ occurs in $\Qcal_K(a,b)$ only if $\frac{3+2j}{2}<a-2$ : so
$j\in\{0,\ldots, a-4\}$. For all this cases the corresponding reduced spaces are points and one could check that the corresponding quantizations are all equal to $-1$ (see (\ref{eq:Q-point})).

In this case, two orbits $\Pcal_i= K(\frac{3}{2}\omega_i), i=1,2$ are the ancestors of the trivial representation in $\Acal((\kgot_{\omega_1}))$, and the multiplicity of the trivial representation in $\Qcal_{K}( M,\Scal_M \otimes \Lcal(a,b))$
is equal to
$$
\QS(M_{\Pcal_1})+\QS(M_{\Pcal_2})=-2.
$$




\end{document}